\documentclass{amsart}
\usepackage[dvips]{epsfig}
\usepackage{graphicx}
\usepackage{latexsym}
\usepackage{amsmath}
\usepackage{amsthm}
\usepackage{amssymb}
\RequirePackage[colorlinks=true,citecolor=blue,urlcolor=blue]{hyperref}

\usepackage{eucal}
\usepackage{float}
\usepackage{xcolor}
\usepackage{multirow} 
\usepackage{array}

 \usepackage[applemac]{inputenc}

\newtheorem{thm}{Theorem}
\newtheorem{assumption}[thm]{Assumption}
\newtheorem{lem}[thm]{Lemma}

\newtheorem{cor}[thm]{Corollary}
\newtheorem{defi}[thm]{Definition}

\newtheorem{prop}[thm]{Proposition}
\newtheorem{rk}[thm]{Remark}

\newcommand{\field}[1]{\mathbb{#1}}

\newcommand{\E}{\field{E}}

\newcommand{\GG}{\field{G}}

\newcommand{\PP}{\field{P}}

\newcommand{\R}{\field{R}}

\newcommand{\vip}{\vskip.2cm}

\DeclareMathOperator*{\argmax}{argmax}
\DeclareMathOperator*{\sgn}{sgn}
\newcommand{\Ttransition}{{\mathcal P}}

\newcommand{\ktilde}{\widetilde{\kappa}}
\newcommand{\supp}{\text{supp}}

\begin{document}

\title[Estimation in a randomly structured branching population]{Statistical estimation in a randomly structured branching population}
\author{Marc Hoffmann and Aline Marguet}

%
%
\address{Marc Hoffmann, Universit\'e Paris-Dauphine PSL, CEREMADE, 75016 Paris, France. 
}
\email{hoffmann@ceremade.dauphine.fr}
\address{Aline Marguet, Univ. Grenoble Alpes, INRIA, 38000 Grenoble, France.}
\email{aline.marguet@inria.fr }
\begin{abstract}
We consider a binary branching process structured by a stochastic trait that evolves according to a diffusion process that triggers the branching events, in the spirit of Kimmel's model of cell division with parasite infection. Based on the observation of the trait at birth of the first $n$ generations of the process, we construct nonparametric estimator of the transition of the associated bifurcating chain and study the parametric estimation of the branching rate. In the limit $n \rightarrow \infty$, we obtain asymptotic efficiency in the parametric case and minimax optimality in the nonparametric case. 
\end{abstract}
\maketitle
{\small \noindent \textbf{Mathematics Subject Classification (2010)}:  62G05, 62M05, 60J80, 60J20, 92D25.\\
\textbf{Keywords}: Branching processes, bifurcating Markov chains, statistical estimation, geometric ergodicity, scalar diffusions.}
\section{Introduction}
\subsection{Motivation}

The study of structured populations, with a strong input from evolutionary or cell division modelling in mathematical biology (see for instance the textbooks \cite{Meleard, Perthame} and the references therein) has driven the statistics of branching Markov processes over the last few years. Several models have been considered, with data processed either in discrete or continuous time. In this context,  one typically addresses the inference of critical parameters like branching rates, modelled as functions of biological traits like age, size and so on. In many cases, this approach is linked to certain piecewise deterministic Markov models or bifurcating Markov chains (BMC) in discrete time. These models are well understood from a probabilist point of view (in discrete time Guyon \cite{guyon2007limit}, Bitseki-Penda {\it et al.} \cite{Guillin2, Guillin1}, in continuous time Bansaye and M\'el\'eard \cite{BaMe}, Bansaye {\it et al.} \cite{BDMT} or more recently Marguet \cite{marguet2016} for a general approach). For the statistical estimation, we refer to  \cite{BHO2016, BDHR, DHKR1, HO, BasawaZhou}, and the references therein, see also Bitseki-Penda and Olivier \cite{BO}, de Saporta {\it et al.} \cite{GP3, GP1}, Aza\"\i s {\it et al.} \cite{GP2} or recently Bitseki-Penda and Roche \cite{BRoche}. In these models, the traits of a population between branching events like cell division evolve through time according to a dynamical system. The next logical step is to replace this deterministic evolution by a random flow, that allows one to account for traits that may have their own random evolution according to some exogeneous input. A paradigmatic example is Kimmel's model (see Kimmel \cite{Kimmel} and Bansaye \cite{Bansaye}) where the trait is given by a density of parasites within a cell that evolve according to a diffusion process. The statistical analysis of such models is the topic of the present paper.\\

 We consider a population model with binary division triggered by a trait $x \in \mathcal{X}$ where  $\mathcal X \subseteq \mathbb{R}$ is an open (possibly unbounded) interval. The trait $\phi_x(t)$ of each individual evolves according to
\begin{equation} \label{stoch flow}
d\phi_x(t)=r(\phi_x(t))dt+\sigma(\phi_x(t))dW_t,\;\;\phi_{x}(0)=x,
\end{equation}
where $r,\sigma:\mathcal X \rightarrow \mathcal X$ are regular functions and $(W_t)_{t \geq 0}$ is a standard Brownian motion. Each individual with trait $x$ dies according to a killing or rather division rate $x \mapsto B(x)$, {\it i.e.} an individual with trait $\phi_x(t)$ at time $t$ dies with probability $B(\phi_x(t))dt$ during the interval $[t,t+dt]$. At division, a particle with trait $y$ is replaced by two new individuals with trait at birth given respectively by $\theta y$ and $(1-\theta)y$ where $\theta$ is  drawn according to $\kappa(y)dy$ for some probability density function $\kappa(y)$ on $[0,1]$.The model is described by the traits of the population, formally given as a Markov process
\begin{equation} \label{def big Markov}
X(t) =(X_1(t), X_2(t),\ldots),\;\;t\geq 0
\end{equation}
with values in $\bigcup_{k \geq 1}\mathcal X^k$, where the $X_i(t)$ denote the (ordered) traits of the living particles at time $t$. Its distribution is entirely determined by an initial condition at $t=0$ and by the parameters $(r,\sigma,B,\kappa)$.\\  

\subsection{Statistical setting by reduction to a bifurcating Markov chain model}

We assume we have data at branching events ({\it i.e.} at cell division) and we wish to make inference on the parameters of the model. Using the Ulam-Harris-Neveu notation, for $m\geq 0$, let  $\GG_m = \{ 0,1 \}^m $ (with $\GG_0=\{\emptyset\}$) and introduce the infinite genealogical tree 
\begin{align*}
\mathbb T=\bigcup_{m\in\mathbb{N}}\mathbb G_m.
\end{align*}
For $u \in \mathbb G_m$, set $|u|=m$ and define the concatenation $u0 = (u,0)\in \mathbb G_{m+1}$ and $u1=(u,1)\in \mathbb G_{m+1}$. For $n \geq 0$, let $\mathbb T_n = \bigcup_{m=0}^n \mathbb G_m$ denote the genealogical tree up to the $n$-th generation and $|\mathbb T_n|$ denote its cardinality. We denote by $X_u$ the trait at birth of an individual $u\in\mathbb{T}$. From the branching events, we assume that  
we observe
$$\mathbb X^n = (X_u)_{u \in \mathbb  U_n},$$
where $\mathbb U_n \subseteq \mathbb T_n$ is what we call a regular {\it incomplete tree}, that is a connected subtree of $\mathbb T_n$ that contains at least one individual at the $n$-th generation (see the formal definition \ref{def incomplete tree} in Section \ref{sec conv empirical measures} below) and with cardinality of order $2^{\varrho n}$ for some $0 \leq \varrho \leq 1$. This observation scheme is motivated by typical datasets available in biological experiments, see {\it e.g.} Robert {\it et al.} \cite{DHKR2} and the refrences therein: when moving from generation $m-1$ to $m$ (for $m=1,\ldots, n$) we possibly lose some information, quantified by $\varrho$, due to experimental anomalies or simply because of the design of the experimental process (for instance, in the extreme case $\varrho = 0$, it may well happen that one observes only a single lineage of the bifurcating process due to experimental constraints, as some datasets studied in \cite{DHKR2}). We thus have approximately $2^{\varrho n}$ random variables with value in $\mathcal X$ with a certain Markov structure. Asymptotics are taken as $n$ grows to infinity. An example of trajectory is represented on Figure \ref{fig:trajectoire} with the associated genealogy.
\begin{figure}
\begin{tabular}{m{0.7\textwidth}m{0.3\textwidth}}
\includegraphics[scale=0.27]{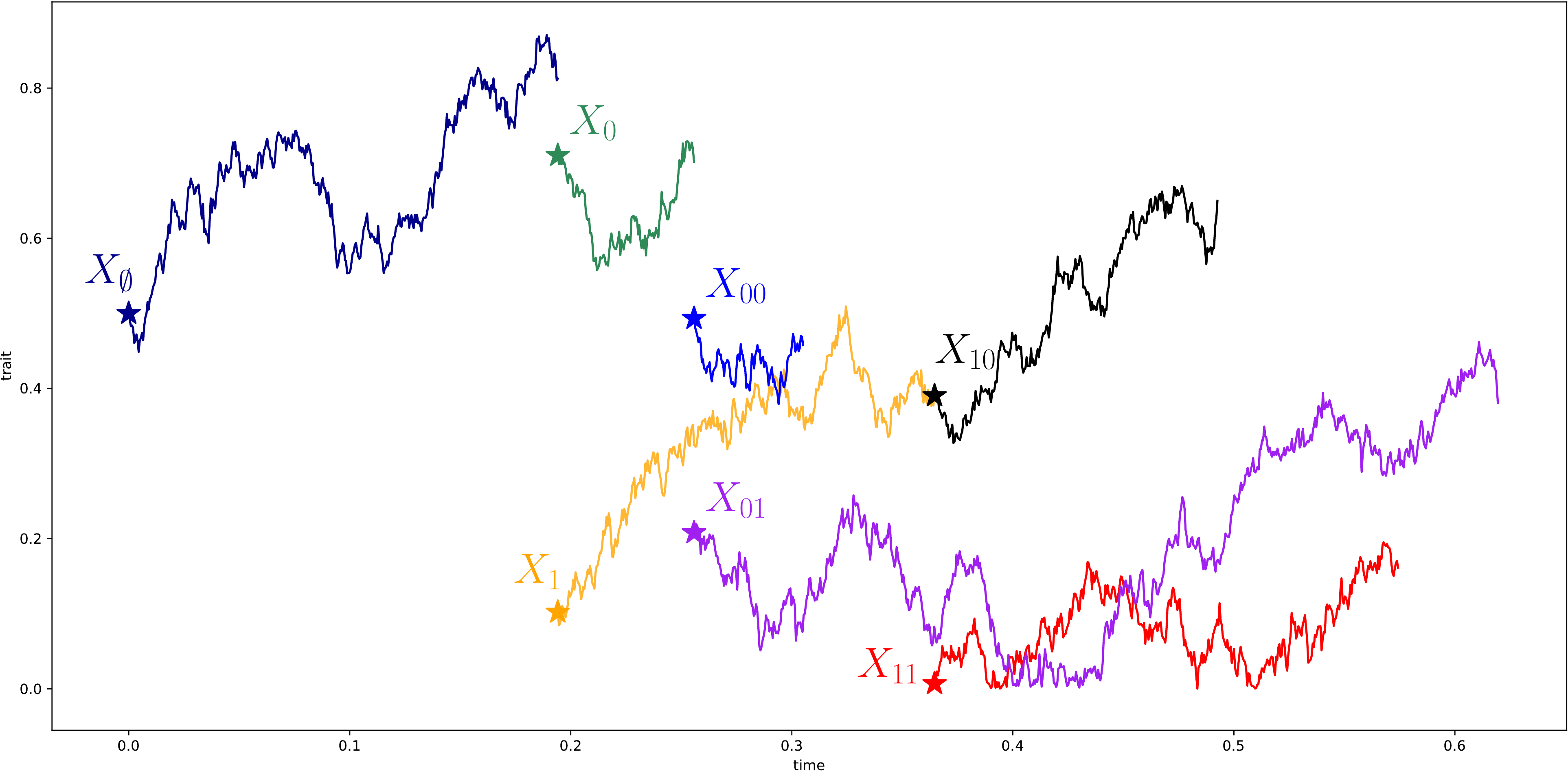}&
\includegraphics[scale=0.45]{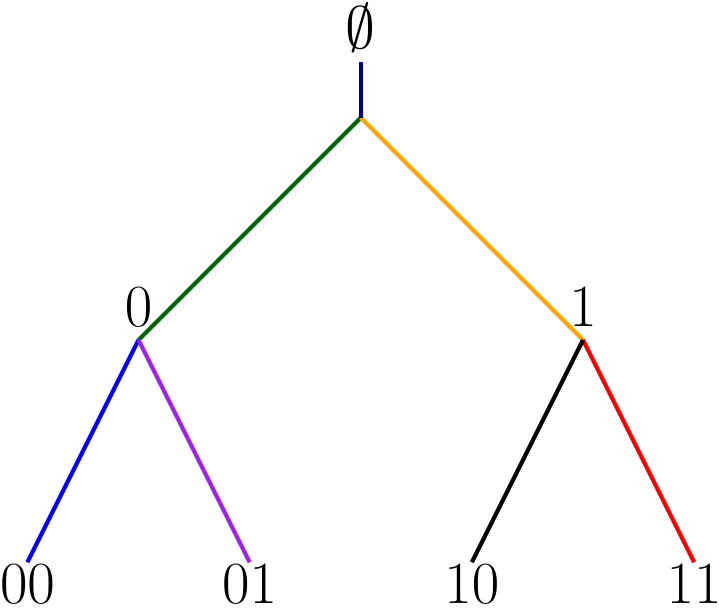}
\end{tabular}
\caption{Example of a trajectory and its associated genealogy.}
\label{fig:trajectoire}
\end{figure}\\ 

There are several objects of interest that we may try to infer from the data $\mathbb X^n$. First, one may notice that the Markov structure of $X$ in \eqref{def big Markov} turns $(X_u, u \in \mathbb T)$ into a {\it bifurcating Markov chain} according to the terminology introduced Basawa and Zhou \cite{BasawaZhou}, later highlighted by Guyon \cite{guyon2007limit}.
A bifurcating Markov chain is specified by {\bf 1)} a measurable state space, here $\mathcal X$ (endowed with its Borel sigma-field) with a Markov kernel $\Ttransition$ from $\mathcal X$ to $\mathcal X \times \mathcal X$ and 
{\bf 2)} a filtered probability space $\big(\Omega, \mathcal F, (\mathcal F_m)_{m \geq 0}, \PP\big)$. 
Following Guyon, 
\cite{guyon2007limit}, Definition 2,  
we have the 
\begin{defi} \label{BHO_def BMC}
A bifurcating Markov chain (BMC) is a family $(X_u)_{u \in \mathbb T}$ of random variables with value in $\mathcal X$ such that 
$X_u$ is $\mathcal F_{|u|}$-measurable for every $u \in \mathbb T$ and 
\begin{equation} \label{def:transBMC}
\E\big[\prod_{u \in \mathbb G_m} \psi_u(X_u, X_{u0}, X_{u1})\,\big|\,\mathcal F_m\big]= \prod_{u \in \mathbb G_m}\Ttransition \psi_u(X_u)
\end{equation}
for every $m \geq 0$ and any family of (bounded) measurable functions $(\psi_u)_{u \in \mathbb G_m}$, where
$\Ttransition \psi(x)=\int_{\mathcal X \times \mathcal X}\psi(x,y_1,y_2)\Ttransition(x,dy_1dy_2)$ denotes the action of $\Ttransition$ on $\psi$.
\end{defi}
The distribution of $(X_u)_{u \in \mathbb T}$ is thus entirely determined by $\Ttransition$ and an initial distribution for $X_{\emptyset}$. 
A key role for understanding the asymptotic behavior of the bifurcating Markov chain is the so-called {\it tagged-branch} chain, that consists in picking a lineage at random in the population $(X_u)_{u \in \mathbb T}$: it is a Markov chain with value in $\mathcal X$ defined by $Y_0=X_{\emptyset}$ and for $m \geq 1$:
$$Y_m=X_{\emptyset \epsilon_1\ldots \epsilon_m},$$
where $(\epsilon_m)_{m \geq 1}$ is a sequence of independent Bernoulli random variables with parameter $1/2$, independent of $(X_u)_{u \in \mathbb T}$, with transition 
\begin{equation} \label{def:meanTransition}
\mathcal Q=(\mathcal P_0+\mathcal P_1)/2
\end{equation} 
obtained from the marginal transitions of $\mathcal P$:
$$\mathcal P_0(x,dy)=\int_{y_1 \in \mathcal X}\mathcal P(x,dy\,dy_1)\;\;\text{and}\;\;\mathcal P_1(x,dy)=\int_{y_0 \in \mathcal X}\mathcal P(x,dy_0dy).$$
Guyon proves in \cite{guyon2007limit} that if $(Y_m)_{m \geq 0}$ is ergodic with invariant measure $\nu(dx)$ on $\mathcal X$, then a convergence of the type
\begin{equation} \label{convGuyon}
\frac{1}{|\mathbb T_{n}|} \sum_{u \in \mathbb T_{n}}\psi(X_u,X_{u0},X_{u1}) \rightarrow \int_\mathcal X \mathcal P\psi(x)\nu(dx)
\end{equation}
holds as $n \rightarrow \infty$ for appropriate test functions $g$, almost surely and appended with appropriate central limit theorems (Theorem 19 in \cite{guyon2007limit}). Under appropriate regularity assumptions, an analogous result shall hold when $\mathbb T_n$ is replaced by a regular incomplete tree $\mathbb U_n$.

\subsection{Main results}\label{sec:main results}

In this context, there are several quantities that can be inferred from the data $\mathbb X_n$ as $n$ grows and that are important in order to understand the dynamics of $(X_u)_{u\in \mathbb T}$.  Under suitable assumptions on the stochastic flow \eqref{stoch flow}, the transition $\mathcal Q$ admits an invariant measure $\nu$ and we have fast convergence of the tagged-chain $(Y_m)_{m \geq 1}$ to equilibrium. This enables us to construct in a first part nonparametric estimators of $\nu$ and $\mathcal Q$ with an optimal rate of convergence and reveals the structure of the underlying BMC.\\ 

However, estimators of $\nu$ and $\mathcal Q$ do not give us any insight about the parameters $(r,\sigma,B,\kappa)$ of the model. In a second part, we investigate the inference of the division rate $x \mapsto B(x)$ as a function of the trait $x \in \mathcal X$ when the other parameters $r,\sigma$ and $\kappa(x)$ are known. This  seemingly stringent assumption is necessary given the observation scheme $\mathbb X^n$. If extraneous data were available, estimators of the parameters $r,\sigma$ and $\kappa$ could be obtained in a relatively straightforward manner:
\begin{itemize}
\item[i)] As soon as a discretisation of the values of the flow are available, standard techniques about inference in ergodic diffusions can be applied to recover $x\mapsto r(x)$ and $x \mapsto \sigma(x)^2$, see for instance \cite{H4,kutoyants2013statistical}.
\item[ii)] The fact that an individual $u$ distributed its traits to its offspring in a conservative way enables one to recover the fraction $\theta_u$ distributed among the children. Indeed
the individual $u$ born at $b_u$ with lifespan $d_u-b_u$ has trait $\phi_{X_{u}}(d_u-b_u)$ at its time of death.
It follows that its children have trait at birth given by
$$X_{u0}=\theta_u\phi_{X_{u}}(d_u-b_u),\;\;X_{u1}=(1-\theta_u)\phi_{X_{u}}(d_u-b_u),$$
where the $\theta_u$ are drawn independently from the distribution $\kappa(x)dx$ 
and therefore, the relationship
$\frac{X_{u_0}}{X_{u_1}} = \frac{\theta_u}{1-\theta_u}$
identifies $\theta_u$. In turn, the estimation of $x \mapsto \kappa(x)$ reduces to a standard density estimation problem from data $(\theta_u)_{u \in \mathbb U_n}$, see for instance \cite{hoang2015estimating}.
\end{itemize} 
The identification and estimation of the branching rate $x \mapsto B(x)$ from data $\mathbb X^n$ is more delicate and is the topic of the second part of the paper. Under minimal regularity assumptions developed in Section \ref{modele proba} below, it is not difficult to obtain an explicit representation of the transition $\mathcal Q(x,dy)=\mathcal Q_B(x,dy)=q_B(x,y)dy$ that reads
\begin{equation} \label{expl transit}
q_B(x,y)= \int_0^1\Big(\frac{\kappa(z)+\kappa(1-z)}{2z}\Big)B(y/z)\sigma(y/z)^{-2}\mathbb{E}\Big[\int_0^{\infty}e^{-\int_0^t B(\phi_x(s))ds}dL_t^{y/z}(\phi_x)\Big]dz,
\end{equation}
where $L_t^y(\phi_x)$ denotes the local time at $t$ in $y$ of the semimartingale $(\phi_x(t))_{t \geq 0}$. Assuming $(r,\sigma,\kappa)$ known (or identified by extraneous observation schemes) we study the estimation of $x \mapsto B(x)$ when $B$ belongs to a parametric class of functions $\{B_\vartheta, \vartheta \in \Theta\}$ for some regular subset of the Euclidean space $\R^d$. Under a certain ordering property (Definition \ref{def:ordering} in Section \ref{sec:parametric estimation} below) that ensures identifiability of the model and suitable standard regularity properties, we realise a standard maximum likelihood proxy estimation of $B$ thanks to \eqref{expl transit} by maximising the contrast
$$\vartheta \mapsto \prod_{ u \in \mathbb U_{n}^\star}q_{B_\vartheta}(X_{u^-},X_u),\;\;\vartheta \in \Theta,$$
(with $\mathbb U_n^\star=\mathbb U_n \setminus \mathbb G_0$ and where $u^-$ denotes the unique parent of $u$)
and we prove that it achieves asymptotic efficiency and discuss its practical implementation. It is noteworthy that for the parametric estimation of $B$, there is no straightforward contrast minimisation procedure (at least we could not find any) whereas $q_B(x,y)$ is explicit. The fairly intricate dependence of $B$ in the representation \eqref{expl transit} makes however the whole scheme relatively delicate, both mathematically and numerically.\\

Clearly, other observation schemes are relevant in the context of cell division modelling. For instance, one could consider a (large) time $T>0$ and observe the branching process $X_t$ defined in \eqref{def big Markov} for every $t \in [0,T]$. This entails the possibility to extract the times $(T_u)$ at which branching events occur, like {\it e.g.} in \cite{HO}. However, the continuous time setting is drastically different and introduce the additional difficulty of bias sampling, an issue we avoid in the present context. Alternatively, one could consider the augmented statistical experiment where one observes $(X_u,T_u)_{u \in {\mathbb U}_n}$, but the underlying mathematical structure is presumably not simpler. Our results show in particular that for the parametric estimation of the branching rate $B$, although the times at which branching event occur are statistically informative, their observation is not necessary to obtain optimal rates of convergence as soon as $(r,\sigma,\kappa)$ are known.

\subsection{Organisation of the paper}
Section \ref{sec:defhyp modele} is devoted to the construction of the stochastic model, our assumptions and the accompanying statistical experiments. In particular, we have a nice structure enough so that explicit representations of $\mathcal P$ and $\mathcal Q$ are available (Proposition \ref{prop:explicit trans}). We give a first result on the geometric ergodicity of the model via an explicit Lyapunov function in Proposition \ref{th:ergodicity} and derive in Proposition \ref{thm:convl2} a rate of convergence for the variance of empirical measures of the data $\mathbb X^n=(X_u)_{u \in \mathbb U_n}$ against test functions $\varphi(X_u)$ or $\psi(X_{u^-},X_u)$ with a sharp control in terms of adequate norms for $\varphi,\psi$ that do not follow from the standard application of the geometric ergodicity of Proposition \ref{th:ergodicity}.  This is crucial for the subsequent applications to the nonparametric estimation of $\mathcal Q$ and its invariant measure $\nu$ that are given in Theorem \ref{thm:nonparametrique} of Section \ref{sec: est nonparam}. Section \ref{sec:parametric estimation} is devoted to the parametric estimation of the branching rate, where an asymptotically efficient result is proved for a maximum likelihood estimator in Theorem \ref{thm:norm_asympt}. It is based on a relatively sharp study of the transition $\mathcal Q$, thanks to local time properties of the stochastic flow that triggers the branching events. Section \ref{sec:numerical} is devoted to the numerical implementation of the parametric estimator of $B$. In particular, in order to avoid the computational cost of the explicit computation of $q_{\vartheta}(X_{u^-},X_u)$, we take advantage of our preceding results and implement a nonparametric estimator on Monte-Carlo simulations instead, resulting in a feasible procedure for practical purposes. The proofs are postponed to Section \ref{sec: proofs} and an Appendix Section \ref{sec:appendix} contains useful auxiliary results.

\section{A cell division model structured by a stochastic flow} \label{modele proba}
\subsection{Assumptions and well-posedness of the stochastic model}  \label{sec:defhyp modele}

\subsubsection*{Dynamics of the traits} Remember that $\mathcal{X} \subseteq \mathbb{R}$ is an open, possibly unbounded interval. The flow is specified by $r,\sigma:\mathcal{X}\rightarrow \mathcal{X}$ which are measurable and that satisfy the following assumption:

\begin{assumption}\label{assu:unique}
For some $r_1,\sigma_1,\sigma_2>0$, we have
$|r(x)|\leq r_1(1+|x|)$ and $\sigma_1\leq \sigma(x)\leq \sigma_2$, for every $x\in \mathcal X$.
Moreover, for some $r_2>0$, we have 
$\sgn(x)r(x)<0$ for $|x| \geq r_2$ (with $\sgn(x)=\mathbf{1}_{\{x> 0\}}-\mathbf{1}_{\{x\leq0\}}$).
\end{assumption}

%
Under Assumption \ref{assu:unique}, there is a unique strong solution to \eqref{stoch flow} (for instance \cite{oksendal2003stochastic}, Theorem 5.2.1.). 
We denote by $\left(\Phi_x(t),t\geq 0\right)$ the unique solution to \eqref{stoch flow} with initial condition $x\in\mathcal{X}$. In particular, $\left(\Phi_x(t),t\geq 0\right)$ is a strong Markov process and is ergodic (cf. \cite{kutoyants2013statistical}, Theorem 1.16.).
Note that when $\mathcal X$ is bounded, the drift condition $\sgn(x)r(x)<0$ for large enough $x$ can be dropped.

\subsubsection*{Division events.} An individual with trait $x$ dies at an instantaneous rate $x\mapsto B(x)$, where $B:\mathcal X \rightarrow [0,\infty)$ satisfies the following condition:
\begin{assumption}\label{assu:debut} The function $x \mapsto B(x)$ is continuous. Moreover,
for some $b_1,b_2 >0$ and $\gamma\geq 0$, we have $b_1 \leq B(x) \leq b_2\left| x\right|^{\gamma}+b_1$ for every $x\in\mathcal{X}$.
\end{assumption}
Under Assumptions \ref{assu:unique} and \ref{assu:debut}, the process $X$ in \eqref{def big Markov} is well defined and the size of the population does not explode in finite time almost-surely, see for instance Marguet \cite{marguet2016}. Note that the lower bounds for $\sigma$ and $B$ are not needed for the well-posedness of $X$ but rather for later statistical purposes. 

\subsubsection*{Fragmentation of the trait at division}
Finally, we make an additional set of assumptions on the fragmentation distribution $\kappa(z)dz$ that ensures in particular the non-degeneracy of the process.
\begin{assumption}\label{assu:noyau}
We have
\begin{itemize} 
\item[] $\mathrm{supp}(\kappa) \subset [\varepsilon, 1-\varepsilon]$ for some $0 < \varepsilon < 1/2$, 
\item[] $\inf_{z \in [\varepsilon, 1-\varepsilon]}\kappa(z) \geq \delta$.
\end{itemize}

\end{assumption}
This assumption is slightly technical and may presumably be relaxed. We emphasize that the density $\kappa(z)$ needs not be symmetric.

\subsubsection*{Representations of $\mathcal P$ and $\mathcal Q$}\label{sec:transition}
Under Assumptions \ref{assu:unique}, \ref{assu:debut} and \ref{assu:noyau}, we obtain closed-form formulae for the transition $\mathcal P$  defined via \eqref{def:transBMC} and the mean or marginal transition $\mathcal Q$ of the BMC $(X_{u})_{u \in \mathbb T}$, see \eqref{def:meanTransition} that also gives the transition probability of the discrete Markov chain with value in $\mathcal{X}$ corresponding to the trait at birth along an ancestral lineage. These representations are crucial for the subsequent analysis of the variance of the estimators of $\mathcal P$ and of the invariant measure $\nu$.

\begin{prop} \label{prop:explicit trans}
Work under Assumptions \ref{assu:unique}, \ref{assu:debut} and \ref{assu:noyau}. For every $x,y,y_1,y_2\in \mathcal X$, we have 
$$\mathcal P(x,dy_1dy_2)=p(x,y_1,y_2)dy_1dy_2\;\;\text{and}\;\;\mathcal Q(x,dy)=q(x,y)dy,$$
with
\begin{equation} \label{eq:def trans P}
p(x,y_1,y_2)=
\frac{\kappa\big(y_1/(y_1+y_2)\big)}{y_1+y_2}B\left(y_1+y_2\right)\sigma(y_1+y_2)^{-2}\mathbb{E}\Big[\int_0^{\infty}e^{-\int_0^t B(\phi_x(s))ds}dL_t^{y_1+y_2}(\phi_x)\Big]
\end{equation}
and
\begin{equation} \label{eq:def trans Q}
q(x,y)=\int_0^1\frac{\ktilde(z)}{z}B(y/z)\sigma(y/z)^{-2}\mathbb{E}\Big[\int_0^{\infty}e^{-\int_0^t B(\phi_x(s))ds}dL_t^{y/z}(\phi_x)\Big]dz,
\end{equation}
where $\ktilde(z)= \tfrac{1}{2}(\kappa(z)+\kappa(1-z))$ and $L_t^y(\phi_x)$ denotes the local time at $t$ in $y$ of the semimartingale $(\phi_x(t))_{t \geq 0}$.
\end{prop}
Notice that in the case of a symmetric fragmentation kernel, we have $\ktilde = \kappa$.

\subsection{Convergence of empirical measures} \label{sec conv empirical measures}

We study the convergence of empirical means of the form
\begin{equation} \label{eq:def empirical means}
\mathcal M_{\mathbb U_n}(\psi)=\frac{1}{|\mathbb U_n^\star|} \sum_{u \in \mathbb U_n^\star}\psi(X_{u^-},X_u)
\end{equation}
towards $\nu\mathcal Q(\psi)$ if $\mathbb U_n$ is a rich enough incomplete tree, for test functions $\psi:\mathcal X \times \mathcal X \rightarrow \R$. (If $\varphi:\mathcal X \rightarrow \R$ we set 
$\mathcal M_{\mathbb U_n}(\varphi)=|\mathbb U_n|^{-1} \sum_{u \in \mathbb U_n}\varphi(X_u) \rightarrow
\nu(\varphi)$ and we have a formal correspondence between the two expressions by writing $\psi(x,y) = \varphi(y)$ as a function of the second variable.) In order to derive nonparametric estimators of  $\nu$ and $\mathcal Q$ by means of kernel functions $\psi$ that shall depend on $n$, we need sharp estimates in terms of $\psi$, see Remark 1) after Proposition \ref{thm:convl2} below.

\subsubsection*{Convergence of $\mathcal Q$ to equilibrium} 
Assumptions \ref{assu:unique}, \ref{assu:debut} and \ref{assu:noyau} imply a drift condition for the Lyapunov function $V(x)=x^2$ on $\mathcal X$ and a minorisation condition over a small set so that in turn $\mathcal Q$ is geometrically ergodic.


Let $\mathbb Q=\mathbb Q(r_i,b_i,\sigma_i,\gamma,\varepsilon,\delta, i=1,2)$ be the class of all transitions $\mathcal Q = \mathcal Q(r,\sigma,B,\kappa)$ defined over $\mathcal X$ that satisfy Assumptions \ref{assu:unique}, \ref{assu:debut} and \ref{assu:noyau} with appropriate constants.
An invariant probability measure for $\mathcal Q$ is a probability $\nu$ on $\mathcal X$ such that $\nu\mathcal Q=\nu$, where $\nu\mathcal Q(dy)=\int_{x\in \mathcal X}\nu(dx)\mathcal Q(x,dy)$. Define 
$$\mathcal Q^r(x,dy)=\int \mathcal Q(x,dz) \mathcal Q^{r-1}(z,dy)\;\;\text{with}\;\;\mathcal Q^{0}(x,dy)=\delta_x(dy)$$
for the $r$-th iteration of $\mathcal Q$. For $\varphi:\mathcal X \rightarrow \R$, we set
$$|\varphi|_V=\sup_{x \in \mathcal X}\frac{|\varphi(x)|}{1+V(x)}$$
and write $\nu(\varphi) = \int_{\mathcal X}\varphi(x)\nu(dx)$ when no confusion is possible.

\begin{prop}[Convergence to equilibrium] \label{th:ergodicity}
Work under Assumptions \ref{assu:unique}, \ref{assu:debut} and \ref{assu:noyau}. Then any $\mathcal Q \in \mathbb Q$ admits an invariant probability distribution $\nu$. Moreover, for $V(x)=x^2$, there exist $C=C(\mathcal Q)>0$ and $\rho=\rho(\mathcal Q) \in (0,1)$ such that for every $m \geq 1$, the bound
$$\big|\mathcal Q^m\varphi-\nu(\varphi)\big|_V \leq C\rho^m\big|\varphi-\nu(\varphi)\big|_V$$
holds  as soon as $|\varphi|_V<\infty$. Moreover, $\sup_{\mathcal Q \in \mathbb Q}C(\mathcal Q)<\infty$ and $\sup_{\mathcal Q \in \mathbb Q}\rho(\mathcal Q)<1$.
\end{prop}
In particular, if $|\varphi|_\infty = \sup_{x \in \mathcal X} |\varphi(x)|$ is finite, we have
$\left|\mathcal Q^m\varphi(x)-\nu(\varphi)\right|\leq C\rho^m(1+V(x))|\varphi-\nu(\varphi)|_\infty$ for every $x\in \mathcal X$.

\subsubsection*{Sharp controls of empirical variances}
Proposition \ref{th:ergodicity} is the key ingredient in order to control the rate of convergence of empirical means of the form \eqref{eq:def empirical means} for appropriate observation schemes $\mathbb U_n \subset \mathbb T_n$.\\


We need some notation.  We denote by $|\cdot|_1$ the usual $L^1$-norm w.r.t. the Lebesgue measure on $\mathcal X \times \mathcal X$. For a function $\psi=\mathcal X \times \mathcal X \rightarrow \R$ we set
$\psi^\star(x)=\sup_{y \in \mathcal X}|\psi(x,y)|$ and $\psi_\star(y)=\sup_{x\in \mathcal X}|\psi(x,y)|$ and define
$$|\psi|_{\wedge 1} = 
\min\Big(\int_{\mathcal X \times \mathcal X}|\psi(x,y)|dxdy,\int_{\mathcal X}\sup_{x \in \mathcal X}|\psi(x,y)|dy\Big).$$
Note in particular that when $\psi(x,y)=\varphi(y)$ is a function of $y$ only, we may have that $|\psi_\star|_1 = \int_{\mathcal X}|\varphi(y)|dy$ is finite while $\psi$ is not integrable on $\mathcal X \times \mathcal X$ as a function of two variables. For a positive measure $\rho$ on $\mathcal X$, let also
$$|\psi|_{\rho}=\int_{\mathcal X\times \mathcal X}|\psi(x,y)|\rho(dx)dy+ |\psi|_{\wedge 1}.$$


We write $\PP_\mu$ for the law of $(X_{u})_{u \in \mathbb T}$ with initial distribution $\mu$ for $X_{\emptyset}$.
Remember that $V(x)=x^2$ from Proposition \ref{th:ergodicity}.
We shall further restrict our study to transitions $\mathcal Q \in \mathbb Q$ for which the geometric rate of convergence to equilibrium $\rho=\rho(\mathcal Q)$ given in Proposition \ref{th:ergodicity} satisfies $\rho(\mathcal Q)\leq 1/2$. Let $\mathbb Q_{1/2} \subset \mathbb Q$ denote the set of such transitions.

\begin{rk} \label{rk rho petit}
It is delicate to check in general that $\rho\leq 1/2$ but it is for instance satisfied in the following example:
\begin{itemize}
\item[i)] $\phi_x(t)$ is an Ornstein-Uhlenbeck process on $\mathcal X = \R$: we have $r(x)=-\beta x$ and $\sigma(x) = \sigma$ for every $x \in \mathcal X$ and some $\beta,\sigma >0$,
\item[ii)] the division rate is constant: we have $B(x) = b$ for every $x \in \mathcal X$ and some $b >0$,
\item[iii)] the fragmentation distribution is uniform: we have $\kappa(z) = 1/(1-2\varepsilon)$ on $[\varepsilon,1-\varepsilon]$ for some $\varepsilon>0$.
\end{itemize}
Adapting the proof of Proposition \ref{prop:minoration} below to this special case and using the explicit formula of $\rho$ given Theorem 1.2 in \cite{hairer2011}, we show in Appendix \ref{appendix:rho} that for $B$ small enough, we have $\rho<1/2$ in this example.
\end{rk}

Finally we consider observation schemes $\mathbb U_n$ that satisfy a certain sparsity condition that we quantify in the following definition

\begin{defi} \label{def incomplete tree} A regular incomplete tree is a subset $\mathbb U_n \subseteq \mathbb T_n$ (for $n \geq 1$) such that
\begin{itemize}
\item[(i)] $u \in \mathbb U_n$ implies $u^- \in \mathbb U_n$,
\item[(ii)] We have $0 < \liminf_{n \rightarrow \infty}2^{-n\varrho}\big|\mathbb U_n \cap \mathbb G_n\big| \leq \limsup_{n \rightarrow \infty}2^{-n\varrho}\big|\mathbb U_n \cap \mathbb G_n\big| < \infty$ for some $0 \leq \varrho \leq 1$. 
\end{itemize}
\end{defi}

\begin{prop}\label{thm:convl2}
Work under Assumptions \ref{assu:unique}, \ref{assu:debut} and \ref{assu:noyau}. Let $\mu$ be a probability measure on $\mathcal{X}$ such that $\mu(V^2)<\infty$. Let $\psi:\mathcal{X}\times \mathcal X\rightarrow \R$ a bounded function such that $\psi_\star$ is compactly supported. If $\mathbb U_n$ is a regular incomplete tree, the following estimate holds true:
\begin{align*}
\E_{\mu}\big[\big(\mathcal M_{\mathbb U_n}(\psi)-\nu(\mathcal Q\psi)\big)^2\big]&\lesssim |\mathbb U_n|^{-1}\big(|\psi^2|_{\mu+\nu}+
|\psi^\star \psi|_\mu
+\big(1+\mu(V^2)\big)|\psi_\star|_1|\psi|_{\nu}\big),
\end{align*}
where the symbol $\lesssim$ means up to an explicitly computable constant that depends on $\mathcal Q$ and on $\supp(\psi_\star)$ only. Moreover, the estimate is uniform in $\mathcal Q \in \mathbb Q_{1/2}$.
\end{prop}

Several remarks are in order: {\bf 1)} We have a sharp order in terms of the test functions $\psi$, that behave no worse than $\int_{\mathcal X^2}\psi^2$ under minimal regularity on $\nu$ which is satisfied, see Lemma \ref{lem: uplowbound densite} below (and of course $\mu$, although this restriction could be relaxed). This behaviour is the one expected for instance in the IID case and is crucial for the subsequent statistical application of Theorem \ref{thm:nonparametrique} where the functions $\psi$ will be kernel depending on $n$. {\bf 2)} The proof heavily relies on the techniques developed in Biteski Penda {\it et al.} \cite{Guillin2} or Guyon \cite{guyon2007limit} (more specifically, Theorems 11 and 12 of \cite{guyon2007limit} or Theorem 2.11 and 2.1 of \cite{Guillin2}, see also \cite{BHO2016, BRoche}). However, we need a slight refinement here, in order to obtain a sharp control in terms of the trial function $\psi$, similar to the behaviour of $\int_{\mathcal X^2}\psi^2$, while the aformentioned references would give a term of order $\sup_{x,y}|\psi(x,y)|$ that would not be sufficiently sharp for the nonparametric statistical analysis.
{\bf 3)} Proposition \ref{thm:convl2} has an analog in \cite{DHKR1} for piecewise deterministic growth-fragmentation models, but our proof is somewhat simpler here and sharper (we do not pay the superfluous logarithmic term in \cite{DHKR1}). {\bf 4)} Finally, note that in Proposition \ref{thm:convl2}, the observation $\mathbb U_n$ must be deterministic (or at least independent of $(X_u)_{u \in \mathbb T_n}$) otherwise biased selection may occur that would result in completely different behaviours of the empirical means (like for instance if $\mathbb U_n$ is allowed to contain stopping times on the tree).

\section{Statistical estimation}

\subsection{Nonparametric estimation of $\mathcal Q$ and $\nu$} \label{sec: est nonparam}

Under Assumptions \ref{assu:unique}, \ref{assu:debut} and \ref{assu:noyau}, any $\mathcal Q(x,dy)=q(x,y)dy$ admits an invariant probability measure $\nu(dx)=\nu(x)dx$, the regularity of $\nu(x)$ being inherited from that of $\mathcal Q$ via $\nu(x)=\int_{\mathcal X}q(z,x)\nu(dz)$.\\ 

Fix $(x_0,y_0) \in \mathcal X\times \mathcal X$. We are interested in constructing estimators of $q(x_0,y_0)$ and $\nu(x_0)$ from the observation $\mathbb X^n$ when both functions satisfy some H\"older regularity properties in the vicinity of $(x_0,y_0)$. To that end, we need approximating kernels.

\begin{defi}
A function $G:\mathcal X \rightarrow \R$ is a kernel of order $k$ if it is compactly supported and satisfies $\int_{\mathcal X}x^\ell G(x)dx={\bf 1}_{\{\ell=0\}}$ for $\ell=0,\ldots, k$.
\end{defi}
The construction and numerical tractability of approximating kernels is documented in numerous textbooks, see for instance Tsybakov \cite[Chapter 1]{Tsyb}. For bandwidth parameters $h,h_1,h_2>0$, we set 
$$G_{h}(y)= h^{-1}G(h^{-1}y)$$
and
$$G^{\otimes 2}_{h_1,h_2}(x,y)=h_1^{-1}h_2^{-1}G(h_1^{-1}x)G(h_2^{-1}y\big)$$
and obtain approximations of $\nu(y_0)$ and $q(x_0,y_0)$ by setting
$$
G_{h}\star \nu(y_0) = \int_{\mathcal X} G_{h}(y_0-y)\nu(y)dy
$$
and
$$
G_{h}\star q(x_0,y_0) = \int_{\mathcal X\times \mathcal X} G_{h_1,h_2}^{\otimes 2}(x_0-x,y_0-y)\nu(x)q(x,y)dxdy. 
$$
The convergence of $\mathcal M_{\mathbb U_n}(\varphi)$ to $\nu(\varphi)$ suggests to pick $\varphi=G_h(x_0-\cdot)$. Then $\mathcal M_{\mathbb U_n}(G_h(x_0-\cdot))$ is close to
$G_{h}\star \nu(x_0)$ for small enough $h$ and can be used as a proxy of $\nu(x_0)$. We obtain the estimator
$$\widehat \nu_{n}(x_0)=\mathcal M_{\mathbb U_n}\big(G_h(x_0-\cdot)\big),$$
specified by the choice of $h>0$ and the kernel $G$.
Likewise, with $\psi = G_{h_1,h_2}(x_0-\cdot,y_0-\cdot)$, an estimator of $q(x_0,y_0)$ is obtained by considering the quotient estimator with numerator $\mathcal M_{\mathbb U_n}(\psi)$
that is close to  
$G_{h_1,h_2}\star \big(\nu(\cdot) q(\cdot, \cdot)\big)(x_0,y_0)$ and denominator $\widehat \nu_{n,h}(x_0)$ in order to balance the superfluous weight $\nu(x_0)$ in the numerator. 
We obtain the estimator
$$\widehat q_{n}(x_0,y_0)=\frac{\mathcal M_{\mathbb U_n}\big(G_{h_1,h_2}^{\otimes 2}(x_0-\cdot,y_0-\cdot)\big)}{\mathcal M_{\mathbb U_n}\big(G_h(x_0-\cdot)\big)\vee \varpi},$$
specified by the choice of $h,h_1,h_2>0$, a threshold $\varpi >0$ and the kernel $G$. In order to quantify the kernel approximation, we introduce anisotropic H\"older classes. For $\alpha >0$, we write $\alpha=\lfloor \alpha \rfloor+\{\alpha\}$ with $\lfloor \alpha \rfloor$ an integer and $0< \{\alpha\} \leq 1$. 
\begin{defi}
Let $\alpha,\beta >0$ and $\mathcal V_{x_0}$ and $\mathcal V_{y_0}$ be bounded neighbourhoods of $x_0$ and $y_0$.
\begin{itemize}
\item[i)] The function $\varphi:\mathcal V_{x_0}\rightarrow \R$ belongs to the H\"older class $\mathcal H^\alpha(x_0)$ if
\begin{equation} \label{eq:defi holder}
|\varphi^{(\lfloor \alpha \rfloor)}(y)-\varphi^{(\lfloor \alpha \rfloor)}(x)| \leq C|y-x|^{\{\alpha\}}\;\;\text{for every}\;\; x,y\in \mathcal V_{x_0}.
\end{equation}
\item[ii)] The function $\psi: \mathcal V_{x_0} \times \mathcal V_{y_0} \rightarrow \R$ belongs to the anisotropic H\"older class $\mathcal H^{\alpha,\beta}(x_0,y_0)$ if 
$$x \mapsto \psi(x,y_0) \in \mathcal H^{\alpha}(x_0)\;\;\text{and}\;\;y \mapsto \psi(x_0,y) \in \mathcal H^{\beta}(y_0)$$
hold simultaneously.
\end{itemize}
\end{defi}
We obtain a semi-norm on $\mathcal H^\alpha(x_0)$ by setting 
$|\varphi|_{\mathcal H^\alpha(x_0)}=\sup_{x \in \mathcal V_{x_0}}|\varphi(x)|+c_\alpha(\varphi),$
where 
$c_\alpha(\varphi)$ is the smallest constant for which \eqref{eq:defi holder} holds. Likewise, we equip $\mathcal H^{\alpha,\beta}(x_0,y_0)$ with the semi-norm 
$|\psi|_{\mathcal H^{\alpha,\beta}(x_0,y_0)}=|\psi(\cdot,y_0)|_{\mathcal H^\alpha(x_0)}+|\psi(x_0,\cdot)|_{\mathcal H^\beta(y_0)}$.
The space $\mathcal H^{\alpha,\beta}(x_0,y_0)$ is appended with (semi) H\"older balls
$$\mathcal H^{\alpha,\beta}(x_0,y_0)(R)=\big\{\psi:\mathcal X \times \mathcal X \rightarrow \R, |\psi|_{\mathcal H^{\alpha,\beta}(x_0,y_0)} \leq R \big\},\;\;R>0.$$

We are ready to state our convergence result over transitions $\mathcal Q$ that belong to
$$\mathbb Q_{1/2}^{\alpha,\beta}(R) = \mathbb Q_{1/2} \cap \mathcal H^{\alpha,\beta}(x_0,y_0)(R),\;\;R>0,$$
with a slight abuse of notation.
\begin{thm} \label{thm:nonparametrique}
Work under Assumptions \ref{assu:unique}, \ref{assu:debut} and \ref{assu:noyau}. Assume that the initial distribution $\mu$ is absolutely continuous w.r.t. the Lebesgue measure with a locally bounded density function and satisfies $\mu(V^2)<\infty$.

Let $\alpha,\beta >0$. Specify $\widehat \nu_{n}(y_0)$ by a kernel of order $k > \max\{\alpha,\beta\}$ and $h=|\mathbb U_n|^{-1/(2\beta+1)}$ and $\widehat q_{n}(x_0,y_0)$ with the same kernel and $h_1=|\mathbb U_n|^{-s(\alpha,\beta)/(\alpha \wedge \beta)(2s(\alpha,\beta)+1)}$, $h_2=|\mathbb U_n|^{-s(\alpha,\beta)/\beta(2s(\alpha,\beta)+1)}$ and $\varpi = \varpi_n \rightarrow 0$. Then, if $\mathbb U_n$ is an $\varrho$-regular incomplete tree, for every $R>0$,
$$\sup_{\mathcal Q \in \mathbb Q_{1/2}^{\alpha,\beta}(R)}\big(\E_\mu\big[\big(\widehat \nu_{n}(y_0)-\nu(y_0)\big)^2\big]\big)^{1/2} \lesssim |\mathbb U_n|^{-\beta/(2\beta+1)}$$
and
$$\sup_{\mathcal Q \in \mathbb Q_{1/2}^{\alpha,\beta}(R)}\big(\E_\mu\big[\big(\widehat q_{n}(x_0,y_0)-q(x_0,y_0)\big)^2\big]\big)^{1/2} \lesssim \varpi_n^{-1}|\mathbb U_n|^{-s(\alpha,\beta)/(2s(\alpha,\beta)+1)}$$
hold true, where $s(\alpha,\beta)^{-1}=(\alpha\wedge \beta)^{-1}+\beta^{-1}$ is the effective anisotropic smoothness associated with $(\alpha,\beta)$.
\end{thm}
Several remarks are in order:
{\bf 1)} We obtain an optimal result in the minimax sense for estimating $\nu(y_0)$ and in the case $\beta \geq \alpha$ for estimating $q(x_0,y_0)$. This stems from the fact that the representation $\nu(x) = \int_{\mathcal X}\nu(y)q(y,x)dy$ henceforth $q\in \mathcal H^{\alpha,\beta}$ implies that $\nu \in \mathcal H^{\beta}$. In turn, the numerator of $\widehat q_n(x_0,y_0)$ is based on the estimation of the function $\nu(x)q(x,y) \in \mathcal H^{\alpha \wedge \beta,\beta}$.
{\bf 2)} In the estimation of $q(x_0,y_0)$, we have a superfluous term $\varpi_n^{-1}$ in the error that can be taken arbitrarily small, and that comes from the denominator of the estimator. It can be removed, however at a significant technical cost. Alternatively, one can get rid of it by weakening the error loss: it is not difficult to prove
$$
\big(\E_\mu\big[\big(\widehat q_{n}(x_0,y_0)-q(x_0,y_0)\big)^p\big]\big)^{1/p} \lesssim |\mathbb U_n|^{-s(\alpha,\beta)/(2s(\alpha,\beta)+1)}\;\;\text{for every}\;\;0 <p <2,$$
and the result of course also holds in probability. {\bf 3)} The assumption that $\mu$ is absolutely continuous can also be removed. {\bf 4)} Finally, a slightly annoying fact is that the estimators $\widehat \nu_n(x_0)$ and $\widehat q_n(x_0,y_0)$ require the knowledge of $(\alpha,\beta)$ to be tuned optimally, and this is not reasonable in practice. It is possible to tune our estimators in practice by cross-validation in the same spirit as in \cite{HO}, but an adaptive estimation theory still needs to be established. This lies beyond the scope of the paper, and requires concentration inequalities, a result we do not have here, due to the fact that the model is not uniformly geometrically ergodic (otherwise, we could apply the same strategy as in \cite{BHO2016, BRoche}). 

\subsection{Parametric estimation of the division rate} \label{sec:parametric estimation}
In order to conduct inference on the division rate $x \mapsto B(x)$, we need more stringent assumptions on the model so that we can apply the results of Proposition \ref{thm:convl2}. The  main difficulty lies in the fact that we need to apply Proposition \ref{thm:convl2} to test functions of the form $\psi(x,y) = \log q(x,y)$ when applied to the loglikelihood of the data, and that these functions are possibly unbounded.\\

\subsubsection*{A stochastic trait model as a diffusion on a compact with reflection at the boundary} We circumvent this difficulty by assuming that the trait $\phi_x(t)$ of each individual evolves in a bounded interval with reflections at the boundary and with no loss of generality, we take $\mathcal X=[0,L]$ for some $L>0$. The dynamics of the traits now follows 
\begin{align}\label{eq:diffusion_compact}
d\phi_x(t)=r(\phi_x(t))dt+\sigma(\phi_x(t))dW_t+d\ell_t,
%
\end{align}
where the solution $(\ell_t)_{t \geq 0}$ to $\ell_t=\int_0^t(\mathbf{1}_{\{\phi_x(s)=0\}}+\mathbf{1}_{\{\phi_x(s)=L\}})d\ell_s$ accounts for the reflection at the boundary and 
$(W_t)_{t \geq 0}$ is a standard Brownian motion. Under Assumption \ref{assu:unique} (that reduces here to the boundedness of $r,\sigma$ and the ellipticity of $\sigma$) there exists a unique strong solution to \eqref{eq:diffusion_compact}, see for instance Theorem 4.1. in \cite{tanaka79}.\\ 


A slight modification of Proposition \ref{prop:explicit trans} gives the following explicit formulae for the transitions $\mathcal P$ and $\mathcal Q$. Remember that by Assumption \ref{assu:noyau}, we have $\supp(\kappa) \subset [\varepsilon, 1-\varepsilon]$. Define
$$\mathcal D=\big\{ 0< y_1\leq \varepsilon L,\ \tfrac{\varepsilon}{1-\varepsilon}y_1\leq y_2\leq \tfrac{1-\varepsilon}{\varepsilon}y_1\big\} \cup \big\{ \varepsilon L\leq y_1\leq (1-\varepsilon)L,\ \tfrac{\varepsilon}{1-\varepsilon}y_1\leq y_2\leq \tfrac{L-y_1}{y_1} \big\}.$$
Then  the explicit formula for $p(x,y_1,y_2)$ given in \eqref{eq:def trans P} remains unchanged provided $(x,y_1,y_2) \in \mathcal X \times \mathcal D$ and it vanishes outside of $\mathcal X \times \mathcal D$. For $q(x,y)$, the formula \eqref{eq:def trans Q} now becomes
\begin{equation} \label{eq:forme transition param}
q(x,y)=\int_{y/L}^1\frac{\ktilde(z)}{z}B(y/z)\sigma(y/z)^{-2}\mathbb{E}\left[\int_0^{\infty}e^{-\int_0^t B(\phi_x(s))ds}dL_t^{y/z}(\phi_x)\right]dz,
\end{equation}
for $(x,y) \in \mathcal X \times [0, (1-\varepsilon) L]$ and $0$ otherwise.\\

Adapting the proof of Proposition \ref{th:ergodicity}  to the case of a diffusion living on a compact interval (formally replacing $[-w,w]$ by $[0,L]$ in the proof of Proposition \ref{prop:minoration} below) one checks that Proposition \ref{th:ergodicity} remains valid in this setting (applying for instance Theorem 4.3.16 in \cite{CMR05}). In turn, Proposition \ref{thm:convl2} also holds true in the case of a reflected diffusion. For parametric estimation, the control on the variance of $\mathcal M_{\mathbb U_n}(\psi)$ is less demanding and we will simply need the following  

\begin{cor} \label{thm:convl2_compact}
Work under Assumptions \ref{assu:unique}, \ref{assu:debut} and \ref{assu:noyau} in the case of a reflected diffusion on $[0,L]$ for the evolution of the trait $(\phi_x(t),\ t \geq 0)$. Let $\psi:\mathcal X \times \mathcal X \rightarrow \R$.  Then, for any probability measure $\mu$, if $\mathbb U_n$ is a $\varrho$-regular incomplete tree, we have
$$
\sup_{\mathcal Q \in \mathbb Q_{1/2}} \E_\mu\big[(\mathcal M_{\mathbb U_n}(\psi)-\nu(\mathcal Q\psi))^2\big] \lesssim |\mathbb U_n|^{-1}\sup_{x,y}\psi(x,y)^2.$$ 
\end{cor}
\subsubsection*{Maximum likelihood estimation}
From now on, we fix a triplet $(r_0,\sigma_0,\kappa_0)$ and we let the division rate $x\mapsto B(x)$ belong to a parametric class
$$\mathcal B = \big\{B:\mathcal X \rightarrow \R, B(x)=B_0(\vartheta,x),x\in \mathcal X, \vartheta \in \Theta\big\},$$
where $x\mapsto B_0(x,\vartheta)$ is known up to the parameter $\vartheta \in \Theta$, and $\Theta \subset \R^d$ for some $d \geq 1$ is a compact subset of the Euclidean space. In this setting, the model is entirely characterised by $\vartheta$ which is our parameter of interest. A first minimal stability requirement of the parametric model is the following

\begin{assumption} \label{assu:param stability} 
We have $\mathbb Q(\mathcal B) = \{\mathcal Q=\mathcal Q(r_0,\sigma_0,B, \kappa), B \in \mathcal B\} \subset \mathbb Q_{1/2}$.
\end{assumption}
A second minimal requirement is the identifiability of the class $\mathcal B$, namely the fact that the map
$$B\mapsto \mathcal Q(r_0,\sigma_0, B, \kappa_0)$$ 
from $\mathcal B$ to $\mathbb Q$ is injective. This is satisfied in particular if $\mathcal B$ satisfies a certain orderliness property.

\begin{defi} \label{def:ordering}
A class $\mathcal B$ of functions from $\mathcal X \rightarrow [0,\infty)$ is orderly if $\varphi_1,\varphi_2 \in \mathcal B$ implies either $\varphi_1(x)\leq \varphi_2(x)$ for every $x \in \mathcal X$ or $\varphi_2(x)\leq \varphi_1(x)$ for every $x \in \mathcal X$.
\end{defi}

\begin{prop} \label{prop:identifiabilite}
Let $\mathcal B$ be orderly in the sense of Definition \ref{def:ordering} and $\mathcal Q(\mathcal B) \subset \mathbb Q$ for some $(r_0,\sigma_0, \kappa_0)$. Then $B \mapsto \mathcal Q(r_0,\sigma_0, B, \kappa_0)$ is injective.
\end{prop}

We further stress the dependence on $\vartheta$ by introducing a subscript in the notation whenever relevant.
We formally obtain a statistical experiment 
$$\mathcal E^n = \big\{\mathbb P_\vartheta^n,\vartheta \in \Theta \big\}$$
by letting $\mathbb P_\vartheta^n$ denote the law of $\mathbb X^n=(X_{u}, u \in \mathbb U_n)$ under $\mathcal P_{\vartheta}$ with initial condition $X_\emptyset$ distributed according to $\nu_\vartheta$ on the product space $\mathcal X^{|\mathbb U_n|}$ endowed with its Borel sigma-field. Therefore, the process is supposed to be stationary for simplicity. The experiment $\mathcal E^n$ is dominated by the Lebesgue measure on $\mathcal X^{|\mathbb T_n|}$ and we obtain a likelihood-type function by setting
\begin{equation} \label{eq:likelihood}
\mathcal L_n\big(\vartheta, (X_u,u \in \mathbb U_n)\big) = \prod_{u \in \mathbb U_n^\star} q_{\vartheta}(X_{u^-},X_u).
\end{equation}
Taking any maximiser of \eqref{eq:likelihood} we obtain a maximum likelihood estimator
$$\widehat \vartheta_n \in \argmax_{\vartheta \in \Theta} \mathcal L_n\big(\vartheta, \mathbb X^n\big)$$
provided a maximiser exists. As noted by a referee, in the case where we observe the full tree, {\it i.e.} $\mathbb U_n = \mathbb T_n$ and thus $\varrho = 1$ in Definition \ref{def incomplete tree}, we have access to the observation $(X_u, X_{u_0}+X_{u_1})$ for every $u \in \mathbb T_{n-1}$. Going back to the expression of the transition density of the bifurcating process itself in \eqref{eq:def trans P}, we may alternatively maximise the contrast 
\begin{align*}
\vartheta & \mapsto \prod_{u \in \mathbb T_n^\ast} p_{B_{\vartheta}}(X_u, X_{u0}, X_{u1}) \\
& = \prod_{u \in \mathbb T_n^\star}\frac{\kappa\big(X_{u0}/(X_{u0}+X_{u1})\big)}{(X_{u0}+X_{u1})\sigma(X_{u0}+X_{u1})^{2}}B_{\vartheta}\left(X_{u0}+X_{u1}\right)\mathbb{E}\Big[\int_0^{\infty}e^{-\int_0^t B_{\vartheta}(\phi_{X_u}(s))ds}dL_t^{X_{u0}+X_{u1}}(\phi_{X_u})\Big],
\end{align*}
 or equivalently
 $$\vartheta  \mapsto \prod_{u \in \mathbb T_n^\ast} B_{\vartheta}\left(X_{u0}+X_{u1}\right)\mathbb{E}\Big[\int_0^{\infty}e^{-\int_0^t B_{\vartheta}(\phi_{X_u}(s))ds}dL_t^{X_{u0}+X_{u1}}(\phi_{X_u})\Big].$$
In particular, the latter contrast does not depend on $\kappa$ which is merely a nuisance parameter here and that can be ignored, in this specific setting, where one can observe the complete tree $(X_u, u \in \mathbb T_n)$.
\subsubsection*{Convergence results and asymptotic efficiency}

We first have an existence and consistency result of $\widehat \vartheta_n$ under the following non-degeneracy assumption that strengthens Assumption \ref{assu:debut}.
\begin{assumption} \label{assu:nondegeneracy}
The function $B_0: \Theta \times \mathcal X \rightarrow [0,\infty)$ is continuous and for some positive $b_3, b_4$, we have
$$0 < b_3 \leq \inf_{\vartheta, x}B_0(\vartheta, x) \leq \sup_{\vartheta,x}B_0(\vartheta,x)  \leq b_4$$
Moreover,  the class  $\mathcal B = \big\{B_0(\vartheta,\cdot),\vartheta \in \Theta\big\}$ is orderly in the sense of Definition \ref{def:ordering}. 
\end{assumption}

\begin{thm}\label{th:consistency}
Work under Assumptions \ref{assu:unique}, 
\ref{assu:noyau}, \ref{assu:param stability} and \ref{assu:nondegeneracy}. Then, for every $\vartheta \in \Theta$, $\widehat \vartheta_n$ converges to $\vartheta$ in probability as $n \rightarrow \infty$.
\end{thm}

Our next result gives an explicit rate of convergence and asymptotic normality for $\widehat \vartheta_n$. We need further regularity assumptions.
\begin{assumption}\label{assu:reg_B}
The set $\Theta$ has non empty interior and, for every $x\in\mathcal{X}$ the map $\vartheta\mapsto B_0(\vartheta,x)$ is three times continuously differentiable.
Moreover, for every $1\leq i,j,k\leq d$:
$$\sup_{\vartheta,x}|\partial_{\vartheta_i}B_0(\vartheta,x)|+\sup_{\vartheta,x}|\partial^2_{\vartheta_i\vartheta_j}B_0(\vartheta,x)|+\sup_{\vartheta,x}|\partial^3_{\vartheta_i\vartheta_j\vartheta_k}B_0(\vartheta,x)|<\infty.$$
\end{assumption}

Introduce the Fisher information operator $\Psi(\vartheta)=\nu_{\vartheta}\mathcal{Q}_{\vartheta}\Big((\partial_{\vartheta}\log q_\vartheta)(\partial_{\vartheta}\log q_\vartheta)^T\Big)$ at point $\vartheta \in \Theta$ as the $d\times d$-matrix with entries:
$$
\Psi(\vartheta)_{i,j}=\nu_{\vartheta}\mathcal Q_{\vartheta}\Big(\frac{\partial_{\vartheta_i}q_{\vartheta}\,\partial_{\vartheta_j}q_{\vartheta}}{q_{\vartheta}^2}\Big) = \int_{\mathcal X \times \mathcal X}\frac{\partial_{\vartheta_i}q_{\vartheta}(x,y)\,\partial_{\vartheta_j}q_{\vartheta}(x,y)}{q_{\vartheta}(x,y)^2}\nu_{\vartheta}(x)q_\vartheta(x,y)dxdy,
$$
for $1\leq i,j\leq d$.
\begin{assumption} \label{assu:Ppsi_invertible}
For every $\vartheta$ in the interior of $\Theta$, the matrix $\Psi(\vartheta)$ is nonsingular.
\end{assumption}
Although standard in regular parametric estimation, Assumption \ref{assu:Ppsi_invertible} is not obviously satisfied even if we have the explicit formula  \eqref{eq:forme transition param}, for $q_\vartheta(x,y)$, due to its relatively intricate form. We can however show that it is satisfied in the special case of a trait evolving as a reflected diffusion with constant drift. More general parametrisations are presumably possible, adapting the proof delayed until Appendix \ref{appendix:preuve prop constant drift}.   
\begin{prop} \label{prop:constant drift}
Assume $d=1$, $B_0(\vartheta,x) = \vartheta$ for every $x\in \mathcal X$, with $\Theta=[\vartheta_1,\vartheta_2] \subset (0,\infty)$, $r(x) = r_1<0$ and $\sigma(x)= \sigma_0 >0$ for every $x\in\mathcal{X}$. Let $\kappa(z) = (1-2\varepsilon)^{-1}$ for every $z\in [\varepsilon,1-\varepsilon]$. There exists an explicit open interval $\mathcal I\subset (0,1/2)$ such that Assumption \ref{assu:Ppsi_invertible} is satisfied as soon as $\varepsilon \in \mathcal I$.
\end{prop}
We are ready to state our final result on asymptotic normality of $\widehat \vartheta_n$.
\begin{thm}\label{thm:norm_asympt}
Work under Assumptions \ref{assu:unique}, 
\ref{assu:noyau}, \ref{assu:param stability}, \ref{assu:nondegeneracy}, \ref{assu:reg_B} and \ref{assu:Ppsi_invertible}. For every $\vartheta$ in the interior of $\Theta$, if $\mathbb U_n$ is a $\varrho$-regular incomplete tree, we have
\begin{align*}
\big|\mathbb{U}_n\big|^{1/2}(\widehat{\vartheta}_n-\vartheta) \rightarrow \mathcal N\big(0,\Psi(\vartheta)^{-1}\big)
\end{align*}
in distribution as $n\rightarrow \infty$, where $\mathcal N(0,\Psi(\vartheta)^{-1})$ denote the $d$-dimensional Gaussian distribution with mean $0$ and covariance the inverse of the Fisher matrix $\Psi(\vartheta)$.
\end{thm}
Several remarks are in order:
{\bf 1)} Although asymptotically optimal, the practical implementation of $\widehat \vartheta_n$ is a challenging question that we plan to address in a systematic way. {\bf 2)} As for classical estimation in diffusion processes (see {\it e.g.} \cite{Dacunha, vgcjj}), the assumptions of Theorem \ref{thm:norm_asympt}, especially Assumption \ref{assu:Ppsi_invertible} are standard. However, the fact that they hold true in the simple case of Proposition \ref{prop:constant drift} and a glance at the proof is an indication that they are certainly true in wider generality.
\section{Numerical implementation}\label{sec:numerical}
We consider the implementation of the estimator $\widehat \vartheta_n$ in the case of a branching population structured by a trait drawn according to a Brownian motion reflected on $[0,1]$, namely
\begin{align*}
\phi_x(t)=x+W_t+\ell_t,\;\;\ell_t=\int_0^t(\mathbf{1}_{\{\phi_x(s)=0\}}+\mathbf{1}_{\{\phi_x(s)=1\}})d\ell_s,\;\;t \geq 0,
\end{align*}
where $(W_t)_{t\geq 0}$ is a standard Brownian motion. We pick $\kappa(z) = (1-2\varepsilon)^{-1}\mathbf{1}_{[\varepsilon,1-\varepsilon]}(z)$ so that an individual with trait $x$ at division splits into two individuals with traits $Ux$ and $(1-U)x$ respectively, where $U$ is uniformly distributed on $[\varepsilon,1-\varepsilon]$. We pick $\varepsilon = 10^{-4}$. 

\subsection{Generation of simulated data}\label{sec:simu_data}
We test our estimation procedure on simulated data. Given a division rate $B$ and an initial trait $x_0\in\mathbb{R}$, we construct a dataset constituted of a full tree of size $|\mathbb T_n|=2^{n+1}-1$ using a queue. \\

{\it Initialisation step.} We begin with one individual in the queue with trait $x_0$ at time $0$. It is the ancestor of the population\\

{\it While step.} While the queue is not empty, we pick $u$ in the queue,  
\begin{itemize}
\item[i)] we simulate the dynamics $(\phi_{X_u}(t), 0\leq t\leq T)$ for the trait of $u$ using the Euler scheme for reflected stochastic differential equations of \cite{lepingle} with initial condition $X_u$ and time step $\Delta t$ until time $T$, for some $T$ sufficiently large,
\item[ii)] we draw the lifetime $\zeta_u$ of $u$ by rejection sampling,
\item[iii)] if $|u|\leq n-1$, we add to the queue two new individuals with respective traits at birth given by $X_{u0} = \eta x$ and $X_{u1} = (1-\eta)x$ where $\eta$ is a realisation of a uniform random variable on $[\varepsilon,1-\varepsilon]$ and $x=\phi_{X_u}(\zeta_u)$ is the trait of $u$ at division,
\item[iv)] we add the pairs $(X_u,X_{u0})$ and $(X_u,X_{u1})$ to the dataset,
\item[v)] we remove the individual $u$ from the queue.
\end{itemize}
\subsection{Implementation of the maximum likelihood type contrast}
We pick $\mathbb U_n = \mathbb T_n$. For a given dataset $\mathbb{X}^n$, we approximate $\mathcal L_n\big(\vartheta, (X_u,u \in \mathbb T_n)\big)$ using, for a given $\vartheta$, the nonparametric estimator $\widehat q_{n}(X_{u^-},X_u)$ introduced in Section \ref{sec: est nonparam}.\\
More specifically, we implement $\widehat q_{n}(x_0,y_0)=\tfrac{\mathcal M_{\mathbb U_n}\big(G_{h_1,h_2}^{\otimes 2}(x_0-\cdot,y_0-\cdot)\big)}{\mathcal M_{\mathbb U_n}\big(G_h(x_0-\cdot)\big)\vee \varpi}$  for every $(x_0,y_0)$ on a grid  of mesh $n_1^{-1}\times n_2^{-1}$ of $[0,1]\times[0,1]$ with $n_1=n_2=200$, $G(x)=(2\pi)^{-1/2}\exp(-x^2/2)$, $h = 2|\mathbb{T}_n|^{-1/3}$, $h_1 = h_2 = 10^{-1} \cdot h^{1/2}$, $\varpi = 10^{-6}$. We next use an interpolation scheme with splines provided by the package {\it{Interpolations}} in Julia \cite{bezanson2017julia} to compute the value of the transition at each point of the dataset $(X_{u^-},X_u)\in\mathbb{X}^n$.  For synthetic data, we pick $n=19$, resulting in a tree of size $2^{20}-1=1\,048\,575$ with initial value $x_0=0.5$ and $\Delta t=5\times 10^{-4}$. 
\subsection{Results}
We consider the following parametric classes $B_0(\vartheta,x)=\vartheta$ and $B_1(\vartheta,x)=1+\vartheta x$.
We compute $300$ Monte-Carlo samples of size $|\mathbb{T}_n| = 2^{15}-1=32\,767$ for $\vartheta =\vartheta_1 = 2$ and $\vartheta = \vartheta_2 = 15$ in both cases. Therefore, we apply our results to four different cases. In each case, we approximate $q_{\vartheta}(X_{u^-},X_u)$ for different values of $\vartheta \in \Theta = [\vartheta_{\min},\vartheta_{\max}]$ and we compute the corresponding $\widehat \vartheta_n$. We progressively reduce the increment $\Delta\vartheta$ for the choice of $\vartheta$ until the contrast of likelihood starts to be noisy (see Figure \ref{fig:smooth}), adapting at each level the choice for the upper and lower bounds of $\Theta$.
\begin{figure}
\includegraphics[scale=0.45]{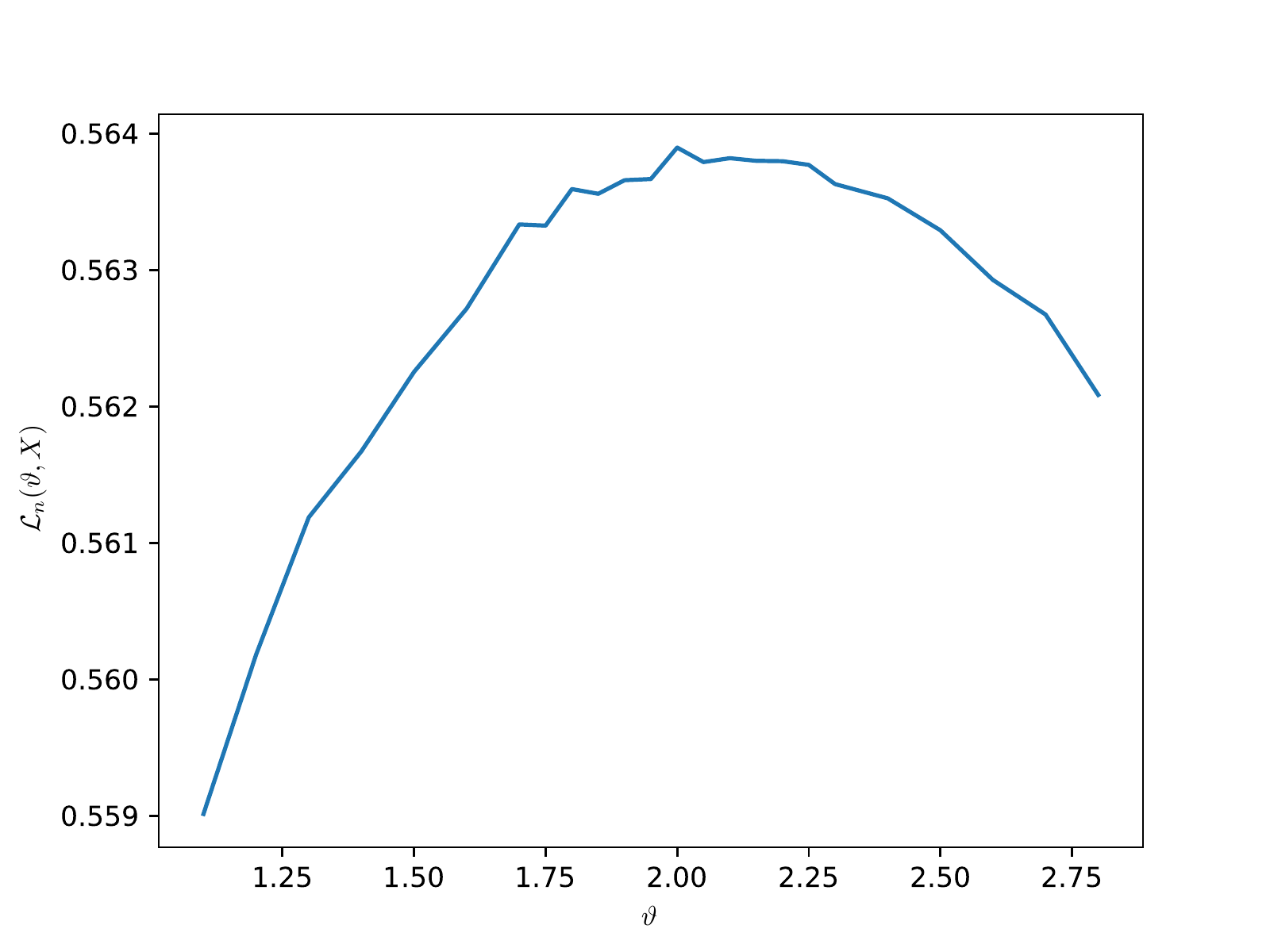}
\includegraphics[scale=0.45]{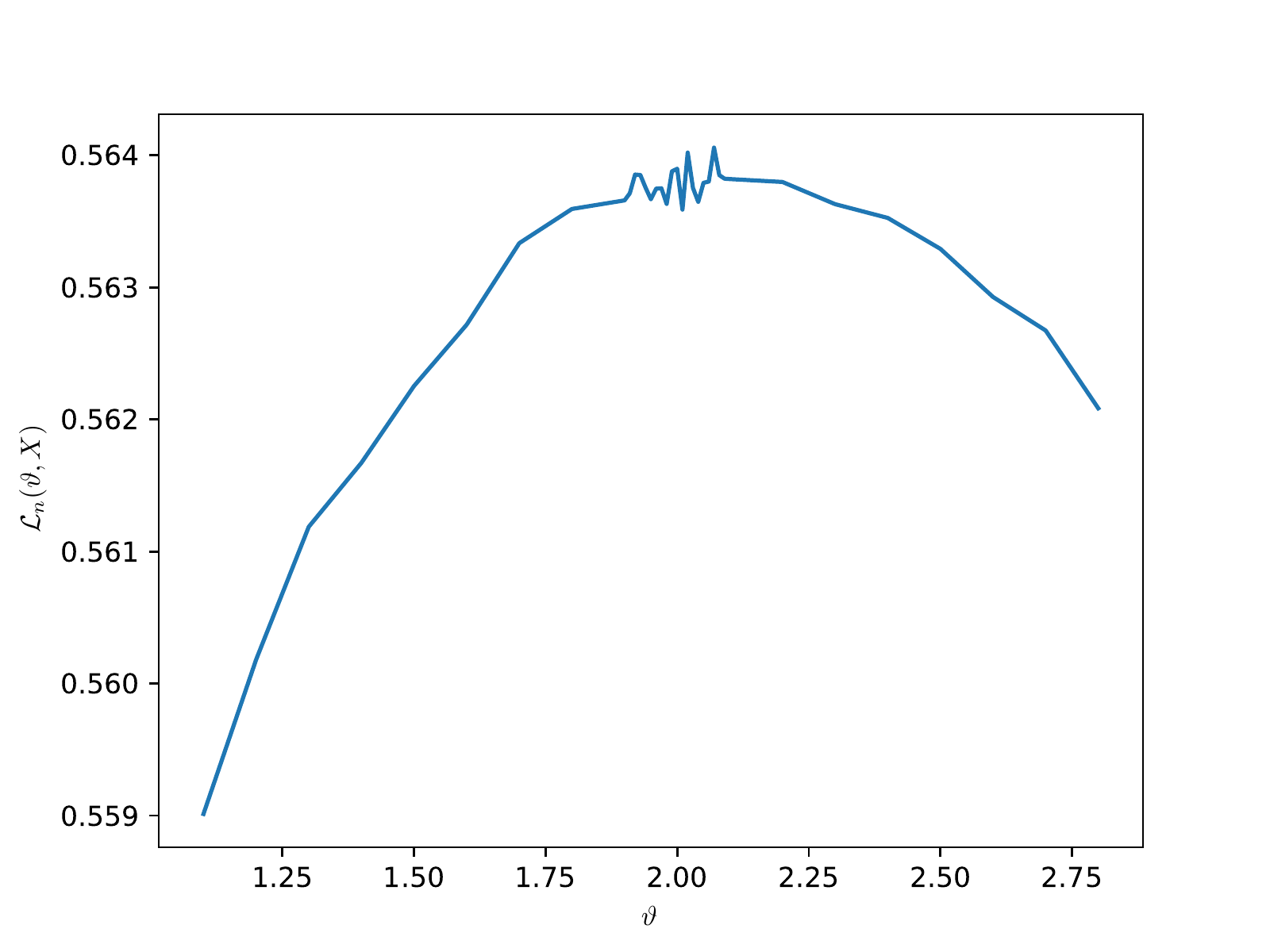}
\caption{The contrast $\mathcal L_n\big(\vartheta, (X_u,u \in \mathbb T_n)\big)$ for $n=14$ and $B(x,\vartheta)\equiv 2$ with $\Delta\vartheta=0.05$ (left) and $\Delta\vartheta=0.01$ (right). We look at the results only for $\Delta\vartheta = 0.05$ because for smaller values of the increment, the noise due to computational errors is too important compared to the different values of the contrast.}
\label{fig:smooth}
\end{figure} The results are displayed in Table \ref{table:Bconstant}. We recover the parameter in all four cases, with various accuracies. The most accurate value is obtained for $B_1$ with a small value of the parameter, {\it{i.e.}} $\vartheta=2$. We did not reach the optimal accuracy $|\mathbb T_n^\star|^{-1/2}\approx 6\times 10^{-3}$.
One could presumably obtain a better accuracy by choosing a finer discretisation of $[0,1]\times [0,1]$ for the computation of the $\widehat q_n$. But this choice leads to an important increase of the computational time. The results in the case of a linear division rate are less accurate. Those results could also probably be improved using a cross-validation procedure for the choice of the bandwidth parameters $h,h_1,h_2$.
\begin{table}[h]
\centering

\begin{tabular}{|c|c|c|c|c|c|c|}
\hline
$B$ & $\vartheta$ & Mean & Std Dev. & $\vartheta_{\min}$ & $\vartheta_{\max}$ & $\Delta \vartheta$\\
\hline
$B_0$&$2$ & $1.9908$ & $0.0845$ & $1.8$ & $2.2$ & $0.05$\\
&$15$ & $15.029$ & $0.4763$ & $14$ & $16$ & $0.25$\\ 
\hline
$B_1$&$2$ & $2.1392$ & $0.3485$ & $1.375$ & $3$ & $0.125$\\
&$15$ & $15.023$& $0.6425$ & $13.75$ & $16.5$ & $0.25$\\
\hline
\end{tabular}
\vspace{0.3cm}
\caption{Results for $B$ in $\mathcal{B}_1$ and $\mathcal{B}_2$. For each parametric class of functions $\mathcal{B}_1$ and $\mathcal{B}_2$, we display the results for $\vartheta=2$ and $\vartheta=15$. The third and fourth columns corresponds respectively to the mean value and the standard deviation of $\widehat{\vartheta}_n$, computed with $300$ different data sets of size $2^{n+1}-1=32767$. The fifth and sixth columns correspond to a $95\%$ confidence interval. The last column corresponds to the value of the step for the discretisation of $\Theta$, which limits the accuracy of the result.}
\label{table:Bconstant}
\end{table}

\section{Proofs} \label{sec: proofs}

\subsection{Proof of Proposition \ref{prop:explicit trans}}
We first prove \eqref{eq:def trans P}.
 By \eqref{def:transBMC}, for any bounded $\psi:\mathcal X^3\rightarrow \R$ and $x \in \mathcal X$, we have
%
\begin{align}\label{eq:P}\nonumber
\mathcal P\psi(x) & =\mathbb{E}\Big[\int_{0}^1 \kappa(z)\int_0^{\infty}\psi\big(x,z\phi_x(t),(1-z)\phi_x(t)\big) B\big(\phi_x(t)\big)e^{-\int_0^t B(\phi_x(s))ds}dtdz\Big] \\
& = \int_\mathbb{R}\int_{0}^{1}\psi\big(x,y_0,\tfrac{1-z}{z}y_0\big)B\big(y_0/z\big)\sigma(y_0/z)^{-2}\mathbb{E}\Big[\int_0^{\infty}e^{-\int_0^t B(\phi_x(s))ds}dL_t^{y_0/z}(\phi_x)\Big]\frac{\kappa(z)}{z}dzdy_0 \\\nonumber
& = \int_{\mathcal{D}} \psi\big(x,y_0,y_1\big)\frac{\kappa\big(y_0/(y_0+y_1)\big)}{y_0+y_1}\frac{B\left(y_0+y_1\right)}{\sigma(y_0+y_1)^{2}}\mathbb{E}\Big[\int_0^{\infty}e^{-\int_0^t B(\phi_x(s))ds}dL_t^{y_0+y_1}(\phi_x)\Big]dy_0dy_1,
\end{align}
where we set $y_1=\frac{1-z}{z}y_0$ in order to obtain the last line and where $(L_t^y(\phi_x))_{t \geq 0}$ is the local time of $\phi_x$ at $y\in\mathcal{X}$. The integral is taken over the domain 
$$\mathcal D = \big\{ (y_0,y_1)\in\mathbb{R}^2,\ \tfrac{\varepsilon}{1-\varepsilon}y_0\leq y_1\leq \tfrac{1-\varepsilon}{\varepsilon}y_0\big\} \subset \supp\big((y_0, y_1) \mapsto \kappa(y_0/(y_0+y_1))\big)$$
therefore the above integral is well defined and the representation \eqref{eq:def trans P} is proved.
We turn to \eqref{eq:def trans Q}. 
From \eqref{eq:P}, we get
\begin{align*}
\mathcal P_0 \varphi(x) =\int_\mathbb{R}\int_{0}^{1}\varphi(y_0)B\big(y_0/z\big)\sigma(y_0/z)^{-2}\mathbb{E}\Big[\int_0^{\infty}e^{-\int_0^t B(\phi_x(s))ds}dL_t^{y_0/z}(\phi_x)\Big]\frac{\kappa(z)}{z}dzdy_0 \end{align*}
and
\begin{align*}
\mathcal P_1 \varphi(x) & =  \int_\mathbb{R}\int_{0}^{1}\varphi\big(\tfrac{1-z}{z}y_0\big)B\big(y_0/z\big)\sigma(y_0/z)^{-2}\mathbb{E}\Big[\int_0^{\infty}e^{-\int_0^t B(\phi_x(s))ds}dL_t^{y_0/z}(\phi_x)\Big]\frac{\kappa(z)}{z}dzdy_0\\
& =\int_\mathbb{R}\int_{0}^{1}\varphi\big(y_1\big)B\big(y_1/\bar{z}\big)\sigma(y_1/\bar{z})^{-2}\mathbb{E}\Big[\int_0^{\infty}e^{-\int_0^t B(\phi_x(s))ds}dL_t^{y_1/\bar{z}}(\phi_x)\Big]\frac{\kappa(1-\bar{z})}{\bar{z}}d\bar{z}dy_1,
\end{align*}
where the second equality is given by two successive changes of variables $y_1 = \tfrac{1-z}{z}y_0$ and $\bar{z} = 1-z$. Finally,
\begin{align*}
\mathcal Q \varphi(x) & = \tfrac{1}{2}\big(\mathcal P_0 \varphi(x) + \mathcal P_1 \varphi(x)\big)\\
& = \int_\mathbb{R}\int_{0}^{1}\varphi(y)B\big(y/z\big)\sigma(y/z)^{-2}\mathbb{E}\Big[\int_0^{\infty}e^{-\int_0^t B(\phi_x(s))ds}dL_t^{y/z}(\phi_x)\Big]\frac{\ktilde(z)}{z}dzdy,
\end{align*}
where $\ktilde(z) = \tfrac{1}{2}(\kappa(z)+\kappa(1-z))$.
Since $\supp(\kappa) \subset [\varepsilon, 1-\varepsilon]$, the above integrals are well defined and \eqref{eq:def trans Q} is established.

\subsection{Proof of Proposition \ref{th:ergodicity}}
The proof goes along a classical path: we establish a drift and a minorisation condition in Proposition \ref{prop:lyapunov} and \ref{prop:minoration} below, and then apply for instance Theorem 1.2. in \cite{hairer2011}, see also the references therein.

\begin{prop}[Drift condition] \label{prop:lyapunov}
Let $V(x)=x^2$. Work under Assumptions \ref{assu:unique}, \ref{assu:debut} and \ref{assu:noyau}. There exist explicitly computable $0 < v_1 = v_1(\varepsilon) < 1$ and $v_2 = v_2(\varepsilon, r_1,r_2,\sigma_1,\sigma_2,b_1)>0$ such that
\begin{align*}
\mathcal QV(x)\leq v_1V(x)+v_2.
\end{align*}
\end{prop}

\begin{prop}[Minorisation condition]\label{prop:minoration} 
Work under Assumption \ref{assu:unique}, \ref{assu:debut} and \ref{assu:noyau}. 
For large enough $w>0$, there exists  $\lambda \in(0,1)$ and a probability measure $\mu$ on $\mathcal{X}$ such that
\begin{align*}
\inf_{\{x, |x|\leq w \}}\mathcal Q(x,\mathcal A)\geq \lambda\mu(\mathcal A)
\end{align*}
for every Borel set $\mathcal A \subset \mathcal X$.
\end{prop}

\begin{proof}[Proof of Proposition \ref{prop:lyapunov}]
Fix $x\in\mathcal{X}$ and let $m(\kappa)=\int_0^1z^2 \ktilde(z)dz$. By It\^o formula, we obtain the decomposition
\begin{align*}
\mathcal QV(x)=\int_0^1\ktilde(z)\int_0^{\infty}\mathbb{E}\Big[z^2\phi_x(t)^2B\big(\phi_x(t)\big)e^{-\int_0^t B(\phi_x(s))ds}\Big]dtdz = m(\kappa)(I+II+III+IV),
\end{align*}
where
\begin{align*}
I&=x^2\int_0^{\infty}\mathbb{E}\Big[B\big(\phi_x(t)\big)e^{-\int_0^t B(\phi_x(s))ds}\Big]dt,\\
II&=2\int_0^{\infty}\mathbb{E}\Big[\int_0^t\phi_x(u)r(\phi_x(u))du \,B\big(\phi_x(t)\big)e^{-\int_0^t B(\phi_x(s))ds}\Big]dt,\\
III&=\int_0^{\infty}\mathbb{E}\Big[\int_0^t\sigma\big(\phi_x(u)\big)^2du\,B\big(\phi_x(t)\big)e^{-\int_0^t B(\phi_x(s))ds}\Big]dt,\\
IV&=2\int_0^{\infty}\mathbb{E}\Big[\int_0^t\phi_x(u)\sigma\big(\phi_x(u)\big)dW_u \, B\big(\phi_x(t)\big)e^{-\int_0^t B(\phi_x(s))ds}\Big]dt.
\end{align*}
First, note that $\int_0^{\infty}B(\phi_x(t))e^{-\int_0^t B(\phi_x(s))ds}dt=1$ holds since $B(\phi_x(t)) \geq b_1 >0$ by Assumption \ref{assu:debut}, therefore $I=x^2$ by Fubini's theorem. We turn to $II$. By Fubini's theorem again:
\begin{align*}
II & =2\,\mathbb{E}\Big[\int_0^{\infty}\phi_x(u)r\big(\phi_x(u)\big)\big(\int_u^{\infty}B\big(\phi_x(t)\big)e^{-\int_0^t B(\phi_x(s))ds}dt\big)du\Big]\\
&= 2\,\mathbb{E}\Big[\int_0^{\infty}\phi_x(u)r\big(\phi_x(u)\big)e^{-\int_0^u B(\phi_x(s))ds}du\Big] \\
& =2\int_{\mathbb{R}}zr(z)\sigma(z)^{-2}\mathbb{E}\Big[\int_0^{\infty}e^{-\int_0^t B(\phi_x(s))ds}dL_t^z(\phi_x)\Big]dz.
\end{align*}
where we used again that $e^{-\int_0^\infty B(\phi_x(s))ds} = 0$ since $B(\phi_x(t)) \geq b_1 >0$ by Assumption \ref{assu:debut} for the second equality and the occupation times formula for the last equality.
By Assumption \ref{assu:unique} we have $zr(z) <0$ for $|z| \geq r_2$, therefore:
\begin{align*}
II & \leq 2\int_{-r_2}^{r_2}zr(z)\sigma(z)^{-2}\mathbb{E}\Big[\int_0^{\infty}e^{-\int_0^t B(\phi_x(s))ds}dL_t^z(\phi_x)\Big]dz \\
& \leq 2r_1r_2(1+r_1)\sigma_1^{-2}\int_{-r_2}^{r_2}\mathbb{E}\Big[\int_0^{\infty}e^{-\int_0^t B(\phi_x(s))ds}dL_t^z(\phi_x)\Big]dz \\
& \leq 2r_1r_2(1+r_1)\sigma_1^{-2}\int_{0}^{\infty}e^{-b_1t}\PP \big(-r_2\leq\phi_x(t)\leq r_2 \big) dt  \leq 2r_1r_2(1+r_1)\sigma_1^{-2}b_1^{-1}.
\end{align*}
using successively Assumption \ref{assu:unique}, \ref{assu:debut} and the occupation times formula.
For the term $III$, by Fubini's theorem, we have
\begin{align*}
III &= \int_0^{\infty}\mathbb{E}\big[\sigma\big(\phi_x(u)\big)^2\int_u^{\infty}B\big(\phi_x(t)\big)e^{-\int_0^t B(\phi_x(s))ds}\Big]du \\
& =\int_0^{\infty}\mathbb{E}\left[\sigma(\phi_x(u))^2e^{-\int_0^u B(\phi_x(s))ds}\right]du
\end{align*}
and this last quantity is less than $\sigma^2_2\int_0^{\infty}e^{-b_1t}dt=\sigma^2_2b_1^{-1}$ by Assumption \ref{assu:unique} and \ref{assu:debut}.
Similarly for the term IV, we have
\begin{align*}
IV& =2\,\mathbb{E}\Big[\int_0^{\infty}\sigma\big(\phi_x(u)\big)\int_u^{\infty}B\big(\phi_x(t)\big)e^{-\int_0^t B(\phi_x(s))ds}dt dW_u\Big] \\
& =2\,\mathbb{E}\Big[\int_0^{\infty}\sigma\big(\phi_x(u)\big)e^{-\int_0^u B(\phi_x(s))ds}dW_u\Big]
\end{align*}
and this last quantity vanishes. Putting the estimates for $I$, $II$, $III$ and $IV$ together, we conclude
\begin{align*}
\mathcal QV(x)\leq m(\kappa)x^2+m(\kappa)(2r_1r_2(1+r_1)\sigma_1^{-2}+\sigma_2^2)b_1^{-1}.
\end{align*}
Since $\supp(\kappa) \subset [\varepsilon, 1-\varepsilon]$, we have $m(\kappa) \leq (1-\varepsilon)^2 < 1$ and this completes the proof with
$v_1 = m(\kappa)$ and $v_2 = m(\kappa)(2r_1r_2(1+r_1)\sigma_1^{-2}+\sigma_2^2)b_1^{-1}$.

\end{proof}

\begin{proof}[Proof of Proposition \ref{prop:minoration}]
\noindent {\it Step 1)}. Let $x\in [-w,w]$ and $\mathcal A \subset \mathcal X$ be a Borel set. Applying Assumption \ref{assu:debut}, introducing the event  $\mathcal W(\phi_x(t)) = \{2w\leq \phi_x(t)\leq 2w\sqrt{(1-\varepsilon)/\varepsilon)}\}$, applying Fubini's theorem and a change of variable, we successively obtain 
\begin{align*}
\mathcal Q(x,\mathcal A) & =\int_0^1\ktilde(z)\int_{0}^{\infty} \mathbb{E}\Big[\mathbf{1}_{\left\lbrace z\phi_x(t)\in \mathcal A\right\rbrace}B\big(\phi_x(t)\big)e^{-\int_0^t B(\phi_x(s))ds}\Big]dt dz \\
& \geq b_1 \int_0^1\ktilde(z)\int_{0}^{\infty} \mathbb{E}\Big[
{\bf 1}_{\mathcal W(\phi_x(t))}\mathbf{1}_{\left\lbrace z\phi_x(t)\in \mathcal A\right\rbrace}e^{-\int_0^t B\big(\phi_x(s)\big)ds}\Big]dtdz\\
& = b_1 \int_{0}^{\infty} \mathbb{E}\Big[
{\bf 1}_{\mathcal W(\phi_x(t))}e^{-\int_0^t B\left(\phi_x(s)\right)ds}\int_0^{\phi_x(t)}\mathbf{1}_{\left\lbrace y\in \mathcal A\right\rbrace}\ktilde \big(y/\phi_x(t) \big) \phi_x(t)^{-1}dy\Big]dt.
\end{align*}
Using again Fubini's theorem, we get
\begin{align*}
\mathcal Q(x,\mathcal A) \geq \int_{\mathbb{R}}\mathbf{1}_{\left\lbrace y\in \mathcal A\right\rbrace}f(x,y)dy,
\end{align*}
with
$
f(x,y)=b_1 \int_{0}^{\infty} \mathbb{E}\big[
{\bf 1}_{\mathcal W(\phi_x(t))}\mathbf{1}_{\left\lbrace y\leq \phi_x(t)\right\rbrace}e^{-\int_0^t B\left(\phi_x(s)\right)ds}\ktilde \big(y/\phi_x(t) \big)\phi_x(t)^{-1}\big]dt.
$\\

\noindent {\it Step 2).} We now prove that $f$ is bounded below independently of $x$. By Assumption \ref{assu:noyau}, $\ktilde\big(y/\phi_x(t) \big)>\delta$ for all $y\in [\varepsilon\phi_x(t),(1-\varepsilon)\phi_x(t)]$ so that
\begin{align*}
f(x,y) & \geq \delta b_1\int_{0}^{\infty} \mathbb{E}\Big[{\bf 1}_{\mathcal W(\phi_x(t))}\mathbf{1}_{\left\lbrace \varepsilon\phi_x(t)\leq y\leq (1-\varepsilon)\phi_x(t)\right\rbrace} e^{-\int_0^t B\left(\phi_x(s)\right)ds}\phi_x(t)^{-1}\Big]dt.
\end{align*}
Next, as $\mathcal W(\phi_x(t))\bigcap\lbrace \varepsilon\phi_x(t)\leq y\leq (1-\varepsilon)\phi_x(t)\rbrace\supseteq \mathcal W(\phi_x(t))\bigcap\lbrace 2w\sqrt{\varepsilon(1-\varepsilon)}\leq y\leq 2w(1-\varepsilon)\rbrace $, we get
\begin{align*}
f(x,y) \geq \frac{\delta b_1}{2w\sqrt{(1-\varepsilon)/\varepsilon}}\mathbf{1}_{\left\lbrace 2w\sqrt{\varepsilon(1-\varepsilon)}\leq y\leq 2w(1-\varepsilon)\right\rbrace}h(x,y),
\end{align*}
where
$h(x,y)=\int_{0}^{\infty} \mathbb{E}\big[
{\bf 1}_{\mathcal W(\phi_x(t))}e^{-\int_0^t B\left(\phi_x(s)\right)ds}\big]dt
$.
Let $\Delta=(1+\sqrt{(1-\varepsilon)/\varepsilon})w$ denote the mid-point of the interval $[2w,2w\sqrt{(1-\varepsilon)/\varepsilon}]$. 
Let also $T^{x}_y=\inf\left\lbrace t>0,\phi_x(t)\geq y\right\rbrace$ denote the exit time of the interval $(\inf \mathcal X ,y)$ by  $(\phi_x(t))_{t \geq 0}$. 
It follows that
\begin{align*}
h(x,y)&\geq\mathbb{E}\Big[\int_{T_{\Delta}^{x}}^{\infty} 
{\bf 1}_{\mathcal W(\phi_x(t))}e^{-\int_0^t B\left(\phi_x(s)\right)ds}dt\Big] \\
&= \int_{0}^{\infty} \mathbb{E}\Big[\mathbf{1}_{\mathcal W(\phi_x(t+T_{\Delta}^x))}e^{-\int_0^{t+T_{\Delta}^{x}} B\left(\phi_x(s)\right)ds}\Big]dt\\
&= \int_{0}^{\infty} \mathbb{E}\Big[\mathbf{1}_{\mathcal W(\phi_x(t+T_{\Delta}^x))}e^{-\int_0^{T_{\Delta}^{x}} B\left(\phi_x(s)\right)ds}e^{-\int_{0} ^{t} B\left(\phi_x(s+T_{\Delta}^{x})\right)ds}\Big]dt \\
&\geq\int_{0}^{\infty} \mathbb{E}\Big[\mathbf{1}_{\mathcal W(\phi_x(t+T_{\Delta}^x))}e^{-T_{\Delta}^{x} \left(b_2\Delta^{\gamma}+b_1\right)}e^{-\int_{0} ^{t} B\left(\phi_x(s+T_{\Delta}^{x})\right)ds}\Big]dt,
\end{align*}
by Assumption \ref{assu:debut} and because $\phi_x(s)\leq \Delta$ for every $s\leq T_{\Delta}^x$. Applying the strong Markov property, we further obtain
\begin{equation} \label{eq:mino apres Markov}
h(x,y)\geq \int_{0}^{\infty}\mathbb{E}\Big[\mathbf{1}_{\mathcal W(\phi_{\Delta}(t))}e^{-T_{\Delta}^{x} \left(b_2\Delta^{\gamma}+b_1\right)}e^{-\int_{0} ^{t} B\left(\phi_{\Delta}(s)\right)ds}\Big]dt
\end{equation}
since $\phi_{x}(T_{\Delta}^{x})=\Delta$ for $x\leq w <\Delta$. Introduce next $T_{\mathcal W}^\Delta=\inf \{t\geq 0,\ \phi_{\Delta}(t)\notin [2w,2w\sqrt{(1-\varepsilon)/\varepsilon}]\}$, {\it i.e.} the exit time of $[2w,2w\sqrt{(1-\varepsilon)/\varepsilon}]$ by $(\phi_\Delta(t))_{t \geq 0}$. By \eqref{eq:mino apres Markov} and Assumption \ref{assu:debut} again, it follows that 
\begin{align*}
h(x,y)&\geq \mathbb{E}\Big[\int_{0}^{T_{\mathcal W}^\Delta}\mathbf{1}_{\mathcal W(\phi_{\Delta}(t))}e^{-T_{\Delta}^{x} \left(b_2\Delta^{\gamma}+b_1\right)}e^{-\int_0^t\left(b_2|\phi_x(s)|^{\gamma}+b_1\right)ds}dt\Big] \\
&\geq \mathbb{E}\left[\int_{0}^{T_{\mathcal W}^\Delta}\mathbf{1}_{\mathcal W(\phi_{\Delta}(t))}e^{-T_{\Delta}^{x} \left(b_2\Delta^{\gamma}+b_1\right)}e^{-tv_3}dt\right],
\end{align*}
using that $\phi_x(t)\leq 2w\sqrt{(1-\varepsilon)/\varepsilon}$ for $t\leq T_{\mathcal W}^\Delta$ and where $v_3 = b_2(2w\sqrt{(1-\varepsilon)/\varepsilon})^{\gamma}+b_1$.
Since $\Delta>-w$, the event $\{T_{\Delta}^{x}\leq T_{\Delta}^{-w}\}$ holds almost-surely for every $x\in [-w,w]$ and therefore
\begin{equation} \label{eq:last mino h}
h(x,y)  \geq \mathbb{E}\Big[\int_{0}^{T_{\mathcal W}^\Delta}e^{-T_{\Delta}^{-w}\left(b_2\Delta^{\gamma}+b_1\right)}e^{-tv_3}dt\Big] 
 \geq v_3^{-1} \mathbb{E}\big[e^{-T_{\Delta}^{-w}\left(b_2\Delta^{\gamma}+b_1\right)}\big]\mathbb{E}\big[1-e^{-T_{\mathcal W}^\Delta v_3}\big].
\end{equation}
by the independence of $T_{\mathcal W}^\Delta$ and $T_{\Delta}^{-w}$.
Furthermore, for every $a,x\in\mathcal{X}$ with $x<a$, we have
\begin{align*}
\mathbb{P}\left(T_a^{x}<\infty\right) & =\lim_{u\rightarrow -\infty}\frac{s(u)-s(x)}{s(u)-s(a)},
\end{align*}
where $s(x)=\int_{\inf \mathcal X}^x\exp(-2\int_{\inf \mathcal X}^yr(z)\sigma(z)^{-2}dz)dy$, is the scale function associated to $(\phi_x(t))_{t \geq 0}$.
By the classical Feller classification of scalar diffusions (see {\it e.g.} Revuz and Yor \cite{RY}), we have the equivalence 
$\mathbb{P}\left(T_a^{x}<\infty\right)=0$ if only if $\int_{\inf \mathcal X}s(x)dx=0$
but that latter property contradicts Assumption \ref{assu:unique}. Therefore, there exist $w_1,\delta_1>0$ such that
$\mathbb{P}(T_{\Delta}^{-w}\leq w_1)>\delta_1$. It follows that
\begin{equation} \label{eq:mino first term}
\mathbb{E}\big[e^{-T_{\Delta}^{-w}(b_2\Delta^{\gamma}+b_1)}\big]\geq \mathbb{E}\big[e^{-w_1(b_2\Delta^{\gamma}+b_1)}\mathbf{1}_{\{T_{\Delta}^{-w}\leq w_1\}}\big]\geq \delta_1 e^{-w_1(b_2\Delta^{\gamma}+b_1)}
\end{equation}
and since $T_{\mathcal W}^\Delta>0$ almost surely, there exists $\delta_2>0$, independent of $x$, such that 
\begin{equation} \label{eq:mino second term}
\mathbb{E}\big[1-e^{-T_{\mathcal W}^\Delta v_3}\big]>\delta_2.
\end{equation}
Back to \eqref{eq:last mino h}, putting together \eqref{eq:mino first term} and \eqref{eq:mino second term}, we obtain
$$
h(x,y) \geq v_3^{-1} \delta_1 e^{-w_1(b_2\Delta^{\gamma}+b_1)} \delta_2
$$
and eventually
$$f(x,y) \geq  \frac{\delta b_1 \delta_1\delta_2}{2w\sqrt{(1-\varepsilon)/\varepsilon}v_3}\mathbf{1}_{\left\lbrace 2w\sqrt{\varepsilon(1-\varepsilon)}\leq y\leq 2w(1-\varepsilon)\right\rbrace} e^{-w_1(b_2\Delta^{\gamma}+b_1)}.
$$
\noindent {\it Step 3)}. Define the probability measure $\mu(dy)=f(y)dy$ on $\mathcal X$ by
\begin{align*}
f(y)=\frac{1}{2w\big(1-\varepsilon-\sqrt{\varepsilon(1-\varepsilon)}\big)}\mathbf{1}_{\left\lbrace 2w\sqrt{\varepsilon(1-\varepsilon)}\leq y\leq 2w(1-\varepsilon)\right\rbrace},
\end{align*}
and let $\lambda=(1-\varepsilon-\sqrt{\varepsilon(1-\varepsilon)})\frac{b_1\delta\delta_1\delta_2}{\sqrt{(1-\varepsilon)/\varepsilon}v_3}e^{-w_1\left(b_2\Delta^{\gamma}+b_1\right)}$. We may assume that $0 < \lambda < 1$ (the lower bound remains valid if we replace $\delta$ by $\delta'<\delta$ for instance)
and we thus have established
\begin{align*}
\mathcal Q(x,\mathcal A)\geq \lambda \mu(\mathcal A),
\end{align*}
for an arbitrary Borel set $\mathcal A \subset \mathcal X$. The proof of Proposition \ref{prop:minoration} is complete.
\end{proof}

\subsection{Proof of Proposition \ref{thm:convl2}}

\subsubsection*{Preparations}

We first state a useful estimate on the local time of $L_t^y(\phi_x)$ as $t \rightarrow \infty$. Its proof is delayed until Appendix \ref{proof of lem loc time}. 
\begin{lem}\label{lem:borne_localtime}
Work under Assumption \ref{assu:unique}. For every compact $\mathcal K \subset \mathcal{X}$ and for every $t\geq 0$, we have
\begin{align*}
\sup_{x \in \mathcal X, y \in \mathcal K}\mathbb{E}\left[L_t^y\left(\phi_x\right)\right] \lesssim 
1+t^{3/2},
\end{align*}
up to a constant that only depends on $r_1$, $r_2$ and $\sigma_2$. In particular, for every $c>0$, the function
\begin{align*}
y \mapsto \int_0^{\infty}e^{-ct}\sup_{x \in \mathcal X}\mathbb{E}\left[L_t^y\left(\phi_x\right)\right] dt
\end{align*}
is well-defined and locally bounded, uniformly over $\mathbb Q$.
\end{lem}

Lemma \ref{lem:borne_localtime} enables us to obtain estimates on the action of $\mathcal P$ and $\mathcal Q$ on functions $\psi:\mathcal X \times \mathcal X \rightarrow \R$ with nice behaviours that will prove essential for obtaining Proposition \ref{thm:convl2}. We set
$$
\mathcal Q\psi(x)=\int_{\mathcal{X}}\psi(x,y)q(x,y)dy,\;\;\mathcal P(\psi\otimes \psi)(x)=\int_{\mathcal{X}\times \mathcal X} \psi(x,y_1)\psi(x,y_2)p(x,y_1,y_2)dy_1dy_2,
$$
where $p(x,y_1,y_2)$ and $q(x,y)$ are given in Proposition \ref{prop:explicit trans}. 
\begin{lem}\label{lem: estimates action transition}
Work under Assumptions \ref{assu:unique}, \ref{assu:debut} and \ref{assu:noyau}. Let $\psi:\mathcal X \times \mathcal X \rightarrow \R$ be bounded and such that $\psi_\star$ has compact support. There exists a constant $c_{\mathrm{supp}(\psi_\star)}$ depending on $\mathrm{supp}(\psi_\star)$ (and $\mathcal Q$) such that, for $i = 0,1$:
\begin{equation} \label{eq:action simple P}
|\mathcal P_i\psi(x)|\leq c_{\mathrm{supp}(\psi_\star)}\int_{\mathcal{X}} |\psi(x,y)|dy
\end{equation}
and
\begin{align*} 
|\mathcal P(\psi \otimes \psi)(x)|\leq c_{\mathrm{supp}(\psi_\star)}
\psi^\star(x) \int_{\mathcal X}|\psi(x,y)|dy.
\end{align*}
\end{lem}
Note in particular that \eqref{eq:action simple P} implies in turn the estimates 
\begin{equation}\label{eq:action simple}
|\mathcal Q\psi(x)|\leq c_{\mathrm{supp}(\psi_\star)}\int_{\mathcal{X}} |\psi(x,y)|dy,
\end{equation}
and, for $i = 0,1$:
$$|\mathcal Q\mathcal P_i\psi(x)|\leq c_{\mathrm{supp}(\psi_\star)}^2|\psi|_1\wedge c_{\mathrm{supp}(\psi_\star)}\big|\int_{\mathcal X}\psi_{\star}(y)dy\big| \lesssim |\psi|_{\wedge 1}.$$ 
\begin{proof}
By Assumption \ref{assu:debut}, we have 
\begin{align*}
|\mathcal P_0\psi(x)|&=\Big|\int_{0}^{1}\kappa(z)\int_0^{\infty}\mathbb{E}\big[\psi(x,z\phi_x(t))B\big(\phi_x(t)\big)e^{-\int_0^t B(\phi_x(s))ds}\big]dtdz\Big|\\
&\leq \int_{0}^{1}\kappa(z)\int_0^{\infty}\mathbb{E}\big[\big|\psi(x,z\phi_x(t))\big|\big(b_2\left|\phi_x(t)\right|^{\gamma}+b_1\big)e^{-b_1t}\big]dtdz\\
&\leq \int_{0}^{1}\kappa(z)\int_0^{\infty}\mathbb{E}\Big[\big|\psi(x,z\phi_x(t))\big|\big(b_2\left|\phi_x(t)\right|^{\gamma}+b_1\big)\int_t^{\infty}b_1e^{-b_1s}ds\Big]dtdz.\\
\end{align*}
Next, by Fubini's theorem and the occupation times formula, we derive
\begin{align*}
|\mathcal P_0\psi(x)|
&\leq \int_{0}^{1}\kappa(z)\int_0^{\infty}b_1e^{-b_1s}\int_0^{s}\mathbb{E}\big[|\psi(x,z\phi_x(t))|(b_2|\phi_x(t)|^{\gamma}+b_1)\big]dtdsdz\\
& \leq \int_0^{1}\kappa(z)\int_0^{\infty}b_1e^{-b_1s}\int_{\mathbb{R}}|\psi(x,zy)|(b_2|y|^{\gamma}+b_1)\sigma(y)^{-2}\E[L_s^y(\phi_x)]dydsdz \\
& =  \int_0^{1}\kappa(z)z^{-1}\int_0^{\infty}b_1e^{-b_1s}\int_{\mathbb{R}}|\psi(x,y)|(b_2|y/z|^{\gamma}+b_1)\sigma(y/z)^{-2}\E[L_s^{y/z}(\phi_x)]dydsdz \\
& \leq c_{\mathrm{supp}(\psi_\star)}\int_{\mathbb{R}}|\psi(x,y)|dy,
\end{align*}
and \eqref{eq:action simple P} is proved for $i = 0$ with
$$c_{\mathrm{supp}(\psi_\star)} = \sup_{y \in \supp(\psi_\star), x \in \mathcal X, z \in [\varepsilon, 1-\varepsilon]}b_1(b_2|y/z|^{\gamma}+b_1)\sigma(y/z)^{-2}z^{-1}\int_0^{\infty}e^{-b_1s}\E[L_s^{y/z}(\phi_x)]ds.$$
This last quantity is finite by Lemma \ref{lem:borne_localtime}. Following the same steps as for $\mathcal P_1$ we get \eqref{eq:action simple P} for $i = 1$.
For the second estimate, we have
\begin{align*}
|\mathcal P(\psi \otimes \psi)(x)| & \leq \int_{0}^{1}\kappa(z)\int_0^{\infty}\mathbb{E}\big[\big|\psi\big(x,z\phi_x(t)\big)\psi\big(x,(1-z)\phi_x(t)\big)\big|B(\phi_x(t))e^{-\int_0^t B(\phi_x(s))ds}\big]dtdz\\
&\leq \big|\mathcal P_0 \psi (x)\big|\sup_{y}\big|\psi (x,y)\big|
\end{align*}
and we conclude by applying  \eqref{eq:action simple}.
\end{proof}

\subsubsection*{Completion of proof of Proposition \ref{thm:convl2}}
Without loss of generality, we may assume that $\nu(\mathcal Q\psi)=0$, the general case being obtained by considering the function $\psi(x,y)-\nu(\mathcal Q\psi)$. Of course, the compact support property is lost by adding a constant and one has to be careful when revisiting the estimates of {\it Step 2)} to {\it Step 4)} below. They exhibit additional error terms that all have the right order using Lemma \ref{lem: estimates action transition} and 
the fact that $\mathcal P{\bf 1} = \mathcal Q{\bf 1}={\bf 1}$.\\ 

By (ii) of Definition \ref{def incomplete tree} we may (and will) assume that for some $0 \leq \varrho \leq 1$ and some positive constants $c_1, c_2$ that do not depend on $n$, we have
$$ c_1 2^{\varrho m} \leq \big|  \mathbb G_m \cap  \mathbb U_n\big| \leq c_2 2^{\varrho m}\;\;\text{for every}\;\;m=1,\ldots, n.$$
We first consider the case $\varrho>0$. The case $\varrho =0$ requires a slightly different method and will be handled in a second phase.\\

\noindent {\it Step 1)}. We start with a standard preliminary decomposition, see for instance \cite{Guillin2, Guillin1}, expanding the sum in $u \in \mathbb U_n^\star$. 

 We have
\begin{align*}
\E_{\mu}[\mathcal M_{\mathbb U_n}(\psi)^2] & =|\mathbb{U}_n^\star|^{-2}\mathbb{E}_{\mu}\big[\big(\sum_{m=1}^n\sum_{u\in\mathbb{G}_m \cap \mathbb U_n}\psi(X_{u^-},X_u)\big)^2\big] \\
& \leq |\mathbb{U}_n^\star|^{-2}\Big(\sum_{m=1}^n\Big(\mathbb{E}_{\mu}\Big[\big(\sum_{u\in\mathbb{G}_m \cap \mathbb U_n}\psi(X_{u^-},X_u)\big)^2\Big]\Big)^{1/2}\Big)^2
\end{align*}
by triangle inequality. Thus Proposition \ref{thm:convl2} amounts to control
$$\mathbb{E}_{\mu}\big[\big(\sum_{u\in\mathbb{G}_m \cap \mathbb U_n}\psi(X_{u^-},X_u)\big)^2\big]=I_m+II_m,$$
with
\begin{align*}
I_m  & = \mathbb{E}_{\mu}\big[\sum_{u\in\mathbb{G}_m \cap \mathbb U_n}\psi(X_{u^-},X_u)^2\big],\\
II_m  & = \mathbb{E}_{\mu}\big[\sum_{u,v\in\mathbb{G}_m \cap \mathbb U_n, u\neq v}\psi(X_{u^-},X_u)\psi(X_{v^-},X_v)\big],
\end{align*}
and the convention $\sum_{\emptyset}=0$. 

\noindent {\it Step 2).} The control of the term $I_m$ is straightforward: by Lemma \ref{lem: estimates action transition} we have
\begin{align*}  
I_1 &= |\mathbb G_1 \cap \mathbb U_1|\,\mu(\mathcal Q\psi^2)\leq c_22^{\varrho}c_{\mathrm{supp}(\psi_\star)}\int_{\mathcal X\times \mathcal X}\psi(x,y)^2\mu(dx)dy\;\;\text{for}\;\;m=1, \\
I_m & =  |\mathbb G_m \cap \mathbb U_m|\, \mu(\mathcal Q^m\psi^2) \leq c_22^{\varrho m} (c_{\mathrm{supp}(\psi_\star)}^2|\psi^2|_1\wedge c_{\mathrm{supp}(\psi_\star)}|\psi_\star^2|_1)\;\;\text{for}\;\;m\geq 2,
\end{align*}
therefore $I_m \lesssim 2^{\varrho m}|\psi^2|_\mu$ holds for every $m\geq 1$. In the case $\nu(\mathcal Q\psi)\neq 0$, we replace $|\psi^2|_\mu$ by $|\psi^2|_{\mu+\nu}$.\\

\noindent {\it Step 3).} We further decompose the main term $II_m=III_m+IV_m$, having
\begin{align*}
III_m & =\mathbb{E}_{\mu}\big[\sum_{w\in\mathbb{G}_{m-1} \cap \mathbb U_n, (w0,w1) \in \mathbb U_n^2}\psi(X_{w},X_{w0})\psi(X_{w},X_{w1})\big], \\
IV_m & = \mathbb{E}_{\mu}\big[\sum_{u,v\in\mathbb{G}_{m-1} \cap \mathbb U_n, u \neq v}\hspace{2mm}\sum_{ui,vj \in \mathbb U_n, i,j=0,1}\psi(X_{u},X_{ui})\psi(X_{v},X_{vj})\big].
\end{align*}
The control of $III_m$ is straightforward: 
$$III_m = 
\E_\mu\big[\sum_{w\in\mathbb{G}_{m-1} \cap \mathbb U_n}\mathcal P( \psi \otimes  \psi)(X_{w})\big]=|\mathbb G_{m-1} \cap \mathbb U_n|\, \mu\big(\mathcal Q^{m-1}\mathcal P (\psi \otimes \psi)\big).$$
In the same way as for the term $I_m$, by Lemma \ref{lem: estimates action transition}, one readily checks that 
$|III_m| \lesssim 2^{\varrho(m-1)}|\psi^\star\psi|_{\mu}$.\\

\noindent {\it Step 4).} We now turn to the main term $IV_m$.
Writing here $u \wedge v$ for the most common recent ancestor of $u$ and $v$, conditioning w.r.t. 
$\mathcal{F}_{|u\wedge v|+1}$ and using the conditional independence of $(X_{u},X_{ui})$ and $(X_{v},X_{vj})$ given $\mathcal{F}_{|u\wedge v|+1}$ thanks to the BMC property \eqref{def:transBMC}, we successively obtain 
\begin{align*}
IV_m & = \mathbb{E}_{\mu}\Big[\sum_{u, v\in\mathbb{G}_{m-1} \cap \mathbb U_n, u \neq v}\hspace{2mm}\sum_{ui,vj \in \mathbb U_n, i,j=0,1}\mathbb{E}_{\mu}\left[\psi(X_{u},X_{ui})\psi(X_{v},X_{vj})\big|\mathcal{F}_{|u\wedge v|+1}\right]\Big]\\
&=\mathbb{E}_{\mu}\Big[\sum_{u, v\in\mathbb{G}_{m-1} \cap \mathbb U_n, u \neq v}\hspace{2mm}\sum_{ui,vj \in \mathbb U_n, i,j=0,1}\mathbb{E}_{\mu}\left[\psi(X_{u},X_{ui})\big|\mathcal{F}_{|u\wedge v|+1}\right]\mathbb{E}_{\mu}\left[\psi(X_{v},X_{vj})\big|\mathcal{F}_{|u\wedge v|+1}\right]\Big] \\
&=\mathbb{E}_{\mu}\Big[\sum_{u, v\in\mathbb{G}_{m-1} \cap \mathbb U_n, u \neq v}\hspace{2mm}\sum_{ui,vj \in \mathbb U_n, i,j=0,1}\mathbb{E}_{\mu}\left[\mathcal P_{i}\psi(X_{u})\big|\mathcal{F}_{|u\wedge v|+1}\right]\mathbb{E}_{\mu}\left[\mathcal P_j\psi(X_{v})\big|\mathcal{F}_{|u\wedge v|+1}\right]\Big]\\
&= \mathbb{E}_{\mu}\Big[\sum_{u, v\in\mathbb{G}_{m-1} \cap \mathbb U_n, u \neq v}\hspace{2mm} \sum_{ui,vj \in \mathbb U_n, i,j=0,1}\mathcal Q^{m-2-|u\wedge v|}\mathcal P_i\psi(X_{u^*})\mathcal Q^{m-2-|u\wedge v|}\mathcal P_j\psi(X_{v^*})\Big], 
\end{align*} 
where $u^*$ (respectively $v^*$) is the descendant of $u\wedge v$ which is an ancestor of $u$ (respectively $v$).
Conditioning further w.r.t. $\mathcal{F}_{\left|u\wedge v\right|}$ we obtain
\begin{align*}
IV_m &=\mathbb{E}_{\mu}\Big[\sum_{u, v\in\mathbb{G}_{m-1} \cap \mathbb U_n, u \neq v}\hspace{2mm} \sum_{ui,vj \in \mathbb U_n, i,j=0,1}\mathcal P\big(\mathcal Q^{m-2-|u\wedge v|}\mathcal P_i\psi\otimes \mathcal Q^{m-2-|u\wedge v|}\mathcal P_j\psi\big)(X_{u\wedge v})\Big]\\
&=\sum_{u, v\in\mathbb{G}_{m-1} \cap \mathbb U_n, u \neq v}\hspace{2mm} \sum_{ui,vj \in \mathbb U_n, i,j=0,1}\mu \big(\mathcal Q^{|u\wedge v|}\mathcal P(\mathcal Q^{m-2-|u\wedge v|}\mathcal P_i\psi\otimes \mathcal Q^{m-2-|u\wedge v|}\mathcal P_j\psi)\big) \\
&=\sum_{l=1}^{m-1}\sum_{u, v\in\mathbb{G}_{m-1} \cap \mathbb U_n, |u\wedge v|=m-l-1}\hspace{2mm} \sum_{ui,vj \in \mathbb U_n, i,j=0,1}\,\mu\big(\mathcal Q^{m-l-1}\mathcal P(\mathcal Q^{l-1}\mathcal P_i\psi\otimes \mathcal Q^{l-1}\mathcal P_j\psi)\big),
\end{align*}
obtaining the last term by rearranging the sum $u,v\in\mathbb{G}_{m-1}$ that expands over indexes $|u|-|u\wedge v|=m-1-|u\wedge v|$ that vary from $1$ to $m-1$. 
By Lemma \ref{lem: estimates action transition} and Proposition \ref{th:ergodicity} one obtains the following estimates for all $i,j = 0,1$: 
\begin{align*}
|\mu \left(\mathcal Q^{m-2}\mathcal P\left(\mathcal P_i\psi\otimes \mathcal P_j\psi\right)\right)|\lesssim |\psi_\star|_1|\psi|_{\wedge 1}\;\;\text{for}\;\;l=1,
\end{align*}
and for $l\geq 2$:
\begin{align*}
|\mu \big(\mathcal Q^{m-l-1}\mathcal P(\mathcal Q^{l-1}\mathcal P_i\psi\otimes \mathcal Q^{l-1}\mathcal P_j\psi)\big)| \lesssim 
|\psi|_{\wedge 1}^2 \,\wedge  \,\rho^{2(l-1)}|\psi_\star|_{1}^2\,\mu\big(\mathcal Q^{m-l-1}\mathcal P((1+V)\otimes(1+V))\big). 
\end{align*}
In the case $\nu(\mathcal Q\psi)\neq 0$, we replace $|\psi|_{\wedge 1}$ by $|\psi|_{\nu}$. We claim that
\begin{equation} \label{eq: bornitude V}
\mu\big(\mathcal Q^{m-l-1}\mathcal P((1+V)\otimes(1+V))\big) \lesssim 1+\mu(V^2)
\end{equation}
and postpone the proof of \eqref{eq: bornitude V} to Step 6 below. Notice also that for $l=1,\ldots, m-1$,
\begin{align*}
\big|\{ u\neq v\in\mathbb{G}_{m-1} \cap \mathbb U_n, |u\wedge v|=m-l-1 \}\big| \lesssim  2^{\varrho m} \times 2^{l},
\end{align*}
where $2^{\varrho m}$ is an upper bound for the number of choices for $u$ (the first descendant in generation ${m-1}$ of the ancestor from generation $m-l-1$) and $2^{l}$ is the (order of the) number of choices of $v$ (the second descendant, satisfying $u\wedge v\in\mathbb{G}_{m-l-1}$) in the worst case where there is a full subtree with $l$ generations. It follows that for any $p \geq 1$:
\begin{align*}
|IV_m|& \lesssim 2^{\varrho m}|\psi_\star|_1|\psi|_{\wedge 1}+ 2^{\varrho m}\sum_{l=1}^{\infty}2^{l}\big( |\psi|_{\wedge 1}^2 \,\wedge  \,\rho^{2(l-1)}|\psi_\star|_1^2\big(1+\mu(V^2)\big)\big) \\
&\leq 2^{\varrho m} |\psi_\star|_1|\psi|_{\wedge 1}+ 2^{\varrho m} \Big( |\psi|_{\wedge 1}^2 \sum_{l=1}^{p}2^{l} + |\psi_\star|_1^2\big(1+\mu(V^2)\big)\sum_{l=p+1}^\infty2^{- l}(2\rho)^{2l}\Big)\\
& \lesssim 2^{\varrho m} |\psi_\star|_1|\psi|_{\wedge 1}+ 2^{\varrho m} \Big( |\psi|_{\wedge 1}^2 2^{p} + |\psi_\star|_1^2\big(1+\mu(V^2)\big)2^{- p}\Big)
\end{align*}
where we crucially used the fact that $\rho \leq 2^{-1}$. Then, taking the infimum over all $p\geq 1$, we get
\begin{align*}
|IV_m|\lesssim 2^{\varrho m}|\psi_\star|_1|\psi|_{\wedge 1}+ 2^{\varrho m}|\psi_\star|_1|\psi|_{\wedge 1}\big(1+\mu(V^2)\big) \inf_{p \geq 1}\Big(\frac{|\psi|_{\wedge 1}}{|\psi_\star|_1} 2^{p} + \frac{|\psi_\star|_1}{|\psi|_{\wedge 1}}2^{-p}\Big).
\end{align*}
As $|\psi|_{\wedge 1}\leq |\psi_\star|_1$, we get that $\inf_{p \geq 1}\Big(\frac{|\psi|_{\wedge 1}}{|\psi_\star|_1} 2^{p} + \frac{|\psi_\star|_1}{|\psi|_{\wedge 1}}2^{-p}\Big)\lesssim 1$ and
\begin{align*}
|IV_m|\lesssim 2^{\rho m}\big(1+\mu(V^2)\big)|\psi_\star|_1|\psi|_{\wedge 1}.
\end{align*}
\noindent {\it Step 5).} Putting together the estimates obtained for $I_m$ in Step 2, $III_m$ in Step 3 and $IV_m$ in Step 4, and recalling $II_m=III_m+IV_m$ we eventually derive:
$$\mathbb{E}_{\mu}\big[\big(\sum_{u\in\mathbb{G}_m \cap \mathbb U_n}\psi(X_{u^-},X_u)\big)^2\big] \lesssim 2^{\varrho m}\big(|\psi^2|_\mu+|\psi^\star\psi|_{\mu}+ \big(1+\mu(V^2)\big)|\psi_\star|_1|\psi|_{\wedge 1}\big).$$
In the case $\nu(\mathcal Q\psi)\neq 0$, we replace $|\psi^2|_\mu$ by $|\psi^2|_{\mu+\nu}$ and $|\psi|_{\wedge 1}$ by $|\psi|_{\nu}$ as follows from Step 2 and 4. Taking square root, summing in $1 \leq m \leq n$, taking square again and normalising by $|\mathbb U_n^{\star}|^{-2}$ which is of order $2^{-2\varrho n}$, we obtain Proposition \ref{thm:convl2}.\\

\noindent {\it Step 6).} It remains to establish \eqref{eq: bornitude V}. We only sketch the argument which is similar to the proof of Proposition \ref{prop:lyapunov}. First, one obtains 
$$\mathcal P((1+V)\otimes(1+V)) \lesssim 1+\mathcal QV(x) + \mathcal Q V^2(x),$$
and it follows that
\begin{align*}
\mu(\mathcal Q^{m-l-1}\mathcal P\big((1+V)\otimes(1+V))\big) & \lesssim 1+\mu(\mathcal Q^{m-l}V)+\mu(\mathcal Q^{m-l}V^2) \\
&\lesssim 1+\mu(V)+\mu(\mathcal Q^{m-l}V^2)
\end{align*}
by Proposition \ref{prop:lyapunov}. Applying It\^o's formula and using Assumptions \ref{assu:unique} and \ref{assu:debut}
on can check that 
\begin{align*}
\mathcal QV^2(x)\lesssim 1+V(x)^2+\mathcal QV(x) \lesssim 1+V(x)^2
\end{align*}
by Proposition \ref{prop:lyapunov} again. We obtain \eqref{eq: bornitude V} by integrating w.r.t. $\mu$.
Finally
the case $\varrho = 0$ has to be treated separately mainly for notational reason, the proof following the same line as in the case $\varrho>0$. We delay it until Appendix \ref{appendix:varrho=0}. 

\subsection{Proof of Theorem \ref{thm:nonparametrique}}


\subsubsection*{Preparations} We first establish local estimates on the invariant density $\nu$.

\begin{lem} \label{lem: uplowbound densite}
Work under Assumptions \ref{assu:unique}, \ref{assu:debut} and \ref{assu:noyau}. Let $\mathcal Q \in \mathbb Q$ and let $\nu$ be the associated invariant density of Proposition \ref{th:ergodicity}. Let $x_0 \in \mathcal X$. There exist positive constants  $c_i = c_i(x_0,\mathcal Q)$ and a bounded neighbourhood $\mathcal V_{x_0}$ with non-empty interior such that
$$0 < c_1 \leq \inf_{x \in \mathcal V_{x_0}}\nu(x) \leq  \sup_{x \in \mathcal V_{x_0}}\nu(x) \leq c_2.$$
Moreover, $0 < \inf_{\mathcal Q \in \mathbb Q} c_1(x_0,\mathcal Q) \leq \sup_{\mathcal Q \in \mathbb Q} c_2(x_0,\mathcal Q)<\infty.$
\end{lem}
\begin{proof}
Let $\mathcal V_{x_0}=[a,b] \subset \mathcal X$ be a bounded neighbourhood of $x_0$ and
$$\mathcal{V}_{x_0}^{\varepsilon}=\big[a/(1-\varepsilon)\wedge a/\varepsilon,\ b/\varepsilon\vee b/(1-\varepsilon)\big].$$
Let $x \in \mathcal V_{x_0}$. By Proposition \ref{prop:explicit trans}, using Assumptions \ref{assu:unique} and \ref{assu:debut}, we obtain
\begin{align*}
\nu(x) & = \int_{\mathcal X}\nu(y)q(y,x)dy \\
& = \int_{\mathcal X}\nu(y)\int_0^1\frac{\ktilde(z)}{z}B(x/z)\sigma(x/z)^{-2}\mathbb{E}\Big[\int_0^{\infty}e^{-\int_0^t B(\phi_y(s))ds}dL_t^{x/z}(\phi_y)\Big]dz dy\\
& \leq \int_{\mathcal X}\nu(y)\int_0^1(b_2|x/z|^\gamma+b_1)\frac{\ktilde(z)}{z}\sigma_1^{-2}\mathbb{E}\Big[\int_0^{\infty}e^{-b_1t}dL_t^{x/z}(\phi_y)\Big] dzdy.
\end{align*}
Noticing that for all $z\in[\varepsilon,1-\varepsilon]$ and $x\in \mathcal V_{x_0}$, $x/z\in \mathcal	V_{x_0}^{\varepsilon}$ and using Assumption \ref{assu:noyau}, we get
\begin{align*}
\nu(x) & \leq (b_2|x|^\gamma+b_1)\varepsilon^{-(1+\gamma)}\sigma_1^{-2}\int_{\mathcal X}\nu(y)\sup_{x \in \mathcal V_{x_0}^\varepsilon}\mathbb{E}\Big[\int_0^{\infty}e^{-b_1t}dL_t^{x}(\phi_y)\Big] dy \\
& = (b_2|x|^\gamma+b_1)\varepsilon^{-(1+\gamma)}\sigma_1^{-2}b_1\int_{\mathcal X}\nu(y)\int_0^{\infty}e^{-b_1t}\sup_{x \in \mathcal V_{x_0}^\varepsilon}\mathbb{E}\big[L_t^{x}(\phi_y)\big]dt dy,
\end{align*}
where the last equality comes from the integration by parts formula, see Appendix \ref{appendix:IPP} for a detailed version. By Lemma \ref{lem:borne_localtime}, we have
$\sup_{x \in  \mathcal V_{x_0}^\varepsilon}\mathbb{E}\big[L_t^x\left(\phi_y\right)\big] \lesssim 1+t^{3/2}$
 uniformly over $\mathbb Q$ and the first part of the lemma follows.
For the second part of the lemma, we have
\begin{align*}
\nu(x) 
& \geq b_1\sigma_2^{-2}\int_{\mathcal X}\nu(y)\inf_{x \in  \mathcal V_{x_0}^\varepsilon}\mathbb{E}\Big[\int_0^{\infty}e^{-b_2\int_0^t|\phi_y(s)|^\gamma ds-b_1t}dL_t^{x}(\phi_y)\Big] dy \\
& \geq  b_1\sigma_2^{-2}\int_{[-N,N] \cap \mathcal X}\nu(y)\inf_{x \in  \mathcal V_{x_0}^\varepsilon}\mathbb{E}\Big[\int_0^{T}e^{-(b_2M^\gamma+b_1)t}dL_t^{x}(\phi_y){\bf 1}_{\{\sup_{t \leq T}|\phi_y(t)|\leq M\}}\Big] dy \\
& \geq    b_1\sigma_2^{-2}e^{-(b_2M^\gamma+b_1)T}\int_{[-N,N] \cap \mathcal X}\nu(y)\inf_{x \in \mathcal V_{x_0}^\varepsilon}\mathbb{E}\big[L_T^{x}(\phi_y){\bf 1}_{\{\sup_{t \leq T}|\phi_y(t)|\leq M\}}\big] dy
\end{align*}
for arbitrary constants $M,N,T>0$. Since 
$\E\big[L_T^x(\phi_y){\bf 1}_{\{\sup_{t\leq T}|\phi_y(t)| \leq M\}}\big] \uparrow \E\big[L_T^x(\phi_y)\big]$ uniformly in $(x,y) \in \mathcal V_{x_0} \times [-N,N]$ as $M$ grows, pick $M$ large enough so that for every $y \in [-N,N] \cap \mathcal X$, we have
$$\inf_{x \in \mathcal V_{x_0}^\varepsilon}\mathbb{E}\big[L_T^{x}(\phi_y){\bf 1}_{\{\sup_{t \leq T}|\phi_y(t)|\leq M\}}\big] \geq \tfrac{1}{2}\inf_{x \in \mathcal V_{x_0}^\varepsilon}\E\big[L_T^x(\phi_y)\big].$$
Next, we use the fact that Assumption \ref{assu:unique} implies that the law of the random variable $\phi_y(t)$ admits a density $\rho_t(y,x)$ w.r.t. the Lebesgue measure and that this density is bounded away from zero on compact sets in $(y,x)$, see for instance \cite{Dacunha, vgcjj}. In turn
$\mathbb{E}\big[L_T^{x}(\phi_y)\big] =\int_0^T\rho_t(y,x)dt \geq \tau_T >0$
for some $\tau_T$ depending also on $M$ and $N$ and we infer
$$\nu(x)  \geq  \frac{\tau_T}{2}b_1\sigma_2^{-2}e^{-(b_2M^\gamma+b_1)T}\int_{[-N,N] \cap \mathcal X}\nu(y)dy$$
and we obtain the result by taking $N$ sufficiently large. The proof is complete.
\end{proof}

\subsubsection*{Completion of proof of Theorem \ref{thm:nonparametrique}}

\noindent {\it Step 1).} Write
$\widehat \nu_n(y_0)-\nu(y_0) = I + II$,
with
$$I=\frac{1}{|\mathbb U_n|}\sum_{u \in \mathbb U_n}G_h(y_0-X_u)-\nu(\mathcal Q G_h(y_0-\cdot))\;\;\text{and}\;\;II = G_h \star \nu(y_0)-\nu(y_0).$$
We plan to apply Proposition \ref{thm:convl2} to $I$ with $\psi(x,y)=\varphi(y)=G_h(y_0-y)$. By Lemma \ref{lem: uplowbound densite}, $\nu$ is locally bounded and we check that
\begin{align*}
& |\psi^2|_{\mu+\nu}  \lesssim \int_{\mathcal X} |G_h(y_0-y)^2| dy \lesssim h^{-1}\int_{\R}G(y)^2dy \lesssim h^{-1}, 
\end{align*}
\begin{align*}
& |\psi^\star \psi|_{\mu}  \lesssim \sup_y|G_h(y)|\int_{\R}|G_h(y_0-y)|dy \lesssim h^{-1}, 
\end{align*}
and
\begin{align*}
& |\psi_\star|_1|\psi|_{\wedge 1} \lesssim \Big(\int_{\mathcal X}|G_h(y_0-y)|dy\Big)^2 \lesssim 1. 
\end{align*}
Therefore, by Proposition \ref{thm:convl2}, we have $\E_\mu[I^2] \lesssim |\mathbb{U}_n|^{-1}h^{-1}$ and this term is of order $|\mathbb U_n|^{-2\beta/(2\beta+1)}$ from the choice of $h$. For the term $II$, Lemma \ref{lem: uplowbound densite} and the representation $\nu(x)  = \int_{\mathcal X}\nu(y)q(y,x)dy$ show that  $\nu \in \mathcal H^{\beta}(y_0)$ as soon as $q\in \mathcal H^{\alpha,\beta}(x_0,y_0)$. Then, by classical kernel approximation (see {\it e.g.} Chapter 1 of the book by Tsybakov \cite{Tsyb}) we have that $II^2 \lesssim h^{2\beta}$ since the order $k$ of the kernel $G$ satisfies $k > \beta$, and thus $II^2$ has the same order as $I^2$ from the choice of $h$.\\

\noindent {\it Step 2).} For the estimation of $q(x_0,y_0)$, write
$$\widehat q_n(x_0,y_0)-q(x_0,y_0) = I+II,$$
with
$$I = \frac{\mathcal M_{\mathbb U_n}\big(G_{h_1,h_2}^{\otimes 2}(x_0-\cdot,y_0-\cdot)\big)-\nu(x_0)q(x_0,y_0)}{\mathcal M_{\mathbb U_n}\big(G_{h}(x_0-\cdot)\big)\vee \varpi_n},$$
and
$$II = \frac{q(x_0,y_0)\big(\nu(x_0)-\mathcal M_{\mathbb U_n}\big(G_{h}(x_0-\cdot)\big)\vee \varpi_n\big)}{\mathcal M_{\mathbb U_n}\big(G_{h}(x_0-\cdot)\big)\vee \varpi_n}.$$
We have $\big|I\big| \leq III +IV,$
with
$$III  = \varpi_n^{-1}\big|\mathcal M_{\mathbb U_n}\big(G_{h_1,h_2}^{\otimes 2}(x_0-\cdot,y_0-\cdot)\big)-G_{h_1,h_2}^{\otimes 2}\star\nu(x_0)q(x_0,y_0)\big|$$
and
$$  IV  =  \varpi_n^{-1}\big|G_{h_1,h_2}^{\otimes 2}\star\nu(x_0)q(x_0,y_0)-\nu(x_0)q(x_0,y_0)\big|.$$
We plan to apply Proposition \ref{thm:convl2} to bound $III$ with $\psi(x,y)=G_{h_1,h_2}(x_0-x,y_0-y)$. Using Lemma \ref{lem: uplowbound densite} and the fact that $\mu$ is absolutely continuous, we have $|\psi|_{\mu+\nu} \lesssim |\psi|_1$. It readily follows that
\begin{align*}
 |\psi^2|_{\mu+\nu}  & \lesssim |G_{h_1}(x_0-\cdot)^2|_1|G_{h_2}(y_0-\cdot)^2|_1 \lesssim h_1^{-1}h_2^{-1}, 
 \end{align*}
\begin{align*}
 |\psi^\star \psi|_{\mu}  \;\,& \lesssim |G_{h_1}(x_0-\cdot)\sup_y|G_{h_2}(y_0-y)| G_{h_1}(x_0-\cdot)G_{h_2}(y_0-\cdot) |_1\\
& =|G_{h_1}(x_0-\cdot)^2|_1 \sup_y|G_{h_2}(y_0-y)| |G_{h_2}(y_0-\cdot) |_1 \lesssim h_1^{-1}h_2^{-1}, 
\end{align*}
and
\begin{align*}
& |\psi_\star|_1|\psi|_1 \lesssim \sup_x|G_{h_1}(x_0-x)||G_{h_2}(y_0-\cdot)|_1^2| G_{h_1}(x_0-\cdot)|_1 \lesssim  h_1^{-1}.\\
\end{align*}
We conclude
$$ \E_{\mu}\big[III^2\big] \lesssim \varpi_n^{-2}|\mathbb U_n|^{-1}h_1^{-1}h_2^{-1},$$
and this term has order $\varpi_n^{-2}|\mathbb U_n|^{-2s(\alpha,\beta)/(2s(\alpha,\beta)+1)}$ from the choice of $h_1$ and $h_2$. By kernel approximation and the fact that $G$ has order $k > \max\{\alpha,\beta\}$, noting that $(x,y)\mapsto \mu(x)q(x,y) \in \mathcal H^{\alpha \wedge \beta, \beta}$, we have
$$ |IV| \lesssim h_1^{\alpha \wedge \beta}+h_2^\beta \lesssim \varpi_n^{-1}|\mathbb U_n|^{-s(\alpha,\beta)/(2s(\alpha,\beta)+1)}$$
from the choice of $h_1,h_2$.\\

We turn to the term $II$. We plan to use
$$ \big(\nu(x_0)-\mathcal M_{\mathbb U_n}\big(G_{h}(x_0-\cdot)\big)\vee \varpi_n\big)^2 \lesssim \big(\nu(x_0)-\mathcal M_{\mathbb U_n}\big(G_{h}(x_0-\cdot)\big)\big)^2 +{\bf 1}_{\big\{\mathcal M_{\mathbb U_n}(G_{h}(x_0-\cdot)) < \varpi_n\big\}}.$$
Pick $n$ large enough so that 
$0 < \varpi_n \leq \tau(x_0)=\tfrac{1}{2}\inf_{\mathcal Q \in \mathbb Q, x \in \mathcal V_{x_0}}\nu(x)$,
a choice which is possible by Lemma \ref{lem: uplowbound densite}. Since $\{\mathcal M_{\mathbb U_n}(G_{h}(x_0)-\cdot) < \varpi_n\} \subset \{\mathcal M_{\mathbb U_n}(G_{h}(x_0-\cdot))-\nu(x_0) < -\tau(x_0)\}$, we further infer 

\begin{align*}
& \E_\mu\big[\big(\nu(x_0)-\mathcal M_{\mathbb U_n}\big(G_{h}(x_0-\cdot)\big)\vee \varpi_n\big)^2\big] \\
\leq & \,\E_\mu\big[\big(\nu(x_0)-\mathcal M_{\mathbb U_n}\big(G_{h}(x_0-\cdot)\big)\big)^2\big]+\PP_\mu\big(\big|\nu(x_0)-\mathcal M_{\mathbb U_n}\big(G_{h}(x_0-\cdot)\big)\big| \geq \tau(x_0)\big) \\
 \lesssim &\,\E_\mu\big[\big(\nu(x_0)-\mathcal M_{\mathbb U_n}\big(G_{h}(x_0-\cdot)\big)\big)^2\big].
\end{align*}

Applying {\it Step 1)} of the proof, we derive
$$\E_\mu\big[II^2\big] \lesssim \varpi_n^{-2}|\mathbb U_n|^{-2\beta/(2\beta+1)}$$
and this term has negligible order. The proof of Theorem \ref{thm:nonparametrique} is complete.

\subsection{Proof of Proposition \ref{prop:identifiabilite}}
Let $s(x)=\int_0^x\exp\left(-2\int_0^y\frac{r(z)}{\sigma(z)^2}dz\right)$ and $m(x)= \frac{2}{\sigma(x)^2s'(x)}$. Consider the infinitesimal generator $\mathcal L$ associated to the diffusion process \eqref{eq:diffusion_compact}, written in its divergence form
$$\mathcal L f(x) = \frac{1}{m(x)}\frac{d}{dx}\Big(\frac{1}{s(x)}\frac{d}{dx}f(x)\Big),\;\;f\in \mathcal D(\mathcal L),$$
with domain $\mathcal D(\mathcal L)$ densely defined on twice continuously differentiable functions $f$ satisfying the boundary condition $f'(0)=f'(L)=0$.
By It\^o formula and Fubini's theorem, for $f \in \mathcal D(\mathcal L)$, we have
\begin{align*}
\int_{\mathcal X}f(y)q(x,y)dy & =-\int_{0}^{1} \ktilde(z)\mathbb E\Big[\int_0^\infty f\left(z\phi_x(t)\right)\frac{d}{dt}e^{-\int_0^t B\left(\phi_x(s)\right)ds}dt\Big]dz \\
& = \int_{0}^{1}\ktilde(z)f_z(x)dz+\int_{0}^{1}\ktilde(z)\mathbb E\Big[\int_0^\infty \mathcal L f_z\left(\phi_x(t)\right) e^{-\int_0^t B(\phi_x(s))ds}dt\Big]dz,
\end{align*}
where we set $f_z(x)=f(zx)$ for $z \in [\varepsilon, 1-\varepsilon]$ since $\supp(\ktilde) \subset [\varepsilon, 1-\varepsilon]$ by Assumption  \ref{assu:noyau}.
Pick
$f(x)=\int_0^x\exp\big(2\varepsilon^{-1}\int_0^y e^{2\int_0^{u\varepsilon^{-1}}\frac{|r(v)|}{\sigma(v)^2}dv}s(u)^{-1}du\big)dy$,
and note that
\begin{align*}
f''(zx)=\frac{2}{\varepsilon s(zx)}\exp\Big(2\int_0^{zx\varepsilon^{-1}}\frac{|r(v)|}{\sigma(v)^2}dv\Big)f'(zx).
\end{align*}
It follows that for $z \geq \varepsilon$ and every $x \in \mathcal X$, we have
\begin{align*}
\mathcal{L}f_z(x) & =\frac{z}{m(x)}\frac{zs(x)f''(zx)-s'(x)f'(zx)}{s^2(x)}  \\
& =\frac{zf'(zx)}{m(x)s(x)^2}\Big(\frac{2zs(x)}{\varepsilon s(zx)}e^{2\int_0^{zx\varepsilon^{-1}}\frac{|r(v)|}{\sigma(v)^2}dv}-e^{-2\int_0^{x}\frac{r(v)}{\sigma(v)^2}dv}\Big)>0.
\end{align*}

Now let $B_1, B_2:\mathcal X \rightarrow [0,\infty)$ be two functions in an orderly class $\mathcal{B}$ according to Definition \ref{def:ordering} and write $q_{B_1}$ and $q_{B_2}$ for the associated transition densities. With no loss of generality, we may (and will) assume that $B_1(x) \leq B_2(x)$ for every $x \in \mathcal{X}$. Assume that $q_{B_1}=q_{B_2}$. Since $\supp(\ktilde) \subset [\varepsilon, 1-\varepsilon]$, we have
\begin{align*}
& \int_{\mathcal X}f(y)\big(q_{B_1}(x,y)-q_{B_2}(x,y)\big)dy \\
 = &\int_{\varepsilon}^{1-\varepsilon} \ktilde(z)\mathbb E\Big[\int_0^\infty \mathcal L f_z(\phi_x(t))\big(e^{-\int_0^t B_1(\phi_x(s))ds}-e^{-\int_0^t B_2(\phi_x(s))ds} \big)\Big]dtdz = 0.
\end{align*}
Our choice of $f$ and the property $B_1 \leq B_2$ implies that the integrand is non-negative. It follows that
$$\ktilde(z)\big(e^{-\int_0^t B_1(\phi_x(s))ds}-e^{-\int_0^t B_2(\phi_x(s))ds} \big)=0$$
$dzdt \otimes \mathbb P$-a.s.
Picking $z$ such that $\ktilde(z) >0$, we obtain 
$\int_0^t B_1(\phi_x(s))ds=\int_0^t B_2(\phi_x(s))ds$
$\mathbb P$-a.s. for every $t \geq 0$ by continuity of the integrand in $t$.
By the occupation times formula, it follows that
$\int_{\mathcal X}\big(B_1(y)-B_2(y)\big)L_t^y(\phi_x)dy=0,$
almost-surely, and by the ordering property, $B_1(y)= B_2(y)$ for every $y$ such that $L_t^y(\phi_x)>0$, {\it i.e.} for $y\in [\inf_{0 \leq s \leq t}\phi_x(s), \sup_{0 \leq s \leq t}\phi_x(s)] \rightarrow \mathcal{X}$
as $t \rightarrow \infty$. The proof of Proposition \ref{prop:identifiabilite} is complete. 

\subsection{Proof of Theorem \ref{th:consistency}}
\subsubsection*{Preparation for the proof}
We first establish uniform bounds for $q_{\vartheta}(x,y)$. Remember that in the reflected case, we have $\mathcal X=[0,L]$ and $\supp(\kappa) \subset [\varepsilon, 1-\varepsilon]$ under Assumption  \ref{assu:noyau}.
\begin{lem}\label{lemma:bound_p}
Work under Assumptions \ref{assu:unique}, \ref{assu:noyau} and \ref{assu:nondegeneracy}. For sufficiently small $\eta >0$, we have:
\begin{align*}
0 < \inf_{x \in \mathcal X, y \in \mathcal X_\eta, \vartheta\in \Theta}q_{\vartheta}(x,y) \leq \sup_{x,y\in \mathcal X, \vartheta \in \Theta}q_{\vartheta}(x,y)<\infty, 
\end{align*}
where $\mathcal X_\eta = [0,(1-\varepsilon)L-\eta]$.
\end{lem}
\begin{proof}
The proof is close to that of Lemma \ref{lem: uplowbound densite}. Let $x\in \mathcal X$ and $y \in \mathcal X_\eta $. We have 
\begin{align*}
\inf_{\vartheta\in\Theta}q_{\vartheta}(x,y) & \geq b_3\sigma_2^{-2} \int_{\varepsilon\vee yL^{-1}}^{1-\varepsilon}\frac{\ktilde(z)}{z}\mathbb{E}\big[\int_0^{\infty}e^{-b_4t}dL_t^{y/z}(\phi_x)\big]dz \\
&\geq  b_3\sigma_2^{-2}b_4 \int_{\varepsilon\vee yL^{-1}}^{1-\varepsilon}\frac{\ktilde(z)}{z}\int_0^{\infty}e^{-b_4t}\mathbb{E}\big[L_t^{y/z}(\phi_x)\big]dtdz  \\
& \geq (1-\varepsilon)^{-1}b_3\sigma_2^{-2}b_4\frac{\eta}{L}\int_0^\infty e^{-b_4t}\inf_{x,y \in \mathcal X} \mathbb{E}\big[L_t^{y}(\phi_x)\big]dt
\end{align*} 
According to \cite{cattiaux1992stochastic}, Section 5, proof of Lemma 5.37, the law of $\phi_x(t)$ is absolutely continuous with density $y\mapsto \rho_t(x,y)$  that can be taken continuous and that satisfies $\inf_{x,y\in\mathcal X}\rho_t(x,y)>0$ for every $t >0$. Therefore
$$\inf_{x,y\in \mathcal X}\mathbb{E}\big[L_t^{y}(\phi_x)\big] = \inf_{x,y\in \mathcal X}\int_0^t\rho_s(x,y)ds> 0$$
and the result follows.
The upper bound readily follows from
\begin{align*}
\sup_{\vartheta\in\Theta}q_{\vartheta}(x,y) & \leq b_4\sigma_1^{-2}\Big(\int_{\varepsilon}
^{1-\varepsilon}\frac{\ktilde(z)}{z}dz\Big)\sup_{x,y\in \mathcal X}\mathbb{E}\big[\int_{0}^{\infty}e^{-b_3t}dL_t^{y}(\phi_x)\big] \\
& \leq \varepsilon^{-1}b_4\sigma_1^{-2}\sup_{x,y\in \mathcal X}\mathbb{E}\big[\int_{0}^{\infty}e^{-b_3t}dL_t^{y}(\phi_x)\big] \\
& = \varepsilon^{-1}b_4\sigma_1^{-2}b_3\sup_{x,y\in \mathcal X}\Big(\int_{0}^{\infty}e^{-b_3t}\mathbb{E}\big[L_t^{y}(\phi_x)\big]dt\Big) 
\end{align*}
which is finite by Lemma \ref{lem:borne_localtime}.
\end{proof}


\subsubsection*{Completion of proof of Theorem \ref{th:consistency}}
This proof is classical (see for instance van der Vaart \cite{vdV00} Theorem 5.14). We nevertheless give it  for self-containedness. For $a \in \Theta$, let
$$\mathcal M(a,\vartheta)=\int_{\mathcal{X}}\nu_{\vartheta}(dx)\int_{\mathcal{X}}\log q_{a}(x,y)q_{\vartheta}(x,y)dy.$$
First, $a\mapsto \mathcal M(a,\vartheta)$ has a unique maximum at $a=\vartheta$, as stems from the inequality $\log(x)\leq 2(\sqrt{x}-1)$ for $x\geq 0$. Indeed 
\begin{align*}
\mathcal M(a,\vartheta)-\mathcal M(\vartheta,\vartheta)&=\int_{\mathcal{X}}\nu_{\vartheta}(dx)\int_{\mathcal{X}}\log\frac{q_{a}(x,y)}{q_{\vartheta}(x,y)}q_{\vartheta}(x,y)dy\\
&\leq \int_{\mathcal{X}}\nu_{\vartheta}(dx)\big(\int_{\mathcal{X}}2\sqrt{q_{a}(x,y)}\sqrt{q_{\vartheta}(x,y)}dy-2\big)\\
&\leq -\int_{\mathcal{X}}\nu_{\vartheta}(dx)\int_{\mathcal{X}}\Big(\sqrt{q_{a}(x,y)}-\sqrt{q_{\vartheta}(x,y)}\Big)^2dy \leq 0.
\end{align*} 
Next, writing $m_{\mathcal U}(x,y)=\sup_{a \in \mathcal U}\log q_{a}(x,y)$, we prove that for every $a \neq \vartheta \in\Theta$, there exists a  neighbourhood $\mathcal U_{a}$ of $a$ such that:
\begin{equation}
\label{eq:maj_m}
\nu_{\vartheta}\big(\mathcal Q_{\vartheta}m_{\mathcal U_{a}}\big) <\nu_{\vartheta}\big(\mathcal Q_{\vartheta} \log q_{\vartheta}\big) = \mathcal M(\vartheta,\vartheta).
\end{equation}


Pick a decreasing sequence of open balls $(\mathcal U_\ell(a))_{\ell \geq 1}$ around $a$ with vanishing diameters. For every $x,y\in\mathcal{X}$ we have $m_{\mathcal U_\ell(a)}(x,y)\downarrow \log q_{a}(x,y)$ by continuity of $a \mapsto \log q_{a}(x,y)$ thanks to the continuity of $B_0$ according to Assumption \ref{assu:nondegeneracy}. By Lemma \ref{lemma:bound_p}, we also have
$
\nu_{\vartheta}\big(\mathcal Q_{\vartheta}m_{\mathcal U}\big)<\infty
$
for any $\mathcal U\subset \Theta$ therefore
\begin{align*}
\nu_{\vartheta}\big(\mathcal Q_{\vartheta}m_{\mathcal U_l(a)}\big)\downarrow \nu_{\vartheta}\big(\mathcal Q_{\vartheta}\log q_{a})= \mathcal M(a,\vartheta)\leq \mathcal M(\vartheta,\vartheta)
\end{align*}
by monotone convergence with equality only if $a=\vartheta$, and this proves the existence of $\mathcal U_a$ such that \eqref{eq:maj_m} holds.
We are now ready to prove the consistency result. For $\eta' >0$, the compact ball
$$\mathcal C_{\eta'}(\vartheta) = \{a\in \Theta, |a-\vartheta| \geq \eta'\}$$
can be covered by finitely many open neighbourhoods $\mathcal U_{a_1}, \ldots \mathcal U_{a_p}$ with $a_i \in \mathcal C_{\eta'}(\vartheta)$ and such that \eqref{eq:maj_m} holds for every $\mathcal U_{a_i}$. For $\eta >0$, let 
$$m_{\mathcal U}^{(\eta)}(x,y) = \sup_{a \in \mathcal U} \log q_a(x,y){\bf 1}_{\{q_a(x,y) \geq \eta\}}.$$ 
Abbreviating $\mathcal L_n(a, (X_u)_{u \in \mathbb U_n})$ by $\mathcal L_n(a) $, it follows that
\begin{align}
|\mathbb{U}_n^\star|^{-1}\sup_{a\in \mathcal C_{\eta'}(\vartheta)}\log \mathcal L_n(a) 
& \leq  
\max_{1 \leq i \leq p}|\mathbb{U}_n^\star|^{-1}\sum_{u\in\mathbb{T}_n^\star}m_{\mathcal U_{a_i}}^{(\eta)}\left(X_{u^-},X_u\right) \nonumber \\ 
& \rightarrow  \max_{1\leq i \leq p}\nu_{\vartheta}\big(\mathcal Q_{\vartheta}m_{\mathcal U_{a_i}}\big) < \mathcal M(\vartheta, \vartheta) \label{eq:maj essential}
\end{align}
in probability as $n \rightarrow \infty$ and letting $\eta \rightarrow 0$, as stems from
Corollary \ref{thm:convl2_compact} and the fact that $\sup_{x,y \in \mathcal X_\eta}m_{\mathcal U_{a_i}}(x,y) < \infty$ by Lemma \ref{lemma:bound_p}. Finally, if $\widehat{\vartheta}_n\in  \mathcal C_{\eta'}(\vartheta)$, then, by definition of $\widehat \vartheta_n$, we have
\begin{align*}
|\mathbb{U}_n^\star|^{-1}\sup_{a \in \mathcal C_{\eta'}(\vartheta)}\log \mathcal L_n(a)\geq |\mathbb{U}_n^\star|^{-1} \log \mathcal L_n(\widehat{\vartheta}_n)\geq |\mathbb{U}_n^\star|^{-1}\log  \mathcal L_n(\vartheta)> \mathcal M(\vartheta, \vartheta)-\epsilon_n,
\end{align*}
where $\epsilon_n \rightarrow 0$ in probability, as follows from Corollary \ref{thm:convl2_compact}. We conclude the proof by noticing that
\begin{align*}
\big\{\widehat{\vartheta}_n\in   \mathcal C_{\eta'}(\vartheta) \big\} \subset \big\{ |\mathbb{U}_n^\star|^{-1}\sup_{a\in  \mathcal C_{\eta'}(\vartheta) }\mathcal L_n(a)\geq \mathcal M(\vartheta, \vartheta)-\epsilon_n\big\}
\end{align*}
and the fact that the probability of this last event converges to $0$ by \eqref{eq:maj essential} as $n \rightarrow \infty$.

\subsection{Proof of Theorem \ref{thm:norm_asympt}}

\subsubsection*{Preparation for the proof} We start by proving some useful estimates on the gradient and Hessian of $\log q_\vartheta$.
Let \[\Gamma_{\vartheta}=\nabla_{\vartheta}\log q_{\vartheta}=\left(\partial_{\vartheta_1}\log q_{\vartheta},\ldots,\partial_{\vartheta_d}\log q_{\vartheta}\right),\quad \Gamma_{\vartheta,i}=\partial_{\vartheta_i}\log q_{\vartheta},\quad 1\leq i\leq d.\]

\begin{lem}\label{lemma:partialp}
Work under Assumptions \ref{assu:unique}, 
\ref{assu:noyau}, \ref{assu:nondegeneracy} and \ref{assu:reg_B}. For every $1\leq i,j\leq d$ and $\eta >0$, we have
\begin{align*}
\sup_{x \in \mathcal X,y\in \mathcal X_\eta,\vartheta\in\Theta}\left|\Gamma_{\vartheta,i}(x,y)\right|<\infty, \sup_{x \in \mathcal X,y\in \mathcal X_\eta,\vartheta\in\Theta}\left|\partial_{\vartheta}\Gamma_{\vartheta}(x,y)_{i,j}\right|<\infty, \sup_{x \in \mathcal X,y\in \mathcal X_\eta,\vartheta\in\Theta}\left\Vert\partial^2_{\vartheta}\Gamma_{\vartheta}(x,y)\right\Vert<\infty
\end{align*}
where $\left\Vert\cdot\right\Vert$ corresponds to the operator norm for the Hessian $\partial^2_{\vartheta}\Gamma_{\vartheta}(x,y)$.
\end{lem}
\begin{proof}
According to Lemma \ref{lemma:bound_p}, since
\begin{align*}
\Gamma_{\vartheta}(x,y)=\frac{\partial_{\vartheta}q_{\vartheta}(x,y)}{q_{\vartheta}(x,y)}
\end{align*}
componentwise, it suffices to show $|\partial_{\vartheta_i}q_{\vartheta}(x,y)|\lesssim 1$ in order to establish the first bound.
Recall that for $(x,y) \in \mathcal X \times[0, (1-\varepsilon) L]$,
\begin{align*}
q_\vartheta(x,y)=\int_{y/L}^1\frac{\ktilde(z)}{z}B_0(\vartheta,y/z)\sigma(y/z)^{-2}\mathbb{E}\left[\int_0^{\infty}e^{-\int_0^t B_0(\vartheta,\phi_x(s))ds}dL_t^{y/z}(\phi_x)\right]dz.
\end{align*}
Taking the derivative with respect to $\vartheta_i$ yields
\begin{align*}
\partial_{\vartheta_i}&q_{\vartheta}(x,y)=\int_{y/L}^1\frac{\ktilde(z)}{z}\partial_{\vartheta_i}B_0(\vartheta,y/z)\sigma(y/z)^{-2}\,\mathbb{E}\Big[\int_{0}^{\infty}e^{-\int_0^t B_0(\vartheta,\phi_x(s))ds}dL_t^{y/z}\left(\phi_x\right)\Big]dz\\
&-\int_{y/L}^1\frac{\ktilde(z)}{z}B_0(\vartheta,y/z)\sigma(y/z)^{-2}\,\mathbb{E}\Big[\int_{0}^{\infty} e^{-\int_0^t B_0(\vartheta, \phi_x(s))ds}\Big(\int_0^t\partial_{\vartheta_i}B_0(\vartheta,\phi_x(s))ds\Big) \,dL_t^{y/z}(\phi_x)\Big].
\end{align*}
By Assumption \ref{assu:unique}, \ref{assu:nondegeneracy} and \ref{assu:reg_B},
\begin{align*}
|\partial_{\vartheta_i}q_{\vartheta}(x,y)|\lesssim &\int_{y/L}^1\mathbb{E}\Big[\int_{0}^{\infty}(1+t)e^{-\int_0^t B_0(\vartheta,\phi_x(s))ds}dL_t^{y/z}\left(\phi_x\right)\Big]\ktilde(z)dz.
\end{align*}
Next, by Assumption \ref{assu:nondegeneracy},
\begin{align*}
\mathbb{E}\Big[\int_{0}^{\infty} (1+t)e^{-\int_0^t B_0(\vartheta,\phi_x(s))ds}dL_t^{y/z}(\phi_x)\Big] & \leq \mathbb{E}\big[\int_{0}^{\infty} (1+t)e^{-b_3t}dL_t^{y/z}\left(\phi_x\right)\big]  \\
& =\int_{0}^{\infty} (1-b_3(1+t))e^{-b_3t}\E[L_t^{y/z}(\phi_x)]dt,
\end{align*}
where the last equality comes from the integration by parts formula (see Appendix \ref{appendix:IPP}).
This last term is bounded by Lemma \ref{lem:borne_localtime} and $|\partial_{\vartheta_i}q_{\vartheta}(x,y)|\lesssim 1$ follows. We turn to the second bound: clearly, for $1 \leq i,j \leq d$ 
\begin{align*}
\partial_{\vartheta}\Gamma_{\vartheta}(x,y)_{i,j}=\frac{\partial_{\vartheta_i\vartheta_j}^2q_{\vartheta}(x,y)q_{\vartheta}(x,y)-\partial_{\vartheta_i}q_{\vartheta}(x,y)\partial_{\vartheta_j}q_{\vartheta}(x,y)}{q_{\vartheta}(x,y)^2}
\end{align*}
and thanks to Lemma \ref{lemma:bound_p} and the first bound, we only need to show $|\partial_{\vartheta_i\vartheta_j}^2q_{\vartheta}(x,y)|\lesssim 1$
in order to obtain the second bound. Define
$\omega_t(y,z,\vartheta) = {\bf 1}_{[yL^{-1},1]}(z)\frac{\ktilde(z)}{z\sigma(y/z)^2}\exp(-\int_0^t B_0(\vartheta, \phi_x(s))ds)$.
We have
\begin{align*}
&\partial^2_{\vartheta_i\vartheta_j}q_{\vartheta}(x,y)  =  \int_{0}^{1}\partial^2_{\vartheta_i\vartheta_j}B_0(\vartheta, y/z)\mathbb{E}\Big[\int_{0}^{\infty}\omega_t(y,z,\vartheta)dL_t^{y/z}\left(\phi_x\right)\Big]dz\\
 & + \int_0^1B_0(\vartheta,y/z)\mathbb{E}\Big[\int_{0}^{\infty} \omega_t(y,z,\vartheta)\int_0^t\partial_{\vartheta_i}B_0(\vartheta,\phi_x(s))ds \int_0^t\partial_{\vartheta_j}B_0(\vartheta,\phi_x(s))ds\,dL_t^{y/z}\left(\phi_x\right)\Big]dz \\
 & -\sum_{(\ell,\ell')\in\lbrace(i,j), (j,i)\rbrace}\int_{0}^{1}\partial_{\vartheta_\ell}B_0(\vartheta,y/z)\mathbb{E}\Big[\int_{0}^{\infty} \omega_t(y,z,\vartheta)\int_0^t\partial_{\vartheta_{\ell'}}B_0(\vartheta,\phi_x(s))dsdL_t^{y/z}(\phi_x)\Big]dz\\
& + \int_{0}^{1}B_0(\vartheta,y/z)\mathbb{E}\Big[\int_{0}^{\infty}\omega_t(y,z,\vartheta)\int_0^t\partial^2_{\vartheta_i\vartheta_j}B_0(\vartheta, \phi_x(s))dsdL_t^{y/z}\left(\phi_x\right)\Big]dz
\end{align*}
and we proceed in the same way as for the first estimate, using repeatedly Assumption \ref{assu:unique}, \ref{assu:nondegeneracy} and \ref{assu:reg_B}. The proof of the third bound is analogous.
\end{proof}

\subsubsection*{Completion of proof of Theorem \ref{thm:norm_asympt}} 
This proof is classical (see for instance van der Vaart \cite{vdV00} Theorem 5.41). We nevertheless give it  for self-containedness.
By definition of $\widehat{\vartheta}_n$
and a Taylor expansion around $\vartheta$, we have  
\begin{align*}
0 & = \sum_{u\in\mathbb{U}_n^\star}\Gamma_{\widehat{\vartheta}_n}(X_{u^-},X_u) \\
& =\sum_{u\in\mathbb{U}_n^\star}\Big(\Gamma_{\vartheta}(X_{u^-},X_u)+\partial_{\vartheta}{\Gamma}_{\vartheta}(X_{u^-},X_u)(\widehat{\vartheta}_n-\vartheta)
+(\widehat{\vartheta}_n-\vartheta)^T\partial^2_{\vartheta}{\Gamma}_{\widetilde{\vartheta}_n}(X_{u^-},X_u)(\widehat{\vartheta}_n-\vartheta)\Big),
\end{align*}
for some $\widetilde{\vartheta}_n$ on the segment line between $\vartheta$ and $\widehat{\vartheta}_n$.
Rearranging the sum and introducing the normalisation $|\mathbb U_n^\star|^{1/2}$, we derive
\begin{align} \label{eq:decom fond}
& \big(|\mathbb{U}_n^\star|^{-1}\sum_{u\in\mathbb{U}_n^\star}\partial_{\vartheta}{\Gamma}_{\vartheta}(X_{u^-},X_u)+|\mathbb{U}_n^\star|^{-1}\sum_{u\in\mathbb{U}_n^\star}(\widehat{\vartheta}_n-\vartheta)^T\partial^2_{\vartheta}{\Gamma}_{\widetilde{\vartheta}_n}(X_{u^-},X_u)\big)|\mathbb{U}_n^\star|^{1/2}(\widehat{\vartheta}_n-\vartheta) \\
=& \,-|\mathbb{U}_n^\star|^{-1/2}\sum_{u\in\mathbb{U}_n^\star}\Gamma_{\vartheta}(X_{u^-},X_u). \nonumber
\end{align}
We plan to apply an extension of the central limit theorem for bifurcating Markov chain proved in Guyon for $\mathbb U_n = \mathbb T_n$, see \cite{guyon2007limit} Corollary 24 on the right-hand side. It is not difficult to see that the result still holds if one replaces $\mathbb T_n$ by an incomplete tree according to Definition \ref{def incomplete tree}. We omit the details. By Lemma \ref{lemma:bound_p} and \ref{lemma:partialp} we have that $\mathcal Q_{\vartheta}(\Gamma_{\vartheta,i}\Gamma_{\vartheta,j})$ and $\mathcal Q_{\vartheta}(\Gamma_{\vartheta,i}\Gamma_{\vartheta,j}\Gamma_{\vartheta,k}\Gamma_{\vartheta,l})$ are bounded functions on $\mathcal{X}$ for all $1\leq i,j,k,l\leq d$.
Moreover, we have $\nu_{\vartheta}(\mathcal Q_{\vartheta}\Gamma_{\vartheta,i})=0$. Therefore
\begin{equation} \label{eq:TCL guyon}
|\mathbb{U}_n^\star|^{-1/2}\sum_{u\in\mathbb{U}_n^\star}\Gamma_{\vartheta}(X_{u^-},X_u) \rightarrow \mathcal{N}\big(0,\Psi(\vartheta)\big)
\end{equation}
in distribution as $n \rightarrow \infty$, where $\Psi(\vartheta)$ is the Fisher information matrix defined after Assumption \ref{assu:reg_B}. Next, since $\partial_{\vartheta}\Gamma_{\vartheta}$ is bounded by Lemma \ref{lemma:partialp}, we have
\begin{equation} \label{eq:conv proba den}
|\mathbb{U}_n^\star|^{-1}\sum_{u\in\mathbb{U}_n^\star}\partial_{\vartheta}{\Gamma}_{\vartheta}(X_{u^-},X_u) \rightarrow \Psi(\vartheta)
\end{equation}
in probability as $n \rightarrow \infty$. Moreover, by Lemma \ref{lemma:partialp}, we have:
$\sup_{x,y\in\mathcal{X},\vartheta\in\Theta}
 \big\|  \partial_{\vartheta}^2 \Gamma_{\vartheta}(x,y) \big\| < \infty$
and since $\widehat \vartheta_n -\vartheta$ converges to $0$ by Theorem \ref{th:consistency}, it follows that
\begin{equation} \label{eq:conv 0 reste}
|\mathbb{U}_n^\star|^{-1}\sum_{u\in\mathbb{U}_n^\star}(\widehat{\vartheta}_n-\vartheta)^T\partial^2_{\vartheta}{\Gamma}_{\widetilde{\vartheta}_n}(X_{u^-},X_u) \rightarrow 0
\end{equation} 
in probability as $n \rightarrow \infty$ tends to infinity. Combining \eqref{eq:TCL guyon}, \eqref{eq:conv proba den} and \eqref{eq:conv 0 reste} in \eqref{eq:decom fond} we finally obtain
$$\Psi(\vartheta) |\mathbb{U}_n^\star|^{1/2}(\widehat{\vartheta}_n-\vartheta) \rightarrow \mathcal{N}\big(0,\Psi(\vartheta)\big)$$
in distribution as $n \rightarrow \infty$. We conclude thanks to the invertibility of $\Psi(\vartheta)$ granted by Assumption \ref{assu:Ppsi_invertible}.

\section{Appendix} \label{sec:appendix}
\subsection{Example of a model satisfying $\rho<1/2$}\label{appendix:rho}
We elaborate on Remark \ref{rk rho petit}.
\begin{lem}
Assume that
\begin{itemize}
\item[i)] $\phi_x(t)$ si an Ornstein-Uhlenbeck process on $\mathcal X = \R$: we have $r(x)=-\beta x$ and $\sigma(x) = \sigma$ for every $x \in \mathcal X$ and some $\beta,\sigma >0$,
\item[ii)] the division rate is constant: we have $B(x) = b$ for every $x \in \mathcal X$ and some $b >0$,
\item[iii)] the fragmentation distribution is uniform: we have $\kappa(z) = 1/(1-2\varepsilon)$ on $[\varepsilon,1-\varepsilon]$ for some $\varepsilon>0$.
\end{itemize}
Then, $\mathcal Q$ admits an invariant probability distribution $\nu$ and for $V(x)=x^2$, there exist $C>0$ and $\rho \in (0,1)$ such that for every $m \geq 1$, the bound
$$\big|\mathcal Q^m\varphi-\nu(\varphi)\big|_V \leq C\rho^m\big|\varphi-\nu(\varphi)\big|_V$$
holds  as soon as $|\varphi|_V<\infty$. Moreover, for $b$ small enough, we have $\rho<1/2$. 
\end{lem}
\begin{proof}
According to Proposition \ref{prop:lyapunov}, the drift condition holds true. Next, we slightly modify the proof of Proposition \ref{prop:minoration}.\\

\noindent {\it Step 1).} For large enough $w>0$, we aim at finding $\lambda>1/2$ and a probability measure $\mu$ on $\mathcal{X}$ such that
\begin{align*}
\inf_{\{x, |x|\leq w \}}\mathcal Q(x,\mathcal A)\geq \lambda\mu(\mathcal A)
\end{align*}
for every Borel set $\mathcal A \subset \mathcal X$. Let $x\in [-w,w]$ and $\mathcal A \subset \mathcal X$ be a Borel set. In this setting, we have
\begin{align*}
\mathcal Q(x,\mathcal A) & =\frac{b}{1-2\varepsilon}\int_\varepsilon^{1-\varepsilon}\int_{0}^{\infty} \mathbb{E}\Big[\mathbf{1}_{\left\lbrace z\phi_x(t)\in \mathcal A\right\rbrace}\Big]e^{-bt}dt dz.
\end{align*}
Using successively Fubini's theorem and the occupation time formula, we get
\begin{align*}
\mathcal Q(x,\mathcal A) =\frac{b}{\sigma^2(1-2\varepsilon)}\int_\varepsilon^{1-\varepsilon}\int_{\mathbb R} \mathbf 1_{\{zy\in\mathcal A\}} \mathbb{E}\Big[ \int_0^\infty e^{-bt}dL_t^y(\phi_x)\Big]dydz.
\end{align*}
Integration by parts (see Appendix \ref{appendix:IPP}) yields
\begin{align}\label{eq:1}
\mathcal Q(x,\mathcal A) =\frac{b^2}{\sigma^2(1-2\varepsilon)}\int_\varepsilon^{1-\varepsilon}\int_{\mathbb R} \mathbf 1_{\{zy\in\mathcal A\}} \int_0^\infty e^{-bt}\mathbb{E}\Big[L_t^y(\phi_x)\Big]dtdydz.
\end{align}
We next compute the expectation of the local time. We have
\begin{align*}
\frac{1}{\sigma^2}\mathbb{E}\Big[L_t^y(\phi_x)\Big] & = \mathbb{E}\Big[\lim_{\varepsilon\downarrow 0}\frac{1}{\varepsilon}\int_0^t \mathbf 1_{\{y\leq \phi_x(s)\leq y+\varepsilon\}}ds\Big]\\
& =\lim_{\varepsilon\downarrow 0}\frac{1}{\varepsilon}\int_0^t \mathbb P \Big(y\leq \phi_x(s)\leq y+\varepsilon\Big)ds.
\end{align*}
Since $\phi_x(t)$ is an Ornstein-Uhlenbeck process, we have the representation
$$\phi_x(t) = \frac{\sigma}{\sqrt{2\beta}}W_x\big(e^{2\beta t}\big)e^{-\beta t},
$$
where $W_x$ is a Brownian motion starting from $x$. Then,
\begin{align}\label{eq:2}\nonumber
\frac{1}{\sigma^2}\mathbb{E}\Big[L_t^y(\phi_x)\Big] & = \lim_{\varepsilon\downarrow 0}\frac{1}{\varepsilon}\int_0^t \mathbb P \Big(y\leq \frac{\sigma}{\sqrt{2\beta}}W_x\big(e^{2\beta s}\big)e^{-\beta s}\leq y+\varepsilon\Big)ds\\\nonumber
&  = \lim_{\varepsilon\downarrow 0}\frac{1}{\varepsilon}\int_0^t \Big[F\Big(e^{2\beta s},(y+\varepsilon)\frac{\sqrt{2\beta}}{\sigma}e^{\beta s}\Big)-F\Big(e^{2\beta s},y\frac{\sqrt{2\beta}}{\sigma}e^{\beta s}\Big)\Big]ds\\\nonumber
& = \int_0^te^{\beta s}\frac{\sqrt{2\beta}}{\sigma}\frac{1}{\sqrt{2\pi e^{2\beta s}}}e^{-\frac{1}{2e^{2\beta s}}\Big(x-ye^{\beta s}\frac{\sqrt{2\beta}}{\sigma}\Big)^2}ds\\
& = \frac{\sqrt{2\beta}}{\sigma}\frac{1}{\sqrt{2\pi}}\int_0^te^{-\frac{1}{2}\Big(x e^{-\beta s}-y\frac{\sqrt{2\beta}}{\sigma}\Big)^2}ds
\end{align}
where $F(t,x)$ is the cumulative density function of $W_x(t)$. Therefore, combining \eqref{eq:1} and \eqref{eq:2} we get
\begin{align*}
\mathcal Q(x,\mathcal A) =\frac{b^2}{1-2\varepsilon}\frac{\sqrt{2\beta}}{\sigma}\frac{1}{\sqrt{2\pi}}\int_\varepsilon^{1-\varepsilon}\int_{\mathbb R} \mathbf 1_{\{zy\in\mathcal A\}} \int_0^\infty e^{-bt}\int_0^te^{-\frac{1}{2}\big(x e^{-\beta s}-y\frac{\sqrt{2\beta}}{\sigma}\big)^2}dsdtdydz. 
\end{align*}
Integrating by parts again and a change of variables yield
\begin{align*}
\mathcal Q(x,\mathcal A) & =\frac{b}{1-2\varepsilon}\frac{\sqrt{2\beta}}{\sigma}\frac{1}{\sqrt{2\pi}}\int_\varepsilon^{1-\varepsilon}\frac{1}{z}\int_{\mathbb R} \mathbf 1_{\{u\in\mathcal A\}} \int_0^\infty e^{-bt}e^{-\frac{1}{2}\big(x e^{-\beta t}-\frac{u}{z}\frac{\sqrt{2\beta}}{\sigma}\big)^2}dtdudz.
\end{align*}
Next, $\big(x e^{-\beta t}-\frac{u}{z}\frac{\sqrt{2\beta}}{\sigma}\big)^2\leq \big(we^{-\beta t}+\frac{|u|}{z}\frac{\sqrt{2\beta}}{\sigma}\big)^2$ for all $x\in [-w,w]$, $t\geq 0$ and $z\in[\varepsilon,1-\varepsilon]$. 
Finally, using Fubini's theorem,
\begin{align*}
\mathcal Q(x,\mathcal A)\geq \frac{1}{\sqrt{\pi}} \int_\mathbb{R} \mathbf{1}_{\lbrace u\in\mathcal{A}\rbrace}f_w(u)du,
\end{align*}
with 
\begin{align*}
f_w(u) = b\frac{\sqrt{\beta}}{\sigma}\int_\varepsilon^{1-\varepsilon}\frac{1}{z} \int_0^\infty e^{-bt}e^{-\frac{1}{2}\big(w e^{-\beta t}+\frac{|u|}{z}\frac{\sqrt{2\beta}}{\sigma}\big)^2}dt\frac{dz}{1-2\varepsilon}.
\end{align*}
We now construct a probability measure from $f_w$. First, 
\begin{align*}
\int_{\mathbb R} e^{-\frac{1}{2}\big(we^{-\beta t}+\frac{|u|}{z}\frac{\sqrt{2\beta}}{\sigma}\big)^2}du = 2 \frac{z\sigma}{\sqrt{\beta}}\int_{\frac{w}{\sqrt{2}}e^{-\beta t}}^\infty e^{-y^2}dy.
\end{align*}
Combining this with Fubini's theorem, we get
\begin{align*}
\int_{\mathbb R}f_w(u)du = 2b\int_0^{\infty}e^{-bt}\Big( \int_{\frac{w}{\sqrt{2}}e^{-\beta t}}^\infty e^{-y^2}dy\Big)dt.
\end{align*}
Fubini's theorem again yields
\begin{align*}
\int_{\mathbb R}f_w(u)du = 2b\int_0^{\frac{w}{\sqrt{2}}}e^{-y^2}\Big( \int_{\frac{1}{\beta}\ln(w/(y\sqrt 2))}^\infty e^{-bt}dt\Big)dy = 2\int_0^{\frac{w}{\sqrt{2}}}e^{-y^2}e^{-\frac{b}{\beta}\ln(w/(y\sqrt 2 ))}dy.
\end{align*}
Finally, define the probability measure $\mu_w(dy)=g_w(y)dy$ on $\mathcal X$ by
\begin{align*}
g_w(y)=\Big(2\int_0^{\frac{w}{\sqrt{2}}}e^{-y^2}e^{-\frac{b}{\beta}\ln(w/(y\sqrt 2))}dy\Big)^{-1}f_w(y),
\end{align*}
and let $$\lambda=\frac{2}{\sqrt{\pi}} \int_0^{\frac{w}{\sqrt{2}}}e^{-y^2}e^{-\frac{b}{\beta}\ln(w/(y\sqrt 2))}dy.$$
Moreover, as
\begin{align*}
\lambda\underset{b\rightarrow 0}{\longrightarrow}\frac{2}{\sqrt{\pi}} \int_0^{\frac{w}{\sqrt{2}}}e^{-y^2}dy\quad \text{and}\quad \frac{2}{\sqrt{\pi}} \int_0^{\frac{w}{\sqrt{2}}}e^{-y^2}dy\underset{w\rightarrow \infty}{\longrightarrow} 1,
\end{align*}
there exists $b_0>0$ and $w_0>0$ such that $\lambda=\lambda(w_0,b_0)>1/2$ and we thus have established
\begin{align*}
\mathcal Q(x,\mathcal A)\geq \lambda\mu(\mathcal A),
\end{align*}
with $\lambda>1/2$.\\

\noindent {\it Step 2).} Applying Theorem 1.2. in \cite{hairer2011}, we obtain the exponential convergence of the tagged-chain at rate $$\rho=(1-(\lambda-\lambda_0))\vee \frac{2+w\gamma v_0}{2+w\gamma},$$
for any $\lambda_0\in(0,\lambda)$ and $v_0\in (v_1+2v_2/w, 1)$, where $\gamma = \lambda_0/v_2$. We just proved that $\lambda(w_0,b_0)>1/2$ so that we can choose $\lambda_0\in(0,\lambda)$ such that $1-\lambda+\lambda_0<1/2$. Next, let $\eta\in(0,1)$ be such that $v_0 = \eta+(1-\eta)(v_1+2v_2/w)$. Then,  
$$R(w)=\frac{2+w\gamma v_0}{2+w\gamma}=\frac{2+w\gamma\eta+ (1-\eta)w\gamma v_1+2(1-\eta)\gamma v_2}{2+w\gamma}$$
 is a decreasing function in $w$ and 
$$\lim_{w\rightarrow \infty} R(w) = \eta+(1-\eta)v_1.$$ 
Moreover, according to the proof of Proposition \ref{prop:lyapunov}, $v_1 = m(\kappa)=\frac{1}{3}(\varepsilon^2-\varepsilon+1)\leq 1/3$. Finally, we can choose $\eta\in(0,1)$ and $w>2v_2/(1-v_1)$ such that $R(w)<1/2$ and we get the result. 
\end{proof}

\subsection{Proof of Lemma \ref{lem:borne_localtime}} \label{proof of lem loc time}
{\it Step 1)}. Fix $\delta >0$ and let $\mathcal K_\delta = \{y \in \mathcal X, \inf_{z \in \mathcal K}|y-z|\leq \delta\}$ denote the $\delta$-enlargement of $\mathcal K$. 
For $x \in \mathcal X$, let
$$\tau_x=\inf\{t\geq 0, \phi_x(t) \in \mathcal K_\delta\},\;\;\inf \emptyset = \infty,$$
and 
$$\phi^{\mathcal K_\delta}_x(t) = 
\left\{
\begin{array}{lll}
\phi_{\sup \mathcal K_\delta }\big((t-\tau_x)_+\big) & \mathrm{if} & x > \sup \mathcal K_\delta \\ 
\phi_x(t) & \mathrm{if} & x \in \mathcal K_\delta \\ 
\phi_{\inf \mathcal K_\delta }\big((t-\tau_x)_+\big)& \mathrm{if} & x < \inf \mathcal K_\delta.
\end{array}
\right.
$$
For every $y \in \mathcal K$, we have $L_t^y(\phi_x) = L_t^y(\phi^{\mathcal K_\delta}_x)$, and by It\^o-Tanaka's formula, it follows that
$$L_t^y(\phi_x) = L_t^y(\phi^{\mathcal K_\delta}_x) = \big|\phi^{\mathcal K_\delta}_x(t)-y\big|-|\phi^{\mathcal K_\delta}_x(0)-y|-\int_0^t \mathrm{sgn}\big(\phi^{\mathcal K_\delta}_x(s)-y\big)d\phi^{\mathcal K_\delta}_x(s).$$
Assume first that $x > \sup \mathcal K_\delta$. 
Observing that $L_t^y(\phi_x)=0$ on  $\{\tau_x \geq t\}$, and that $d\phi^{\mathcal K_\delta}_x(s)$ vanishes on $[0,\tau_x)$ on $\{\tau_x < t\}$,
we readily have
\begin{align}
L_t^y(\phi_x) & =  \big|\phi^{\mathcal K_\delta}_x(t)-y\big|-|\sup \mathcal K_\delta-y|-\int_{\tau_x \wedge t}^t \mathrm{sgn}\big(\phi^{\mathcal K_\delta}_x(s)-y\big)d\phi^{\mathcal K_\delta}_x(s) \nonumber\\
& = \big|\phi_{\sup \mathcal K_\delta}\big((t-\tau_x)_+\big)-y\big|-|\sup \mathcal K_\delta-y|-\int_{0}^{(t-\tau_x)_+} \mathrm{sgn}\big(\phi_{\sup \mathcal K_\delta}(s)-y\big)d\phi_{\sup \mathcal K_\delta}(s). \label{eq:tanaka pseudostopping}
\end{align}
We plan to bound each term separately.\\

\noindent {\it Step 2).} By It\^o's formula,
$(\phi_{\sup \mathcal K_\delta}(t)-y)^2 =  (\sup \mathcal K_\delta-y)^2+I+II,$
with 
\begin{align*}
I & =  \int_0^{t} \big(2(\phi_{\sup \mathcal K_\delta}(s)-y)r(\phi_{\sup \mathcal K_\delta}(s))+\sigma(\phi_{\sup \mathcal K_\delta}(s))^2\big)ds, \\
II & = 2\int_0^{t} (\phi_{\sup \mathcal K_\delta}(s)-y)\sigma(\phi_{\sup \mathcal K_\delta}(s))dW_s.
\end{align*}
First,
$$I  \leq 2\sigma_1^{-2}\int_{\mathbb{R}}(z-y)r(z)L_{t}^z(\phi_{\sup \mathcal K_\delta})dz+t\sigma_2^2 $$
by the occupation times formula and Assumption \ref{assu:unique}. 
Introduce $|y|_{r_2} = |y|\vee r_2$, where $r_2$ is defined in Assumption \ref{assu:unique}.
Since $z-y>0$ and $r(z)<0$ for $z > |y|_{r_2}$, we have $\int_{|y|_{r_2}}^{\infty} (z-y)r(z)L_t^z(\phi_x)dz < 0$. Similarly $\int_{-\infty}^{-|y|_{r_2}} (z-y)r(z)L_t^z(\phi_x)dz < 0$. It follows that
\begin{align*}
\int_{\mathbb{R}}(z-y)r(z)L_t^z(\phi_x)dz & \leq \int_{-|y|_{r_2}}^{|y|_{r_2}}(z-y)r(z)L_t^z(\phi_x)dz \\
& \leq r_1\int_{-|y|_{r_2}}^{|y|_{r_2}} |z-y|(1+|z|)L_t^z(\phi_x)dz \\
& \leq r_1(|y|_{r_2}-y)(1+|y|_{r_2})\int_{\mathbb{R}}L_t^z(\phi_x)dz \\
& \leq r_1(|y|_{r_2}-y)(1+|y|_{r_2})t,
\end{align*}
therefore
$$I \leq 2\sigma_1^{-2}r_1(|y|_{r_2}-y)(1+|y|_{r_2})t+\sigma_2^2t = t\alpha(y)$$
say.  Since $\E[II]=0$, we derive by Cauchy-Schwarz's inequality
\begin{equation} \label{controle Ito}
\E\big[\big|\phi_{\sup \mathcal K_\delta}(t)-y\big|\big]  \leq \sqrt{(\sup \mathcal K_\delta-y)^2+t \alpha(y)}.
\end{equation}

\noindent {\it Step 3).} We are ready to control each term of \eqref{eq:tanaka pseudostopping}. We have
\begin{align*}
& \E\big[\big|\phi_{\sup \mathcal K_\delta}\big((t-\tau_x)_+\big)-y\big|\big] \\
 \leq\, & |\sup \mathcal K_\delta-y| + \E\big[\int_0^{(t-\tau_x)_+}\big|r\big(\phi_{\sup \mathcal K_\delta}(s)\big)\big|ds\big]+ \E\big[\big|\int_0^{(t-\tau_x)_+}\sigma\big(\phi_{\sup \mathcal K_\delta}(s)\big)dW_s\big|\big]\\
  \leq\, & |\sup \mathcal K_\delta-y| + r_1\E\big[\int_0^t(1+|\phi_{\sup \mathcal K_\delta}(s)|)ds\big]+ \E\big[\sup_{u \leq t}\big(\int_0^{u}\sigma\big(\phi_{\sup \mathcal K_\delta}(s)\big)dW_s\big)^2\big]^{1/2}\\
   \leq\, & |\sup \mathcal K_\delta-y| + r_1t + r_1\E\big[\int_0^t|\phi_{\sup \mathcal K_\delta}(s)|ds\big]+ \sqrt{2}\sigma_2 t \\
   \leq\, & |\sup \mathcal K_\delta-y| + r_1t + r_1\int_0^t\sqrt{(\sup \mathcal K_\delta)^2+s\alpha(0)}ds+ \sqrt{2}\sigma_2 t \\
   \lesssim\, & 1 + t^{3/2}, 
\end{align*}
where we successively applied Assumption \ref{assu:unique}, Doob's inequality and \eqref{controle Ito}.
In the same way
\begin{align*}
& \big|-\int_{0}^{(t-\tau_x)_+} \mathrm{sgn}\big(\phi_{\sup \mathcal K_\delta}(s)-y\big)d\phi_{\sup \mathcal K_\delta}(s)\big| \\ 
\leq & r_1\int_{0}^{t} (1+|\phi_{\sup \mathcal K_\delta}(s)|)ds+\sup_{u \leq t}\big|\int_0^{u}\mathrm{sgn}(\phi_{\sup \mathcal K_\delta}(s)-y)\sigma\big(\phi_{\sup \mathcal K_\delta}(s)\big)dW_s\big|.
\end{align*}
Taking expectation and using the foregoing arguments, this last quantity is also of order $1+t^{3/2}$ and Lemma \ref{lem:borne_localtime} is proved for $x > \sup \mathcal K_\delta$.\\

\noindent {\it Step 4).} If $x < \inf \mathcal K_\delta$, we apply the same arguments, replacing $|\sup \mathcal K_\delta |$ by $|\inf \mathcal K_\delta|$ with obvious changes. Likewise if $x \in \mathcal K_\delta$ we may replace $|\sup \mathcal K_\delta |$ by 
$\max\{|\sup \mathcal K_\delta |, |\inf \mathcal K_\delta |\}$.

\subsection{Proof of Proposition \ref{thm:convl2} in the case $\varrho = 0$}\label{appendix:varrho=0}
Without loss of generality, we assume that there is only one individual in each generation and thus $|\mathbb U_n| = n$. For the sake of readability, we denote by $i$ this unique individual in $\mathbb{G}_i\bigcap \mathbb{U}_n$. We follow the same steps as in the case $\varrho>0$, slightly adapting the proof. In particular, the triangle inequality used in \noindent {\it Step 1).} in the case $\varrho>0$ is not accurate enough when $\varrho = 0$.\\ 

\noindent {\it Step 1').} 
\begin{align*}
\E_{\mu}[\mathcal M_{\mathbb U_n}(\psi)^2] & =|\mathbb{U}_n^\star|^{-2}\mathbb{E}_{\mu}\big[\big(\sum_{m=1}^n\sum_{u\in\mathbb{G}_m \cap \mathbb U_n}\psi(X_{u^-},X_u)\big)^2\big] \\
& = |\mathbb{U}_n^\star|^{-2}\mathbb{E}_{\mu}\big[\big(\sum_{m=1}^n\psi(X_{m-1},X_m)\big)^2\big]\\
& = |\mathbb{U}_n^\star|^{-2}(I+ II),
\end{align*} 
where
\begin{align*}
I  & = \sum_{m=1}^n\mathbb{E}_{\mu}\big[\psi(X_{m-1},X_{m})^2\big],\\
II  & = 2\sum_{i=1}^n\sum_{j = i+1}^n\mathbb{E}_{\mu}\big[\psi(X_{i-1},X_i)\psi(X_{j-1},X_j)\big].
\end{align*}
\noindent {\it Step 2').} The control of $I$ is as before straightforward: using Lemma \ref{lem: estimates action transition} we obtain as before $I\lesssim n |\psi^2|_\mu$.\\
\noindent {\it Step 3').} By the Markov property
\begin{align*}
II & = 2 \sum_{i=1}^n\sum_{j = i+1}^n\mathbb{E}_{\mu}\big[\psi(X_{i-1},X_i)\mathcal Q^{j-i}(X_i)\big].
\end{align*}
We further decompose $II =2( III+ IV)$ having
\begin{align*}
& III = \sum_{i=1}^n\mathbb{E}_{\mu}\big[\psi(X_{i-1},X_i)\mathcal Q(X_i)\big],\\
& IV =\sum_{i=1}^{n}\sum_{j = i+2}^n\mathbb{E}_{\mu}\big[\psi(X_{i-1},X_i)\mathcal Q^{j-i}(X_i)\big].
\end{align*} 
Using Lemma \ref{lem: estimates action transition}, we get $III\lesssim n |\psi_\star|_1|\psi|_{\mu}.$ Moreover,
\begin{align*}
IV \lesssim n |\psi|_{\wedge 1}\sum_{i = 1}^n \mu(Q^i\psi)\lesssim n |\psi|_{\wedge 1}|\psi_\star|_1(1+\mu(V))\rho^n, 
\end{align*}
which yields the result.

\subsection{Integration by parts formula}\label{appendix:IPP}
We prove in this section the following formula for the local time: for all $a>0$ and $x\in\mathcal{X}$,
\begin{align*}
\E\Big[\int_0^\infty e^{-as}dL_s^x\Big]= a\int_0^\infty e^{-as}\E\Big[L_s^x\Big]ds.
\end{align*}
First, for $t\geq 0$, the integration by part formula (see \cite[Proposition 0.4.5]{RY}) yields
\begin{align*}
e^{-at}L_t^x = \int_0^t e^{-as}dL_s^x-a\int_0^t e^{-as}L_s^x ds.
\end{align*}
Moreover, we have
\begin{align*}
\liminf_{t\rightarrow +\infty} \E \Big[e^{-at}L_t^x\Big] = 0 \geq \E\Big[\liminf_{t\rightarrow \infty}e^{-at}L_t^x\Big]\geq 0,
\end{align*}
using Lemma \ref{lem:borne_localtime} and Fatou's lemma. Finally
\begin{align*}
0 = \E\Big[\int_0^\infty e^{-as}dL_s^x\Big]- a\E\Big[\int_0^\infty e^{-as}L_s^xds\Big],
\end{align*}
and the result follows by Fubini's theorem.

\subsection{Proof of Proposition \ref{prop:constant drift}} \label{appendix:preuve prop constant drift}
Remember that
\begin{align*}
\Psi(\vartheta)=\nu_{\vartheta}\Big(\mathcal Q_{\vartheta}\Big(\frac{\partial_{\vartheta}q_{\vartheta}}{q_{\vartheta}}\Big)^2\Big)=\int_{\mathcal{X}}\nu_{\vartheta}(dx)\int_{\mathcal{X}}\frac{\big(\partial_{\vartheta}q_{\vartheta}(x,y)\big)^2}{q_{\vartheta}(x,y)}dy. 
\end{align*}
If $\mathcal A\subset \mathcal X$ is a Borel set with $\text{Leb}(\mathcal A)>0$, we have
\begin{align*}
\nu_{\vartheta}(\mathcal A)=\int_{\mathcal	{X}\times \mathcal X}{\bf 1}_{\mathcal A}(y)q_{\vartheta}(x,y)\nu_{\vartheta}(dx)dy \geq \inf_{x,y}q_{\vartheta}(x,y)\, \text{Leb}(\mathcal A)>0 
\end{align*}
since $\inf_{x,y \in \mathcal X}q_{\vartheta}(x,y) >0$ by Lemma \ref{lemma:bound_p}.
By continuity of $y\mapsto \partial_{\vartheta}q_{\vartheta}(x,y)$ on $[0,L]$, it suffices then to show the existence $x,y\in\mathcal{X}$ such that $\partial_{\vartheta}q_{\vartheta}(x,y)>0$. For $x,y\in \mathcal X$, we have
\begin{align*}
\partial_{\vartheta}q_{\vartheta}(x,y) & =\int_{\varepsilon\vee yL^{-1}}^{1-\varepsilon}\frac{\ktilde(z)}{z}\sigma^{-2}\mathbb{E}\Big[\int_0^\infty (1-\vartheta t)e^{-\vartheta t}dL_t^{y/z}(\phi_x)\Big]dz \\
&=\int_{y(1-\varepsilon)^{-1}}^{y\varepsilon^{-1}}\ktilde(y/u)\sigma^{-2}\mathbb{E}\Big[\int_0^\infty (1-\vartheta t)e^{-\vartheta t}dL_t^{u}\left(\phi_x\right)\Big]\frac{du}{u} \\
&=\mathbb{E}\Big[\int_0^{\infty}\ktilde\big(y/\phi_x(t)\big) (1-\vartheta t)e^{-\vartheta t}\mathbf{1}_{\{y(1-\varepsilon)^{-1}\leq \phi_x(t)\leq y\varepsilon^{-1}\}}\frac{dt}{\phi_x(t)}\Big] \\
&= \frac{1}{1-2\varepsilon}\int_0^{\infty}(1-\vartheta t)e^{-\vartheta t}\mathbb{E}\Big[\mathbf{1}_{\{y(1-\varepsilon)^{-1}\leq \phi_x(t)\leq y\varepsilon^{-1}\}}\frac{1}{\phi_x(t)}\Big]dt
\end{align*}
by the change of variable $u = yz^{-1}$, the occupation times formula, and the specific form of $\kappa$.
For $t\geq 0$, define
\begin{align*}
A_t(x,y)=\mathbb{E}\Big[\mathbf{1}_{\{y(1-\varepsilon)^{-1}\leq \phi_x(t)\leq y\varepsilon^{-1}\}}\frac{1}{\phi_x(t)}\Big]=\int_{y(1-\varepsilon)^{-1}}^{y\varepsilon^{-1}}\rho_t(x,z)\frac{dz}{z},
\end{align*}
for which a closed-form formula is known, see for instance \cite{linetsky2005transition}, Section 4.1, given by
\begin{align*}
\rho_t(x,z)=&\frac{2r_1e^{2r_1x}}{e^{2r_1L}-1}+\frac{2}{L}e^{r_1(z-x)}\sum_{n=1}^\infty \frac{e^{-a(n)t/2}}{a(n)}g(n,x)g(n,z),
\end{align*}
with
\begin{align*}
g(n,x)=\frac{\pi n}{L}\cos\left(x\frac{\pi n}{L}\right)+r_1\sin\left(x\frac{\pi n}{L}\right),\;\;\text{and}\;\;a(n) = r_1^2+\pi^2n^2/L^2. 
\end{align*}
It follows that
\begin{align*}
A_t(x,y)=\frac{2r_1e^{2r_1x}}{e^{2r_1L}-1}\log\left(\frac{1-\varepsilon}{\varepsilon}\right)+\frac{2}{L}e^{-r_1x}\sum_{n=1}^\infty \frac{e^{-a(n)t/2}}{a(n)}g(n,x)\mathcal I(n,y)
\end{align*}
with
$\mathcal I(n,y)=\int_{y(1-\varepsilon)^{-1}}^{y\varepsilon^{-1}}e^{r_1z}g(n,z)\frac{dz}{z}$,
and therefore
\begin{align*}
\partial_{\vartheta}q_{\vartheta}(x,y)& = \frac{1}{1-2\varepsilon}\int_0^{\infty}(1-\vartheta t)e^{-\vartheta t}A_t(x,y)dt
\\
& = \frac{1}{1-2\varepsilon}\frac{2}{L}e^{-r_1x}\sum_{n=1}^\infty \Big(\int_0^{\infty}(1-\vartheta t)e^{-\vartheta t}\frac{e^{-a(n)t/2}}{a(n)}dt\Big)g(n,x)\mathcal I(n,y) \\
& = \frac{1}{(1-2\varepsilon)L}\int_{y(1-\varepsilon)^{-1}}^{y\varepsilon^{-1}}e^{-r_1(x-z)}\sum_{n=1}^\infty\frac{g(n,x)g(n,z)}{(\vartheta+a(n)/2)^2}\frac{dz}{z}.
\end{align*}
Let $x\in [0,L]$ be such that $g(n,x)\neq 0$ for every $n\geq 1$. 
Since $x\mapsto g(n,x)$ is continuous  on $[0,L]$, there exists $0<\varepsilon_n<\frac{1}{2}$ such that $g(n,x)g(n,z)>0$ for all $z\in \mathcal J(\varepsilon_n,x)=[2\varepsilon_n x,2(1-\varepsilon_n)x]$. Let $N>0$ be such that for all $z\in \mathcal J(\varepsilon_1,x)$:
\begin{align*}
\left|R_N(x,z)\right|=\Big|\sum_{n=N+1}^\infty\frac{g(n,x)g(n,z)}{(\vartheta+a(n)/2)^2}\Big|< \frac{g(1,x)g(1,z)}{(\vartheta+a(1)/2)^2},
\end{align*}
which exists because by normal convergence of the above series. Then, for every  $z\in \mathcal J(\max\{\varepsilon_n,1\leq n\leq N\},x)$
we have
\begin{align*}
\sum_{n=1}^\infty\frac{g(n,x)g(n,z)}{(\vartheta+a(n)/2)^2}=\sum_{n=1}^N\frac{g(n,x)g(n,z)}{(\vartheta+a(n)/2)^2}+R_N(x,z)>\frac{g(1,x)g(1,z)}{(\vartheta+a(n)/2)^2}-\left|R_N(x,z)\right|>0.
\end{align*}
Finally, for $\varepsilon>\max\{\varepsilon_n,1\leq n\leq N\}$, picking $y=2\varepsilon(1-\varepsilon)x$ yields $[y(1-\varepsilon)^{-1},y\varepsilon^{-1}]=\mathcal J(\varepsilon,x)\subset \mathcal J(\max\{\varepsilon_n,1\leq n\leq N\},x)$ so that $\partial_{\vartheta}q_{\vartheta}(x,y)>0$.

\subsection*{Acknowledgements} We are grateful to V. Bansaye for helpful discussion and comments, as well the insightful input of two anonymous referees. A.M. acknowledges partial support by the Chaire Mod\'elisation Math\'ematique et Biodiversit\'e of Veolia Environnement - \'Ecole Polytechnique - Museum National Histoire Naturelle - F.X. and by the French national research agency (ANR) via project MEMIP (ANR-16-CE33-0018).

\bibliographystyle{plain}
\bibliography{bibli_HM2017}

\end{document}